\documentclass[10pt, a4paper, reqno, oneside]{amsart}
\usepackage[utf8]{inputenc}
\usepackage{amsmath, calligra, mathrsfs}
\usepackage{enumerate, color}
\usepackage[hidelinks]{hyperref}
\usepackage{bbm}
\usepackage[all]{xy}
\usepackage{amsfonts} 
\usepackage{amssymb} 
\usepackage{graphics} 
\usepackage{amscd} 
\usepackage{enumitem} 
\usepackage{todonotes} 
\usepackage{comment}
\usepackage{tikz-cd}

\newcommand{\bA}{\mathbb{A}}
\newcommand{\A}{\mathbb{A}}

\newcommand{\C}{\mathbb{C}}

\newcommand{\F}{\mathbb{F}}

\newcommand{\bL}{\mathbb{L}}

\newcommand{\N}{\mathbb{N}}

\newcommand{\Q}{\mathbb{Q}}
\newcommand{\Qp}{\mathbb{Q}_p}

\newcommand{\R}{\mathbb{R}}

\newcommand{\bT}{\mathbb{T}}

\newcommand{\Z}{\mathbb{Z}}
\newcommand{\Zp}{\mathbb{Z}_p}

\DeclareMathOperator{\GL}{GL}

\newcommand{\by}{\mathbf{y}}

\newcommand{\cA}{\mathcal{A}}

\newcommand{\cE}{\mathcal{E}}

\newcommand{\cL}{\mathcal{L}}

\newcommand{\cN}{\mathcal{N}}

\newcommand{\cO}{\mathcal{O}}

\newcommand{\cS}{\mathcal{S}}
\newcommand{\cT}{\mathcal{T}}

\newcommand{\crA}{\mathscr{A}}

\newcommand{\crD}{\mathscr{D}}

\newcommand{\crW}{\mathscr{W}}

\newcommand{\sA}{\mathscr{A}}

\newcommand{\sC}{\mathscr{C}}
\newcommand{\sD}{\mathscr{D}}

\newcommand{\sI}{\mathscr{I}}

\newcommand{\sL}{\mathscr{L}}

\newcommand{\sV}{\mathscr{V}}
\newcommand{\sW}{\mathscr{W}}
\newcommand{\sX}{\mathscr{X}}
\newcommand{\sY}{\mathscr{Y}}

\newcommand{\fa}{\mathfrak{a}}

\newcommand{\fg}{\mathfrak{g}}

\newcommand{\m}{\mathfrak{m}}
\newcommand{\fm}{\mathfrak{m}}
\newcommand{\fn}{\mathfrak{n}}

\newcommand{\fp}{\mathfrak{p}}

\newcommand{\fP}{\mathfrak{P}}

\renewcommand{\hbar}{\overline{h}}

\newcommand{\End}{\mathrm{End}}

\newcommand{\h}{\mathrm{H}}

\newcommand{\hc}[1]{\mathrm{H}^{#1}_{\mathrm{c}}}

\renewcommand{\leq}{\leqslant}
\renewcommand{\geq}{\geqslant}

\newcommand*{\longhookrightarrow}{\ensuremath{\lhook\joinrel\relbar\joinrel\rightarrow}}

\newcommand\isorightarrow{\xrightarrow{
		\,\smash{\raisebox{-0.65ex}{\ensuremath{\scriptstyle\sim}}}\,}}

\newcommand*{\defeq}{\mathrel{\vcenter{\baselineskip0.5ex \lineskiplimit0pt
			\hbox{\scriptsize.}\hbox{\scriptsize.}}}%
	=}

\newcommand{\smallmatrd}[4]{\left(\begin{smallmatrix}#1 & #2\\#3 & #4\end{smallmatrix}\right)}

\newcommand{\newmod}[1]{\hspace{2pt}(\mathrm{mod}\hspace{2pt}#1)}

\newcommand{\beqn}{\begin{eqnarray*}}
	\newcommand{\eeqn}{\end{eqnarray*}}
\newcommand{\beqa}{\begin{eqnarray}}
	\newcommand{\eeqa}{\end{eqnarray}}

\numberwithin{equation}{section}

\newcommand{\YK}{K}
\newcommand{\VAN}[3]{#2}

\usepackage[OT2,T1]{fontenc}

\usepackage{environ}
\NewEnviron{myquote}{\vspace{1ex}\par
	\hfill\llap{($\dagger$)}\hfill\parbox{\textwidth-2cm}%
	{\emph{\BODY}}%
	\vspace{1ex}\par}

\DeclareSymbolFont{cyrletters}{OT2}{wncyr}{m}{n}
\DeclareMathSymbol{\Sha}{\mathalpha}{cyrletters}{"58}

\makeatletter 
\newcommand\figcaption{\def\@captype{figure}\caption} 
\newcommand\tabcaption{\def\@captype{table}\caption} 
\makeatother

\usepackage[left=1in, right=1in, top=1in, bottom=1in]{geometry}

\newtheorem{thm}{Theorem}[section]
\newtheorem{theorem}[thm]{Theorem}

\newtheorem{hyp}[thm]{Hypothesis}

\newtheorem{proposition}[thm]{Proposition}
\newtheorem{lemma}[thm]{Lemma}
\newtheorem{corollary}[thm]{Corollary}
\newtheorem*{corollary*}{Corollary}
\newtheorem{conjecture}[thm]{Conjecture}
\newtheorem{assumption}[thm]{Assumption}
\newtheorem*{assumption*}{Assumption}

\newcounter{alphalabels}

\newtheorem{theoremx}[alphalabels]{Theorem}

\newtheorem{corollaryx}[alphalabels]{Corollary}

\theoremstyle{definition}

\newtheorem{set-up}[thm]{Set-up}

\theoremstyle{definition}

\newtheorem{definition}[thm]{Definition}

\newtheorem*{notation*}{Notation}
\newtheorem{convention}[thm]{Convention}
\newtheorem*{convention*}{Convention}
\newtheorem{facts}[thm]{Facts}

\newtheorem{test-data}[thm]{Test data}

\theoremstyle{remark}
\newtheorem{remark}[thm]{Remark}
\newtheorem{example}[thm]{Example}
\newtheorem*{remark*}{Remark}

\renewcommand{\tilde}{\widetilde}

\newcommand{\Iw}{\opn{Iw}}
\newcommand{\cInd}{\mathrm{c}\text{-}\mathrm{Ind}}

\DeclareFontFamily{U}{wncy}{}
\DeclareFontShape{U}{wncy}{m}{n}{<->wncyr10}{}
\DeclareSymbolFont{mcy}{U}{wncy}{m}{n}
\DeclareMathSymbol{\sha}{\mathord}{mcy}{"58}

\newcommand{\ide}[1]{\mathfrak{#1}}
\newcommand{\mbb}[1]{\mathbb{#1}}

\newcommand{\opn}[1]{\operatorname{#1}}

\newcommand{\mbf}[1]{\mathbf{#1}}

\DeclareMathOperator{\bfA}{\mathbf{A}}

\DeclareMathOperator{\bfD}{\mathbf{D}}

\newcommand{\Addresses}{{
  \bigskip
  \footnotesize

  (Barrera Salazar) \textsc{Universidad de Santiago de Chile, Avenida Libertador Bernardo O'Higgins no. 3363, Estaci\'{o}n Central, Santiago, Chile}\par\nopagebreak
  \textit{E-mail address}: \texttt{daniel.barrera.s@usach.cl}

  \medskip

  (Graham) \textsc{Mathematical Institute, University of Oxford, Woodstock Road, Oxford OX2 6GG, United Kingdom}\par\nopagebreak
  \textit{E-mail address}: \texttt{andrew.graham@maths.ox.ac.uk}

  \medskip

  (Williams) \textsc{University of Nottingham, University Park, Nottingham, NG7 2RD, United Kingdom}\par\nopagebreak
  \textit{E-mail address}: \texttt{chris.williams1@nottingham.ac.uk}

}}

\title{The non-abelian Leopoldt conjecture and equalities of $\cL$-invariants}
\author{Daniel Barrera Salazar, Andrew Graham, and Chris Williams}
\date{}

\begin{document}
\begin{abstract} 
Let $G$ be a reductive group quasi-split at $p$.  Using arguments of Hansen--Thorne, we show that under the non-abelian Leopoldt conjecture (NALC), Hansen's $p$-adic overconvergent cohomology eigenvariety for $G$ is \'etale over its image in weight space at any non-critical classical tempered cuspidal point of `cohomological multiplicity one'. This applies to all non-critical classical cuspidal points if $G = \mathrm{Res}_{F/\Q}\GL_n$.
        
We then let $\pi$ be a $p$-ordinary regular algebraic cuspidal automorphic representation of $\mathrm{GL}_n(\mathbb{A}_{\mathbb{Q}})$ such that $\pi_p$ is Steinberg. Combining the above \'etaleness result for the classical point attached to $\pi$, and a local-global compatibility result from our earlier work, we deduce -- under a tangent vector hypothesis that is true for at least half the simple roots -- the equality of Fontaine--Mazur and automorphic $\mathcal{L}$-invariants for $\pi$. Where this assumption is satisfied, we deduce the NALC implies a conjecture of Gehrmann: that automorphic $\mathcal{L}$-invariants are independent of cohomological degree. Our approach is inspired by (and generalises) previous work of Gehrmann--Rosso.

When $\pi = \opn{Sym}^{n-1} \pi_f$ is the symmetric power lift of a modular form, we verify all assumptions other than the NALC, and deduce a functoriality result for the automorphic $\cL$-invariants.

\bigskip

\noindent\textsc{R\'esum\'e.}
Soit $G$ un groupe réductif quasi-déployé en $p$. En utilisant des arguments de Hansen--Thorne, nous montrons que, sous la conjecture de Leopoldt non abélienne (NALC), la variété de Hecke de $G$ construite par Hansen à l'aide de la cohomologie surconvergente est étale sur son image dans l'espace des poids en tout point classique cuspidal tempéré non critique de `multiplicité cohomologique un'. Cela s'applique à tous les points classiques cuspidaux non critiques si $G = \mathrm{Res}_{F/\mathbb{Q}}\mathrm{GL}_n$.

Soit alors $\pi$ une représentation automorphe cuspidale algébrique régulière $p$-ordinaire de $\mathrm{GL}_n(\mathbb{A}_{\mathbb{Q}})$ telle que $\pi_{p}$ est Steinberg. En combinant le résultat ci-dessus pour le point classique associé à $\pi$ et un résultat de compatibilité locale-globale issu de nos travaux antérieurs, nous déduisons -- sous une hypothèse sur des vecteurs tangents qui est vérifiée pour au moins la moitié des racines simples -- l'égalité des $\mathcal{L}$-invariants de Fontaine--Mazur et automorphe associés à $\pi$. Lorsque cette hypothèse est satisfaite, nous en déduisons que la NALC implique une conjecture de Gehrmann : à savoir que les $\mathcal{L}$-invariants automorphes sont indépendants du degré cohomologique. Notre approche s'inspire d'un travail antérieur de Gehrmann-Rosso (et le généralise).

Lorsque $\pi = \mathrm{Sym}^{n-1} \pi_f$ est la puissance symétrique d'une forme modulaire, nous vérifions toutes les hypothèses à l'exception de la NALC, et nous en déduisons un résultat de fonctorialité pour les $\mathcal{L}$-invariants automorphes.
\end{abstract}

\maketitle

\footnotesize
\setcounter{tocdepth}{1}
    \tableofcontents
	\normalsize

\section{Introduction}

Let $\pi$ be a $p$-ordinary regular algebraic cuspidal automorphic representation (RACAR) of $\GL_n(\A)$ such that $\pi_p$ is the Steinberg representation of $\GL_n(\Qp)$. The main result of this note is an equality of ($p$-adic) automorphic and Fontaine--Mazur $\mathcal{L}$-invariants for $\pi$. We also prove results on the structure of overconvergent cohomology and the \'etaleness of eigenvarieties in a more general setting.

\subsection{Structure theorems for overconvergent cohomology and \'etaleness of eigenvarieties}

An essential tool in our comparison of $\mathcal{L}$-invariants is the $p$-adic variation of systems of Hecke eigenvalues, captured geometrically via \emph{eigenvarieties}. Eigenvarieties, and $p$-adic families, have become a staple topic in algebraic number theory and arithmetic geometry, at the heart of breakthroughs in Iwasawa theory (cases of the Birch--Swinnerton-Dyer and Bloch--Kato conjectures) and the Langlands program (through new modularity results). 

We first describe the geometry of eigenvarieties in a more general setting. As such, let $G/\Q$ be a reductive group quasi-split at $p$. When studying the geometry of the eigenvariety for $G$, one encounters a dichotomy.
\begin{itemize}
	\item[(A)] If $G(\R)$ admits discrete series representations, the cuspidal eigenvariety is expected to be flat over weight space, with Zariski-dense sets of `classical' points (arising from cuspidal automorphic representations of $G(\A)$). This case is the main focus of \cite{Urb11}. Here there are general results studying the geometry at `nice' points; for example, \cite[\S2]{BDJ17} provides systematic \'etaleness/structure results at non-critical points in `multiplicity one' settings.
	\item[(B)] By contrast, if $G(\R)$ does not admit discrete series then the cuspidal eigenvariety is never flat over weight space, and we expect existence of `non-classical' families with isolated classical points. See \cite[Intro]{classical-locus} for a detailed summary of this phenomenon and the nature of classical families. 
\end{itemize}

The first result of this note is an analogue of the \'etaleness and structure results of \cite{BDJ17} in case (B) settings, working with Hansen's eigenvarieties from overconvergent cohomology. In particular, using ideas of Hansen--Thorne \cite{HT17}, we show a freeness result for overconvergent cohomology. 

To explain this result, note the above dichotomy is captured via degrees of (the tempered part of) cuspidal cohomology. Let $K_\infty$ be any maximal compact-modulo-centre subgroup of $G(\R)$. Then the \emph{defect} and \emph{amplitude} of $G$ are respectively 
\begin{equation}\label{eq:defect}
	\ell = \ell(G) \defeq \opn{rank}(G) - \opn{rank}(K_\infty), \qquad q = q(G) = [\opn{dim}(G(\R)/K_\infty)-\ell(G)]/2.
\end{equation}
For example, $\ell(\GL_n)= \lfloor(n-1)/2\rfloor$; if $G$ admits a Shimura datum, then $\ell(G) = 0$; and if $F$ is a number field with $r$ real embeddings and $2s$ complex embeddings, then $\ell(\mathrm{Res}_{F/\Q}\GL_2) = s$.

More generally, we have:
\begin{itemize}
	\item $G(\R)$ admits discrete series if and only if $\ell(G) = 0$;
	\item tempered cuspidal cohomology for $G$ is supported exactly in degrees $\{q,q+1,...,q+\ell\}$.
\end{itemize}
In particular, case (A) corresponds to cuspidal cohomology being supported in only one degree, whilst case (B) corresponds to positive defect, and cuspidal cohomology in multiple degrees. 

Now let $K^p \subset G(\A_f^p)$ be a fixed tame level, let $\opn{Iw} \subset G(\Qp)$ be a fixed Iwahori subgroup, and let $K = K^p\opn{Iw}$. Let $\sX_{G,K} \xrightarrow{\ w \ } \sW$ be Hansen's (Iwahori-level) eigenvariety for $G$ of tame level $K^p$, where $\sW$ is the weight space. Let $K_\infty^\circ \leq K_\infty$ be the connected component of the identity; then $\sX_{G,K} = \cup_{\varepsilon} \sX_{G,K}^\varepsilon$, where $\varepsilon$ runs over the characters $K_\infty/K_\infty^\circ \to \{\pm 1\}$. Each point $x_\phi \in \sX_{G,K}^\varepsilon$ corresponds to a finite-slope Hecke eigenpacket $\phi$ for $G$ appearing in the $\varepsilon$-part of overconvergent cohomology (of weight $w(x_\phi)$ and level $K$).

The \emph{non-abelian Leopoldt conjecture} predicts that the dimension of a `sufficiently nice' cuspidal irreducible component of $\sX_{G,K}$ (or $\sX_{G,K}^\varepsilon$) is always $\dim \sW - \ell(G)$; see Conjecture \ref{conj:non-abelian leopoldt} below. This conjecture is known if $\ell(G) \leq 1$ (see \cite[Theorem 1.1.5]{Han17}). We show:

\begin{theoremx}\label{thm:main etale}
		Suppose that:
	\begin{itemize}
		\item[(a)] $\phi$ is a non-critical classical tempered cuspidal eigenpacket of weight $\lambda$ and level $K$,
		\item[(b)] $\phi$ is of cohomological multiplicity one at sign $\varepsilon$ (see Definition \ref{def:cohom mult one}),
		\item[(c)] The non-abelian Leopoldt conjecture holds for $\phi$.
	\end{itemize}
	Then there exists a rigid locally Zariski-closed subspace $\Sigma^\varepsilon \subset \sW$, containing $\lambda$ and equidimensional of dimension $\dim \sW - \ell(G)$, such that:
	\begin{itemize}
		\item[(i)] $\sX_{G,K}^\varepsilon \times_{\sW} \Sigma^\varepsilon \to \Sigma^\varepsilon$ is \'etale at $x_\phi$, and
		\item[(ii)] For $0 \leq i \leq \ell(G)$ and $h \gg 0$, the space $\hc{q+i}(K,\sD_{\Sigma^\varepsilon})^{\leq h, \varepsilon}_{\phi}$ is free of rank $\binom{\ell(G)}{i}$ over $\cO(\Sigma^\varepsilon)_\lambda$. 
	\end{itemize}
    Moreover, one can take $\Sigma^\varepsilon$ to be the image in $\sW$ of a rigid open neighbourhood of $x_\phi$ in $\sX_{G,K}^\varepsilon$.
\end{theoremx}

Note part (i) and the last sentence mean $\sX_{G,K}^\varepsilon$ is \'etale over its image in weight space at $x_\phi$. The groups $\hc{\bullet}(K,\sD_{\Sigma^\varepsilon})$ are \emph{overconvergent cohomology groups} in the style of Ash--Stevens \cite{AS08}. We describe all notation (which we adopt almost directly from \cite{Han17}) and terminology  in \S\ref{sec:set-up} and \S\ref{sec:proof}. The proof goes through Hansen's Tor spectral sequence (see Facts \ref{facts:hansen other}), and occupies  \S\ref{sec:proof}. 

\begin{remark}
We stress that our proof essentially follows Hansen--Thorne \cite[\S4.3]{HT17}. They specialise to $G = \GL_n/\Q$, and make the same assumptions; then in Theorem 4.9 \emph{op.\ cit}., they prove that: the local ring $\bT_\phi$ of $\sX_{G,K}$ at $\phi$ is a complete intersection; that, for an appropriate open affinoid $\Omega \subset \sW$, the localised top degree cohomology $\hc{q+\ell(G)}(K,\sD_\Omega)^{\leq h, \varepsilon}_\phi$ is free of rank one over $\bT_\phi$; and then that
\[
\opn{Tor}_i^{\cO(\Omega)_\lambda}(\hc{q+\ell(G)}(K,\sD_\Omega)_\phi^{\leq h, \varepsilon}, L) \cong \hc{q+\ell(G)-i}(K,\sL_\lambda)_\phi^\varepsilon,
\]
where $\sL_\lambda$ is the local system attached to the $G$-representation of highest weight $\lambda$. We claim no originality beyond explicitly spelling out further applications/consequences of their methods.
\end{remark}

Theorem \ref{thm:main etale} is a key tool in our result on $\mathcal{L}$-invariants, but it should have wider applications. For example, in the special case of Bianchi modular forms (where $G = \mathrm{Res}_{F/\Q}\GL_2$, for $F$ an imaginary quadratic field), the result generalises \cite[Thm.\ 4.5]{BW18}, which was used to construct $p$-adic $L$-functions in families. The result \emph{op.\ cit}.\ worked over an irreducible component $\Sigma'$ of $\Sigma$ and imposed the additional assumption that $\Sigma'$ was smooth at $\lambda$. In particular, Theorem \ref{thm:main etale} also upgrades \cite[Thm.\ 5.1, Cor.\ 5.3]{GehrmannRosso} by removing the smoothness assumption on the curve $\Sigma$ (denoted by $\cS$ in \cite{GehrmannRosso}). That result describes equalities of top and bottom degree $\cL$-invariants attached to Bianchi modular forms of trivial weight.

We believe constructions of families of $p$-adic $L$-functions are being explored for $G = \GL(3)$ and $G=\GL(3)\times\GL(2)$ in the forthcoming works of Dimitrakopoulou--Rockwood and Hara--Namikawa respectively. In both cases, Theorem \ref{thm:main etale} is likely to be an essential ingredient. Note in these cases that assumptions (b) and (c) always hold; (b) by strong multiplicity one, and (c) since $\ell(G) = 1$.

\subsection{The equality of $\mathcal{L}$-invariants for $\mathrm{GL}(n)$}

We now turn to our main result. In general, $\mathcal{L}$-invariants are $p$-adic invariants attached to arithmetic objects sufficiently ramified at $p$. For example, if $n=2$ and $\pi = \pi_E$ is attached to an elliptic curve $E$ with split multiplicative reduction at $p$, then one has an $\mathcal{L}$-invariant $\mathcal{L}_E = \log_p(q)/\opn{ord}_p(q)$, where $q$ is the Tate period, such that $E(\Qp) \cong \Qp^\times/q^{\Z}$. Thanks to work of many people, notably Greenberg--Stevens \cite{GS93}, Fontaine--Mazur \cite{MazurLinvariant}, and Darmon \cite{Dar01}, $\mathcal{L}_E$ is known to have many other descriptions: for example, it encodes infinitesimal Hecke eigenvalues in $p$-adic families through $\pi_E$; it describes extension data in the local Galois representation of $E$ at $p$; and it arises from $p$-arithmetic cohomology/modular symbols on the Bruhat--Tits tree. This quantity also appears in Mazur--Tate--Teitelbaum's \emph{exceptional zero conjecture}, proved in the aforementioned work of Greenberg--Stevens. This result is a key ingredient in proving cases of the $p$-part of the Birch and Swinnerton-Dyer conjecture when the elliptic curve has multiplicative reduction at $p$ \cite{SkinnerSplitMult}. 

The latter two descriptions of $\mathcal{L}_E$ have been generalised hugely. Let now $\pi$ be a $p$-ordinary regular algebraic cuspidal automorphic representation of $\GL_n(\A)$ such that $\pi_p$ is the Steinberg representation of $\GL_n(\Qp)$. Under these assumptions, $\pi$ necessarily has trivial cohomological weight. Then for each $1\leq c \leq n-1$:
\begin{itemize}
\item When $\rho_\pi|_{G_{\Qp}}$ is non-split at $c$ (see Definition \ref{def: non-crystalline at c}), there is an attached \emph{Fontaine--Mazur $\mathcal{L}$-invariant} $\bL_{c}^{\opn{FM}}(\pi)$, constructed from extension data in $\rho_\pi|_{G_{\Qp}}$ at the simple root $\alpha_c$ corresponding to $c$. 
\item There are also \emph{automorphic $\mathcal{L}$-invariants} $\bL_{c}^{r,\varepsilon}(\pi)$ for $\pi$, where $0 \leq r \leq \ell$ and $\varepsilon \colon K_\infty/K_\infty^\circ \to \{\pm 1\}$ is an \emph{admissible sign} for $\pi$ (that is, any character such that $\pi$ contributes to the $\varepsilon$-part of cohomology). These were constructed by Gehrmann \cite{AutomorphicLinvariants} using $\varepsilon$-parts of compactly-supported $p$-arithmetic cohomology (in degree $q-n+1+r$) and extensions of Steinberg representations.
\end{itemize}
We recall the definitions of both types of $\cL$-invariant in \S\ref{s:L-invariants}.

\begin{remark}
In both cases, the $\cL$-invariants described above are subspaces $\bL \subset \opn{Hom}_{\opn{cts}}(\Qp^\times,L) = \langle \log_p,\opn{ord}_p\rangle$ of codimension $\geq 1$, and $\mbb{L}_c^{\opn{FM}}(\pi)$ is always non-zero. Under further assumptions, we can extract quantities $\cL_c^{\opn{FM}}(\pi), \cL_c^{r,\varepsilon}(\pi) \in L$ more closely resembling $\cL_E$ in the elliptic curve case. More precisely:
\begin{itemize}
\item The representation $\rho_\pi|_{G_{\Qp}}$ is non-crystalline at $c$ if and only if $\opn{ord}_p \not\in \mbb{L}_c^{\opn{FM}}(\pi)$. In this case, we define $\cL_c^{\opn{FM}}(\pi) \in L$ such that $\log_p - \cL_c^{\opn{FM}}(\pi)\opn{ord}_p \in \bL_c^{\opn{FM}}(\pi)$.
\item We know $\opn{ord}_p \not\in \bL_c^{r,\varepsilon}(\pi)$ by work of Gehrmann (see Proposition \ref{p: L-invariant} below). If $\bL_c^{r,\varepsilon}(\pi) \neq 0$, then we define $\cL_c^{r,\varepsilon}(\pi) \in L$ such that $\log_p - \cL_c^{r,\varepsilon}(\pi)\opn{ord}_p \in \bL_c^{r,\varepsilon}(\pi)$.
\end{itemize}
The quantity $\cL_c^{\opn{FM}}(\pi)$ appears in the Greenberg--Benois exceptional zero conjecture, formulated in \cite[p.166]{Gre94}. In our generality and notation, this is stated in \cite[Conj. 1.1.4]{GL3-ExceptionalZeros}. 
\end{remark}

In \cite{GL3-ExceptionalZeros}, under certain assumptions we proved the equality of Fontaine--Mazur and (bottom degree) automorphic $\mathcal{L}$-invariants when $n=3$, and thus deduced the first cases of the Greenberg--Benois exceptional zero conjecture \cite{Gre94,Benois11} without any self-duality condition. The main result of this note is a generalisation of this equality to arbitrary dimensions and degrees.

\begin{theoremx}\label{thm:intro equality}
Let $n\geq 2$, and let $\pi$ be a $p$-ordinary regular algebraic cuspidal automorphic representation of $\GL_n(\A)$ such that $\pi_p$ is the Steinberg representation of $\GL_n(\Qp)$. Let $1 \leq c \leq n-1$. Suppose that:
\begin{itemize}
\item[(a)] The non-abelian Leopoldt conjecture holds for $p$-adic families through $\pi$;
\item[(b)] $\pi$ admits non-$c$-parabolic deformations; 
\item[(c)] $p>2n$, and the residual Galois representation $\overline{\rho}_\pi : G_{\Q} \to \GL_n(\overline{\F}_p)$ is irreducible and decomposed generic;
\item[(d)] The local Galois representation $\rho_\pi|_{G_{\Qp}}$ is non-split at $c$.
\end{itemize}
Then for all $0 \leq r \leq \ell$ and all admissible signs $\varepsilon$, we have an equality
\[
    \bL_{c}^{\mathrm{FM}}(\pi) = \bL_{c}^{r,\varepsilon}(\pi)
\]
of Fontaine--Mazur and automorphic $\cL$-invariants at the simple root $\alpha_c$. Moreover, $\rho_{\pi}|_{G_{\mbb{Q}_p}}$ is non-crystalline at $c$ and we have an equality of scalar $\mathcal{L}$-invariants $\mathcal{L}_{c}^{\mathrm{FM}}(\pi) = \mathcal{L}_{c}^{r,\varepsilon}(\pi)$.
\end{theoremx}

We note condition (d) is required for $\bL_{c}^{\opn{FM}}(\pi)$ to have codimension $\geq 1$ in $\opn{Hom}_{\opn{cts}}(\mbb{Q}_p^{\times}, L)$.

\begin{remark}
    Under the assumptions (a)--(c) of Theorem \ref{thm:intro equality}, we have shown that if $*_c$ is non-split, then $*_c$ is non-crystalline. To see this, note that $\bL_c^{\opn{FM}}(\pi) = \bL_{c}^{r,\varepsilon}(\pi) \not\ni\opn{ord}_p$; so $\opn{ord}_p \not\in \bL_c^{\opn{FM}}(\pi)$. If we know this for all $1 \leq c \leq n-1$, then we deduce that the monodromy on $\mbf{D}_{\opn{st}}(\rho_{\pi}|_{G_{\mbb{Q}_p}})$ has maximal rank and hence we obtain the full local-global compatibility statement at $\ell = p$ for $\rho_{\pi}$. Currently, the most general result concerning such local-global compatibility in the non-polarised setting is \cite[Theorem 1.1]{HevesiOrdinaryParts}, which unfortunately doesn't say anything about the monodromy when $\pi_p$ is Steinberg (the inequality `$\preceq$' in \cite[(1.1)]{HevesiOrdinaryParts} is in the opposite direction).
\end{remark}

\begin{remark}\label{rem:esd}
The most subtle assumption in the theorem is (b), which we now explain. The other assumptions ensure $\pi$ has trivial cohomological weight. Attached to $\pi$ is a unique eigenpacket $\phi$ which is non-critical classical of cohomological multiplicity one. As we also assume non-abelian Leopoldt, Theorem \ref{thm:main etale} thus holds for $\phi$, yielding a Zariski-closed rigid weight subspace $\Sigma = \Sigma^\varepsilon \subset \sW$, containing the trivial character $\mathbbm{1}$. The $\GL_n$ weight space $\sW$ is $n$-dimensional, with a general $L$-point having the form $(\lambda_1,\dots,\lambda_n)$, where $\lambda_i : \Zp^\times \to L^\times$ is a continuous character.  For $1 \leq c \leq n-1$, we say $\pi$ \emph{admits non-$c$-parabolic deformations} if there exists a tangent vector $v = (v_1,\dots,v_n) \in \opn{T}_{\mathbbm{1}}(\Sigma)$ such that $v_c \neq v_{c+1}$. 

For any $\pi$, we strongly expect the existence of non-$c$-parabolic deformations for all $c$. Such a result should follow from a higher-dimensional analogue of the Calegari--Mazur conjecture \cite{CM09}; see \cite[Rem.\ 8.2.4]{GL3-ExceptionalZeros} for a discussion of this for $\GL_3$. 

By Theorem \ref{thm:main etale}, we know $\dim \Sigma = n-\lfloor\tfrac{n-1}{2}\rfloor$. Since each of the $n-1$ conditions on $v$ cuts out a different codimension 1 subspace of $\sW$, we thus see that $\pi$ admits non-$c$-parabolic deformations for at least half the possible values of $c$.

If $\pi$ is essentially-self-dual, then we expect $\Sigma$ to be an open subspace of the pure weight space $\sW^0$, that is the subspace where $\lambda_i + \lambda_{n+1-i}$ is independent of $i$. Note this has dimension $n - \lfloor \tfrac{n-1}{2}\rfloor$, and that if $\Sigma \subset \sW^0$ is open, then $\pi$ admits non-$c$-parabolic deformations for all $c$. Essentially-self-dual representations of $\GL_n$ are either of symplectic or orthogonal type; see Remark \ref{rem:symplectic} for a more detailed discussion of the symplectic case.
\end{remark}

Gehrmann--Rosso \cite{GehrmannRosso} proved an analogous result in the special case where $\ell = 0$. Whilst their strategy does not extend immediately to $\ell > 0$ (see Remark \ref{rem:GR bottom}), our approach is inspired by theirs, using Kohlhaase--Schraen's Koszul complex and the cohomology of locally analytic principal series representations. Our key new input is a Tor spectral sequence for the latter, which we analyse using the former. 

Condition (b) furnishes a good choice of infinitesimal deformation for $\pi$, and we study automorphic and Galois data for a $p$-adic family through $\pi$ in this direction. We establish two key formulae:
\begin{itemize}
\item[--] An `automorphic' Benois--Colmez--Greenberg--Stevens formula, proved in Theorem \ref{p: automorphic BCGS}. Here we use assumption (a) (to ensure \'etaleness of the eigenvariety at $\pi$, using Theorem \ref{thm:main etale}) and assumption (b) (to ensure the resulting infinitesimal formula is not an empty statement). 

Establishing this automorphic formula is the heart of the paper, and in particular it is where the Koszul complex and locally analytic principal series are required.

\item[--] A `Galois' Benois--Colmez--Greenberg--Stevens formula, proved in Theorem \ref{t: galois Benois--Colmez--Greenberg--Stevens formula}. Here a crucial input is our earlier work on local--global compatibility at $\ell=p$ for $\GL_n$ torsion classes, which allows us to deduce that the Galois representation for a Hida family through $\pi$ is trianguline at $p$. This work requires assumption (c). 
\end{itemize}
Combining these formulae gives the theorem.

Much like in the $n=2,3$ cases, we anticipate applications of this result towards the Greenberg--Benois exceptional zero conjecture. See \cite[Intro]{GL3-ExceptionalZeros} for a detailed introduction to this topic. 

We also have an immediate application to a conjecture of Gehrmann:

\begin{corollaryx}
Let $n\geq 2$, and let $\pi$ be a $p$-ordinary regular algebraic cuspidal automorphic representation of $\GL_n(\A)$ such that $\pi_p$ is the Steinberg representation of $\GL_n(\Qp)$. Let $1 \leq c \leq n-1$.
\begin{itemize}
\item[(i)] Suppose conditions (a)--(d) of Theorem \ref{thm:intro equality} hold. Then Conjecture A of \cite{AutomorphicLinvariants} holds for $\pi$ at $c$; that is, the automorphic $\cL$-invariant $\bL_{c}^{r,\varepsilon}(\pi)$ is independent of the degree $r$ and sign $\varepsilon$.

\item[(ii)] Suppose only conditions (a) and (b) hold. Then $\bL_{c}^{r,\varepsilon}(\pi)$ is independent of $r$.
\end{itemize}
\end{corollaryx}

Part (i) follows directly from Theorem \ref{thm:intro equality}, since the Fontaine--Mazur $\cL$-invariant has no dependence on degree or sign. Part (ii) follows directly from the automorphic Benois--Colmez--Greenberg--Stevens formula, which requires only parts (a) and (b) (but works at a fixed admissible sign $\varepsilon$).

In particular, given the existence of non-$c$-parabolic deformations, we find the non-abelian Leopoldt conjecture implies Gehrmann's conjecture on independence of degree. Gehrmann has an analogous result proving that this follows from a conjecture of Venkatesh. We comment on this, and the relationship between the two results, in Remark \ref{rem:venkatesh}.

\begin{remark}
As highlighted above, since we assume $\pi_p$ is Steinberg, the $p$-ordinary assumption ensures $\pi$ has trivial cohomological weight. We expect analogues of these results should exist in the non-ordinary/higher weight case as well. For example, the equality of Theorem \ref{thm:intro equality} is known for elliptic modular forms of any weight $k\geq 2$. However there are two major obstacles to such a result at present.

The first main obstacle is that in general, Gehrmann's automorphic $\cL$-invariants have not yet been constructed for non-trivial weights (see \cite[Rem.\ 2.7]{AutomorphicLinvariants}). Without this one cannot even formulate the automorphic Benois--Colmez--Greenberg--Stevens formula. 

A second obstacle comes in the Galois Benois--Colmez--Greenberg--Stevens formula. A crucial ingredient for this is local-global compatibility at $\ell=p$ for torsion classes for $\GL_n/\Q$, and our proof of this result in \cite[\S9]{GL3-ExceptionalZeros} makes essential use of the ordinarity assumption. In order to obtain such a formula in the higher weight case, one would need analogous local-global compatibility results for non-ordinary torsion classes; in the setting of $\opn{GL}_n$ over CM fields, such results have been obtained in recent work of McDonald \cite{McDonald-Trianguline}.
\end{remark}

\subsection{Application to symmetric powers}

In the case of symmetric powers of modular forms, some of our key assumptions are known to hold. Let $f \in S_2^{\opn{new}}(\Gamma_1(Np), \chi_f)$, with attached (unitary) automorphic representation $\pi$ of $\GL_2(\A)$. We assume that $\pi_p$ is the Steinberg representation (equivalently, $a_p(f) = 1$ and the conductor of the nebentypus $\chi_f$ is prime to $p$). Fix $n\geq 3$, and consider the symmetric power lift $\opn{Sym}^{n-1}\pi$, which is a RACAR of $\GL_n(\A)$ constructed by Newton--Thorne \cite{NewtonThorneSymII}. This is $p$-ordinary and Steinberg-at-$p$. In \S\ref{sec:symmetric powers}, we verify condition (b) of Theorem \ref{thm:intro equality} for $\opn{Sym}^{n-1}\pi$, and thus as a consequence of our automorphic Benois--Colmez--Greenberg--Stevens formula, we deduce the following functoriality for automorphic $\cL$-invariants:

\begin{theoremx} 
Suppose the non-abelian Leopoldt conjecture holds for $\opn{Sym}^{n-1}\pi$. Then for all $1 \leq c \leq n-1$, all $0 \leq r \leq \ell(\GL_n)$, and all admissible signs $\pm$ for $\pi$ and $\varepsilon$ for $\opn{Sym}^{n-1}\pi$, we have
\[
    \bL_c^{r,\varepsilon}(\opn{Sym}^{n-1}\pi) = \bL_1^{0,\pm}(\pi) = \bL_1^{\opn{FM}}(\pi) = \bL_{c}^{\opn{FM}}(\opn{Sym}^{n-1}\pi).
\]
In particular Gehrmann's conjecture on independence of degree and sign holds for $\opn{Sym}^{n-1}\pi$.
\end{theoremx}

The second and third equalities are already known; our input is the first equality. Taken together, though, we obtain the main result of Theorem \ref{thm:intro equality} for symmetric powers assuming only condition (a).

Note this result is unconditional for symmetric squares and cubes, since the non-abelian Leopoldt conjecture is known to hold for $\GL_3$ and $\GL_4$ (see \cite[Theorem 1.1.5]{Han17}, recalling that $\ell(\opn{GL}_n) = \lfloor\tfrac{n-1}{2}\rfloor$, whence $\ell(\opn{GL}_3) = \ell(\opn{GL}_4) = 1$). For $\GL_3$ we previously obtained this result for the bottom degree $\cL$-invariant in \cite[\S7]{GL3-ExceptionalZeros}.

\subsection{Acknowledgements}
It is a pleasure to thank Henri Darmon for his impact on our respective mathematical lives. DBS and CW first got interested in $\mathcal{L}$-invariants and exceptional zeros thanks to a generous suggestion of Henri in 2016, which led directly to the paper \cite{BW17}. That project solidified our friendship and collaboration, which has been so important to us over the subsequent decade. Later, at a workshop in Oaxaca in 2018, he suggested to us the possibility of proving equalities of $\mathcal{L}$-invariants via infinitesimal weight deformations, which -- thanks to the further insights of Lennart Gehrmann and Giovanni Rosso -- is exactly the subject of this paper. More recently, all three authors have been collaborating on projects that have their root in Henri's work, notably on arithmetic cohomology and his construction of Stark--Heegner points. 

We thank Kenichi Namikawa, James Newton and the anonymous referee for their comments and corrections on earlier drafts of this article. DBS was supported by ECOS230025, Anid Fondecyt Regular grant 1241702 and a CRM-Simons  Professor position in Montreal. AG was funded by UK Research and Innovation grant MR/V021931/1. For the purpose of Open Access, the authors have applied a CC BY public copyright licence to any Author Accepted Manuscript (AAM) version arising from this submission.

\section{Overconvergent cohomology and eigenvarieties}\label{sec:set-up}
	Throughout we use generalities on rigid spaces, and their irreducible components, as found in \cite{Con99}. We use the eigenvarieties constructed by Hansen in \cite{Han17}, and rather than explain significant notation here, we largely adopt his\footnote{We do, however, mostly write $K$ where Hansen uses $K^p$, to emphasise the Iwahori-level nature of the constructions; and we define the abstract Hecke algebras over $\Zp$ rather than $\Qp$.}, and give precise references to his paper. In particular:
	\begin{itemize}
		\item $G/\Q$ is a reductive group, of the form $\opn{Res}_{F/\Q}H$, where $F$ is a number field and $H/F$ is a reductive group quasi-split at each prime $\fp|p$ of $F$ above $p$;       
        \item $K = K^p\opn{Iw} \subset G(\A_f)$ is an open compact subgroup, where $\opn{Iw} \subset G(\Qp)$ is an Iwahori subgroup (at all $\fp|p$) and $Y(K)$ is the associated locally symmetric space (see \S2.1 of \cite{Han17});
        \item $S$ denotes a finite set of places of $F$ containing all $v$ where $v|p$, or $H/F_v$ is ramified, or $K_v = H(F_v)\cap K^p$ is not hyperspecial maximal compact in $H(F_v)$;
		\item $\bT(K) = \bT^{S}(K^p)\otimes \sA_p^+$ is the (commutative) abstract Hecke algebra of level $K$ over $\Zp$, where $\bT^{S}(K^p)$ is the spherical Hecke algebra over $\Zp$ (a restricted tensor product over $v \not\in S$) and $\sA_p^+$ is the $\Zp$-algebra of controlling operators at $p$ (\S2.1). If $R$ is a $\Zp$-algebra, we write $\bT^S(K^p)_R, \bT(K)_R$ etc.\ for the base-changes to $R$;
        \item For $L/\mbb{Q}_p$ a finite extension, we say that $R$ is an $L$-affinoid algebra if it is the quotient of a Tate algebra of the form $L\langle X_1, \dots, X_d \rangle$, for some $d \geq 0$. We let $R^+ \subset R$ denote the unit ball with respect to the supremum semi-norm on $R$ (see \cite[\S 3.3.1]{Bosch});
		\item $\sW = \sW_{K}$ is the weight space of level $K$ (\S2.2), and for an affinoid $\Omega = \opn{Sp}(\cO(\Omega)) \subset \sW$, weights $\lambda \in \Omega$ are in bijection with maximal ideals $\m_\lambda \subset \cO(\Omega)$; 
		\item For any affinoid $\Omega \subset \sW$, let $\sD_\Omega$ be the local system of locally analytic distributions of weight $\Omega$ (\S2.2);
		\item $\hc{\bullet}(K,\sD_\Omega) \defeq \hc{\bullet}(Y(K), \sD_\Omega)$ denotes the total (singular) cohomology with compact support;
		\item a \emph{slope datum} $(U_t, \Omega,h)$ is a triple of a fixed controlling operator $U_t$ at $p$ (\S2.1), and an affinoid  $\Omega \subset \sW$ such that $\hc{\bullet}(K,\sD_\Omega)$ admits a slope $U_t \leq h$ decomposition, with slope $\leq h$ subspace $\hc{\bullet}(K,\sD_\Omega)^{\leq h}$;
        \item If $M$ is an $R$-module upon which $\bT(K)$ acts, we write $\bT(K)(M)$ for the image of $\bT(K)_R$ in $\opn{End}_R(M)$. For a slope datum $(U_t,\Omega,h)$, we simplify via $\bT_{\Omega,h}(K) \defeq \bT(K)(\hc{\bullet}(K,\sD_\Omega)^{\leq h})$ (\S3.2);
		\item A \emph{slope $\leq h$ eigenpacket} of weight $\lambda \in \sW$ and level $K$ is an algebra homomorphism $\phi : \bT_{\lambda, h}(K) \to \overline{\Q}_p$ (for the affinoid $\{\lambda\} \subset \sW$; Definition 3.2.2);
        \item If $M$ is a $\bT(K)_R$-module, and $\phi \colon \bT(K)_R \to R$ is a homomorphism, we say $\phi$ \emph{contributes to $M$} if the localisation $M_\phi \defeq M_{\ker(\phi)}$ is non-zero;
		\item If $\lambda$ is algebraic and dominant, and $\phi$ contributes to $\hc{\bullet}(K,\sL_{\lambda})$, we say $\phi$ is \emph{classical} (Definition 3.2.3). Here $\sL_\lambda$ is the local system attached to the $G$-representation $L_\lambda$ of highest weight $\lambda$. (For simplicity of notation only, we exclude more general arithmetic weights.)  
        
        \item We say $\phi$ is \emph{cuspidal} if it is classical and 
        \[
        \hc{\bullet}(K,\sL_{\lambda})_\phi \cong \h_{\opn{cusp}}^{\bullet}(K,\sL_{\lambda})_\phi;
        \] in this case it appears\footnote{The converse is not true; $\phi$ can appear in a cuspidal automorphic representation without itself being cuspidal. For example, consider $G = \opn{GSp}_4$, and $\phi$ attached to a Saito--Kurokawa lift of a modular form. This is cohomological (of non-regular highest weight), and appears in a cuspidal $\pi$; but it also appears in a \emph{non}-cuspidal $\pi'$ by \cite{PS83}. Such $\phi$ contributes to multiple degrees of compactly-supported cohomology, but only one degree of cuspidal cohomology. Our results do not apply to such a $\phi$; for example, Proposition \ref{prop:matsushima} already fails.} in (at least one) cohomological cuspidal automorphic representation $\pi$ of weight $\lambda$. We say such a $\phi$ is \emph{tempered} if the associated $(\fg,K_\infty)$-module $\pi_\infty$ is tempered for every possible such $\pi$.
        \item If $\phi$ is an eigenpacket, and $M$ is a module upon which $\bT(K)$ acts, we write $M_\phi$ for the localisation at $\ker(\phi)$. If $\Omega  = \mathrm{Sp}(\cO(\Omega)) \subset \sW$ is an affinoid and $\lambda \in \Omega$ corresponding to a maximal ideal $\m_\lambda \subset \cO(\Omega)$, we write $\cO(\Omega)_\lambda$ for the localisation at $\m_\lambda$.
		\item A classical slope $\leq h$ eigenpacket $\phi$ is \emph{non-critical} if the natural map
		\[
		i_\lambda : \hc{\bullet}(K,\sD_\lambda)_\phi^{\leq h} \to \hc{\bullet}(K, \sL_{\lambda})_\phi^{\leq h}
		\]
		is an isomorphism (Definition 3.2.3).
	\end{itemize}
	
	With all of this data, Hansen constructs in \cite[\S4]{Han17} an (Iwahori-level) \emph{eigenvariety} $\sX_{G,K}$ for $G$ of tame level $K^p$, which is a rigid space together with a \emph{weight map} $w : \sX_{G,K} \to \sW$. The local pieces of $\sX_{G,K}$ are of the form $\opn{Sp}(\bT_{\Omega,h}(K))$, with $(U_t,\Omega,h)$ a slope datum and $\Omega$ open in $\sW$. Points $x$ in this local piece parametrise slope $\leq h$ eigenpackets of weight $\lambda = w(x) \in \Omega$ and level $K$.

\begin{remark}\label{rem:neat}
Note we do not specify that $K$ is neat. If $K$ is not neat, then for any local system $\mathscr{M}$ on $Y(K)$, when we write $\hc{\bullet}(Y(K),\mathscr{M})$, we mean $\hc{\bullet}(Y(K'),\mathscr{M})^{K}$, where $K'\subset K$ is any neat open compact normal subgroup. This is independent of $K'$, and carries a canonical Hecke action (see \cite[\S4.1]{HT17}). This definition also dovetails with the usual relationship between singular cohomology and (cohomological) automorphic representations. All the foundational results of \cite{Han17} we need apply in this case.
\end{remark}

    For precise results, it will be important to also consider $\varepsilon$-analogues of all of the above. 

\begin{definition}
		Let $K_\infty \subset G(\R)$ be a maximal  compact-mod-centre subgroup, and $K_\infty^\circ \subset K_\infty$ the connected component of the identity. Let $\varepsilon : K_\infty/K_\infty^\circ \to \{\pm 1\}$ be a character. If $M$ is a space on which $K_\infty/K_\infty^\circ$ acts, let $M^\varepsilon$ be the submodule where it acts as $\varepsilon$. We call this the \emph{$\varepsilon$-part} of $M$.
\end{definition}
\begin{definition}\label{def:admissible signs}
If $\pi$ is a cohomological cuspidal automorphic representation of weight $\lambda$, we let
\[
    \cE_\pi \defeq \Big\{\varepsilon \colon K_\infty/K_\infty^\circ \to \{\pm1\} : \h^\bullet(\fg,K_\infty^\circ;\pi_\infty\otimes L_\lambda)^\varepsilon \neq 0\Big\}
\]
be the set of \emph{admissible signs for $\pi$}; that is, the signs $\varepsilon$ for which the $(\fg,K_\infty^\circ)$ cohomology is non-trivial. If $\phi$ is a classical cuspidal eigenpacket appearing in $\pi$, we also write $\cE_\phi \defeq \cE_\pi$.\footnote{This is possibly an abuse of notation; in general, we do not make the claim that $\cE_\phi$ depends only on $\phi$, rather than $\pi$. This abuse is harmless for our purposes: under our later assumptions where $\phi$ is a cuspidal tempered eigenpacket of cohomological multiplicity one, the set $\cE_\phi$ does only depend on $\phi$.}
\end{definition}

By repeating Hansen's constructions using only the $\varepsilon$-part of the cohomology, one obtains an eigenvariety $\sX_{G,K}^\varepsilon$. The local pieces of $\sX_{G,K}^\varepsilon$ are, analogously to above, of the form $\opn{Sp}(\bT_{\Omega,h}^\varepsilon(K))$,  where $\bT_{\Omega,h}^\varepsilon(K)$ is the image of $\bT(K)\otimes \cO(\Omega)$ in $\opn{End}_{\cO(\Omega)}(\hc{\bullet}(K,\sD_\Omega)^{\leq h, \varepsilon})$. Its points $x$ parametrise slope $\leq h$ eigenpackets that contribute to $\hc{\bullet}(K,\sD_{w(x)})^\varepsilon$. Note each $\bT_{\Omega,h}^{\varepsilon}$ is naturally a quotient of $\bT_{\Omega,h}$, leading to a canonical closed immersion $\sX_{G,K}^\varepsilon \hookrightarrow \sX_{G,K}$; and $\sX_{G,K} = \cup_\varepsilon \sX_{G,K}^{\varepsilon}$.

\medskip
    
	Recall the amplitude $q = q(G)$ and defect $\ell = \ell(G)$ from \eqref{eq:defect}. If $\phi$ is a classical cuspidal tempered eigenpacket, then for any $\varepsilon \in \cE_\phi$, by \cite{Clo90} and \cite[Thm.\ III.5.1]{BW00}, $\phi$ contributes to the $\varepsilon$-part of cuspidal cohomology in degrees $\{q, q+1,\dots, q+\ell\}$. If $\phi$ is non-critical, it thus yields a point $x_\phi^\varepsilon \in \sX_{G,K}^\varepsilon$. Since the image of this point in $\sX_{G,K}$ is independent of $\varepsilon$, we henceforth just denote this point $x_\phi$, and abusing notation consider it as a point either in $\sX_{G,K}$ or, for any $\varepsilon \in \cE_\phi$, in $\sX_{G,K}^\varepsilon$. 
    
    The following conjecture of Hida and Urban is \cite[Conj.\ 1.1.4]{Han17}:
	
	\begin{conjecture}[The non-abelian Leopoldt conjecture] \label{conj:non-abelian leopoldt} Let $\phi$ be a non-critical classical tempered cuspidal eigenpacket. Any irreducible component $\sI \subset \sX_{G,K}$ (or $\sX_{G,K}^\varepsilon$) through $x_\phi$ has dimension $\dim \sW - \ell(G).$
	\end{conjecture}
	
	\begin{remark}
    \label{rem:non-abelian leopoldt}
    In the appendix of \cite{Han17}, Newton proved that $\dim\sI \geq \dim \sW-\ell(G)$. In \cite[Thm.\ 4.5.1]{Han17}, Hansen proved the conjecture for $\ell(G) = 0, 1$, and showed that if $\ell(G) \geq 1$, then $\dim\sI < \dim\sW$.
    \end{remark}

	\section{Structure theorems and \'etaleness results}\label{sec:proof}

	In this section we prove  Theorem \ref{thm:main etale}. As highlighted in the introduction, the proof here is inspired heavily by \cite[\S4.3]{HT17}. First we must make precise the conditions in the statement.
	
	\begin{convention} \label{convention}
		We will work throughout with a fixed finite-slope tempered eigenpacket $\phi$, which will, if not specified, have weight $\lambda$ and level $K$, corresponding to a point $x_\phi \in \sX_{G,K}$. We will only work with slope data $(U_t,\Omega,h)$ such that $\phi$ has slope $\leq h$ with respect to $U_t$. We will always work over a sufficiently large finite extension $L/\Qp$. Affinoids denoted $\Omega \subset \sW$ will always be open in $\sW$ (though we later also use Zariski-closed affinoids $\Sigma,\Lambda \subset \Omega$).
	\end{convention}
	
	In practice, $L$ will be taken large enough for the following assumption to hold.
	
	\begin{definition}\label{def:cohom mult one}
		We say a classical cuspidal eigenpacket $\phi$ of weight $\lambda$ and level $K$ is \emph{of cohomological multiplicity one (at sign $\varepsilon$)} if 
		\begin{equation}\label{eq:mult one def}
		\dim_L \hc{q+\ell}(K,\sL_{\lambda}(L))_{\phi}^\varepsilon = \dim_L \hc{q}(K,\sL_{\lambda}(L))_\phi^\varepsilon = 1.
		\end{equation}
		In practice, $\varepsilon$ will often be fixed and implicit, and we simply write `of cohomological multiplicity one'. 
	\end{definition}

	\begin{example}\label{ex:mult one}
		Let $G = \mathrm{Res}_{F/\Q}\GL_n$, for $F$ a number field, and let $\pi$ be a RACAR of $G(\A)$ such that the $\phi$-eigenspace of $\bT(K)$ acting on $\pi_f^{K}$ is 1-dimensional.  Then by Matsushima's formula $\phi$ is cohomologically of multiplicity one at any $\varepsilon \in \cE_\pi$ (see Definition \ref{def:admissible signs}).
		
		If $K^p$ is the Whittaker new level, $\pi_p$ is $p$-ordinary, and the $p$-part $\phi|_{\sA_p^+}$ of $\phi$ is the unique ordinary eigensystem in $\pi_p^{\opn{Iw}}$, then $\pi_f^{K}$ is 1-dimensional. Moreover, if $F = \Q$, then $K_\infty/K_\infty^\circ \cong \{\pm 1\}$, and $\cE_\pi$ either contains both possible signs (if $n$ is even) or precisely one sign, depending on the parity of the central character of $\pi_\infty$ (if $n$ is odd); see \cite[(3.2)]{Mah05}.
	\end{example}

\begin{remark}\label{rem:more general}
There are more general examples of $\pi$ satisfying \eqref{eq:mult one def}. An obvious variant of the above is when $G$ is a product of general linear groups, capturing Rankin--Selberg and triple product cases.

More generally, there are multiple conditions that need to be satisfied in order for $\pi$ to be of cohomological multiplicity one with respect to a local system $E$. More precisely, from the description of cuspidal cohomology as in \cite{FS98}, one needs:
\begin{enumerate}
\item The multiplicity of $\pi$ in the discrete spectrum of cuspidal automorphic forms on $G(\mbb{A})$ to be one. In the notation \emph{op.cit.}, this implies that $A_{E, \{G\}, \pi} \cong \pi$.
\item Strong multiplicity one holds for $\phi$, namely, $\pi$ is the unique (up to isomorphism) cuspidal automorphic representation with Hecke eigensystem $\phi$.
\item In the notation \emph{op.cit.}, the Lie algebra cohomology groups (with fixed sign $\varepsilon$) 
\[
\opn{H}^{q}(\mathfrak{m}_G, K_{\infty}; \pi_{\infty} \otimes E)^{\varepsilon} \quad \text{ and } \quad \opn{H}^{q+\ell}(\mathfrak{m}_G, K_{\infty}; \pi_{\infty} \otimes E)^{\varepsilon}
\]
are both $1$-dimensional.
\item The invariants $\pi_f^K$ are $1$-dimensional.
\end{enumerate}
For a non general linear example where all of these conditions are satisfied, one can take $\pi$ to be a discrete series cuspidal automorphic representation of a rank $n$ definite unitary group $G$ (defined with respect to an imaginary quadratic number field $F$), which is ramified only at primes which split in the extension $F/\mbb{Q}$ and whose automorphic base-change to $\opn{GL}_n(\mbb{A}_F)$ is cuspidal. Condition (1) holds by Arthur's multiplicity formula \cite[Thm.\ 2.6]{CZ25}; condition (2) is a consequence (using the ramification assumption) of strong multiplicity one for cuspidal automorphic representations of $\opn{GL}_{n}(\mbb{A}_F)$ and properties of automorphic base-change from $G(\mbb{A})$ to $\opn{GL}_{n}(\mbb{A}_F)$ \cite{ShinAppendix}; condition (3) essentially follows from Schur's lemma (noting that $q = \ell = 0$ in this setting); and condition (4) can be guaranteed if we take $K$ to be equal to the Whittaker new level at primes where $\pi$ is ramified, and maximal special at unramified primes. Such a $\pi$ is the main object of study in forthcoming work of Th{\o}gersen \cite{ThogersenUnitary}.
\end{remark}

    \begin{proposition}\label{prop:matsushima}
        Suppose $\phi$ is a classical cuspidal tempered eigenpacket of weight $\lambda$ of cohomological multiplicity one at $\varepsilon$. Then for $0 \leq i \leq \ell$, we have
     		\[
		\dim_L \hc{q+i}(K,\sL_{\lambda}(L))_{\phi}^\varepsilon = \left(\begin{smallmatrix}\ell \\ i \end{smallmatrix}\right).
		\]
    \end{proposition}
    \begin{proof}
    This is  Matsushima's formula \cite[Thm.\ III.5.1]{BW00} (see, e.g., \cite[(5)]{Ven19}).
    \end{proof}

	 With this, we have completely explained the statement of Theorem \ref{thm:main etale}, and we now turn to the proof. We will prove a number of auxiliary results, assuming varying combinations of conditions (a), (b) and (c); and combine them to deduce Theorem \ref{thm:main etale}. 

     In the proof, we will require some further relevant results from \cite{Han17}:
	\begin{facts}\label{facts:hansen other}
		\begin{itemize}
			\item[(A)] By \cite[Thm.\ 3.3.1, Rem.\ 3.3.2]{Han17}, if $(U_t,\Omega,h)$ is a slope datum and $\Lambda \subset \Omega$ a rigid Zariski-closed subspace, there is a convergent second quadrant Hecke-equivariant Tor spectral sequence
			\[
			E_2^{i,j} = \opn{Tor}_{-i}^{\cO(\Omega)}\Big(\hc{j}(K,\sD_\Omega)^{\leq h}, \cO(\Lambda)\Big) \Rightarrow \hc{i+j}(K,\sD_\Lambda)^{\leq h}.
			\]
			
			\item[(B)] Using \cite[Prop.\ 4.5.2]{Han17}, for $\phi$ a non-critical classical tempered cuspidal eigenpacket of weight $\lambda$, and $(U_t,\Omega,h)$ a slope datum such that $\lambda \in \Omega$, the largest $i$ for which $\hc{i}(K,\sD_\Omega)_\phi^{\leq h} \neq 0$ is $i = q+\ell$. 
			
			\item[(C)] Let $\phi$ and $\Omega$ be as in (B). Suppose the non-abelian Leopoldt conjecture holds for $\phi$, and $\Omega$ is open in $\sW$. Then by the last line of Newton's appendix to \cite{Han17}, $\hc{i}(K,\sD_\Omega)_\phi^{\leq h} \neq 0$ if and only if $i = q+\ell$.

            \item[(D)]   For $\varepsilon \in \cE_\phi$, facts (A), (B) and (C) above all apply identically with $\varepsilon$-parts everywhere.
		\end{itemize}
	\end{facts}

	\subsection{Results assuming only (a) and (b) of Theorem \ref{thm:main etale}}

	\begin{proposition}\label{prop:top cyclic}
		Let $\phi$ be a non-critical classical cuspidal tempered eigenpacket of weight $\lambda$. 
		\begin{itemize}
			\item[(i)] Let $(U_t,\Omega,h)$ be a slope datum, and $\Lambda \subset \Omega$ a Zariski-closed affinoid containing $\lambda$. Then
			\[
			\hc{q+\ell}(K,\sD_\Omega)_\phi^{\leq h} \otimes_{\cO(\Omega)_\lambda}\cO(\Lambda)_\lambda \cong \hc{q+\ell}(K,\sD_\Lambda)_\phi^{\leq h}.
			\]
			\item[(ii)] If $\phi$ is of cohomological multiplicity one at $\varepsilon$, then there exists an ideal $J_\phi^\varepsilon \subset \cO(\Omega)_\lambda$ such that
			\begin{equation}\label{eq:top degree cyclic}
				\hc{q+\ell}(K,\sD_\Omega)^{\leq h, \varepsilon}_\phi \cong \cO(\Omega)_\lambda/J_\phi^\varepsilon.
			\end{equation}
		\end{itemize}
	\end{proposition}
	\begin{proof}
		We localise the Tor spectral sequence (for $\Omega$ and $\Lambda$) at $\phi$. Note that the terms in the limit that contribute to the grading on $\hc{q+\ell}(K,\sD_\Lambda)_\phi$ are $(E_\infty^{i,q+\ell-i})_\phi$. As $E_2^{i,j} = 0$ for $i > 0$, and $(E_2^{i,j})_\phi = 0$ for all $j > q+\ell$ by fact (B) above, we have $(E_\infty^{i,q+\ell-i})_\phi = 0$ unless $i = 0$. Moreover note
		\begin{equation}\label{eq:spectral sequence stabilise}
			E_3^{0,q+\ell} =  \ker(E_2^{0,q+\ell}\to E_2^{2,q+\ell-1})/\opn{im}(E_2^{-2,q+\ell+1}\to E_2^{0,q+\ell});
		\end{equation}
		localising at $\phi$ we find $(E_3^{0,q+\ell})_\phi = (E_2^{0,q+\ell})_\phi$, and continuing $(E_r^{0,q+\ell})_\phi$ stabilises at the $E_2$ page. Part (i) follows as
		\begin{align*} \hc{q+\ell}(K,\sD_\Omega)_\phi^{\leq h} \otimes_{\cO(\Omega)_\lambda}\cO(\Lambda)_\lambda &= (E_2^{0,q+\ell})_\phi = (E_\infty^{0,q+\ell})_\phi \cong \hc{q+\ell}(K,\sD_\Lambda)_\phi^{\leq h}.
		\end{align*}
		
		To see (ii), first apply (i) to $\Lambda = \{\lambda\}$, giving
		\begin{equation}\label{eq:spec weight lambda}
		\hc{q+\ell}(K,\sD_\Omega)_\phi^{\leq h} \otimes_{\cO(\Omega)_\lambda}\cO(\Omega)_\lambda/\m_\lambda \cong \hc{q+\ell}(K,\sD_\lambda)_\phi^{\leq h}.
		\end{equation}
		Now $\hc{q+\ell}(K,\sD_\lambda)^{\leq h, \varepsilon}_\phi$ is 1-dimensional by the multiplicity one assumption and non-criticality. Taking $\varepsilon$-parts in \eqref{eq:spec weight lambda}, and lifting a generator via Nakayama's lemma, we see $\hc{q+\ell}(K,\sD_\Omega)_\phi^{\leq h, \varepsilon}$ is a cyclic $\cO(\Omega)_\lambda$ module, and hence deduce \eqref{eq:top degree cyclic}. 
	\end{proof}

	As $\cO(\Omega)_\lambda$ is Noetherian, $J_\phi^\varepsilon$ is finitely generated. Up to shrinking $\Omega$ to avoid denominators, we may assume that $J_\phi^\varepsilon \subset \cO(\Omega)$ is defined globally. We define
	\begin{equation}\label{eq:sigma}
		\Sigma = \Sigma^\varepsilon := \opn{Sp}\big(\cO(\Omega)/J_\phi^\varepsilon\big).
	\end{equation}
    To ease notation, we will henceforth fix\footnote{Note that if $\Sigma$ contains a Zariski-dense set of classical weights such that $h$ is a non-critical slope, then one can show $\Sigma$ is independent of the admissible sign $\varepsilon$ using the methods of \cite[Cor.\ 7.22]{BDW20}.} $\varepsilon$ and just write $\Sigma$.

	\begin{corollary}\label{cor:top degree etale}
		Let $\phi$ be a non-critical classical cuspidal tempered eigenpacket of cohomological multiplicity one at $\varepsilon$, and let $\Sigma$ be as in  \eqref{eq:sigma}. Then $\hc{q+\ell}(K,\sD_\Sigma)^{\leq h, \varepsilon}_\phi$ is free of rank one over $\cO(\Sigma)_\lambda$.
	\end{corollary}
	
	\begin{proof}
		By definition $\cO(\Omega)_\lambda/J_\phi^\varepsilon \cong \cO(\Sigma)_\lambda$. Applying part (i) of the proposition with $\Lambda = \Sigma$ yields 
		\begin{align*}
			\hc{q+\ell}(K,\sD_\Sigma)^{\leq h, \varepsilon}_\phi &\cong \hc{q+\ell}(K,\sD_\Omega)_\phi^{\leq h, \varepsilon} \otimes_{\cO(\Omega)_\lambda}\cO(\Sigma)_\lambda\\
			&\cong \cO(\Sigma)_\lambda \otimes_{\cO(\Omega)_\lambda} \cO(\Sigma)_\lambda \cong \cO(\Sigma)_\lambda,
		\end{align*}
		where we have used (ii) in the second isomorphism. The result is immediate.
	\end{proof}

In \cite[Def.\ 4.3.2]{Han17} Hansen also constructs eigenvarieties $\sX_{G,K}^n$ using only degree $n$ cohomology, whose points parametrise finite slope eigensystems appearing in $\hc{n}(K,\sD_\lambda)$. Applying the same constructions with just the $\varepsilon$-parts of the cohomology, one obtains eigenvarieties $\sX_{G,K}^{n,\varepsilon}$ parametrising finite slope eigensystems in $\hc{n}(K,\sD_\lambda)^\varepsilon$. Without needing to assume non-abelian Leopoldt, the results above show that the top degree eigenvariety for $\varepsilon$ is \'etale  over its image in weight space at $x_\phi$. Precisely:

	\begin{corollary}\label{cor:top etale}
		Let $\phi$ be a non-critical classical cuspidal tempered eigenpacket of cohomological multiplicity one at $\varepsilon$ and let $\Sigma$ be as in (\ref{eq:sigma}). Then
		\[
		\sX_{G,K}^{q+\ell, \varepsilon}\times_{\sW}\Sigma\longrightarrow \Sigma
		\]
		is \'etale at $x_\phi$. 
	\end{corollary}

	\begin{proof}
		We take $(U_t,\Omega,h)$ a slope datum as above, such that $\Omega$ is open in $\sW$. The local ring $\bT_{\Omega,h}^{q+\ell, \varepsilon}(K)_\phi$ of $\sX_{G,K}^{q+\ell, \varepsilon}$ at $x_\phi$ is the (localisation of the) image of $\bT(K) \otimes \cO(\Omega)$ in $\End_{\cO(\Omega)}(\hc{q+\ell}(K,\sD_\Omega)^{\leq h, \varepsilon})$. In this case Proposition \ref{prop:top cyclic}(ii) implies that
		\[
		0 \neq \bT_{\Omega,h}^{q+\ell, \varepsilon}(K)_\phi \subset \End_{\cO(\Omega)_\lambda}(\hc{q+\ell}(K,\sD_\Omega)^{\leq h, \varepsilon}_\phi) \cong \cO(\Omega)_\lambda/J_\phi^\varepsilon
		\]
		as $\cO(\Omega)_\lambda$-modules, hence $\bT_{\Omega,h}^{q+\ell,\varepsilon}(K)_\phi \cong \cO(\Omega)_\lambda/J_\phi^\varepsilon \cong \cO(\Sigma)_\lambda$ (as $1\in \bT_{\Omega,h}^{q+\ell,\varepsilon}(K)_\phi$). Thus
		\[
		\bT_{\Omega,h}^{q+\ell, \varepsilon}(K)_\phi\otimes_{\cO(\Omega)_\lambda}\cO(\Sigma)_\lambda \cong \cO(\Sigma)_\lambda
		\]
		is the local ring of $\sX_{G,K}^{q+\ell, \varepsilon}\times_{\sW}\Sigma$ at $x_\phi$, from which we deduce the claimed \'etaleness. 
	\end{proof}

	\subsection{Results assuming (a), (b) and (c)  of Theorem \ref{thm:main etale}}

In this section, we additionally assume the non-abelian Leopoldt conjecture for $\phi$ (at $\varepsilon$): so every irreducible component of $\sX_{G,K}^\varepsilon$ through $x_\phi$ has dimension $\dim \sW - \ell$, for $\ell = \ell(G)$. An immediate consequence is:

\begin{corollary}\label{cor:main etale}
		Let $\phi$ be a non-critical classical cuspidal tempered eigenpacket of cohomological multiplicity one at $\varepsilon$. If the non-abelian Leopoldt conjecture holds for $\phi$ at $\varepsilon$, then 
		\[
		\sX_{G,K}^\varepsilon \times_{\sW}\Sigma\longrightarrow \Sigma
		\]
		is \'etale at $x_\phi$.
	\end{corollary}
	\begin{proof}
        We are assuming non-abelian Leopoldt; thus, by facts (C) and (D) above, $\hc{\bullet}(K,\sD_\Omega)^{\leq h, \varepsilon}_\phi = \hc{q+\ell}(K,\sD_\Omega)^{\leq h, \varepsilon}_\phi$, so $\bT_{\Omega,h}^\varepsilon(K)_\phi = \bT_{\Omega,h}^{q+\ell, \varepsilon}(K)_\phi$. The result follows from Corollary \ref{cor:top etale}.
	\end{proof}

    We also have the following (without any multiplicity one).

	\begin{proposition}\label{prop:tor iso}
		Let $\phi$ be a non-critical classical cuspidal tempered eigenpacket, and suppose the non-abelian Leopoldt conjecture holds for $\phi$. Let $\Lambda \subset \Omega$ be a Zariski-closed affinoid containing $\lambda$. Then for $k \geq 0$, we have
        \begin{equation}\label{eq:tor iso}
            \opn{Tor}_k^{\cO(\Omega)_\lambda}\Big(\hc{q+\ell}(K,\sD_\Omega)^{\leq h}_\phi, \cO(\Lambda)_\lambda\Big) \cong \hc{q+\ell-k}(K,\sD_\Lambda)^{\leq h}_\phi.
        \end{equation}
        The same holds when replacing the cohomology with $\varepsilon$-parts.
      \end{proposition}

\begin{proof}
We localise the Tor spectral sequence (for $\Omega$ and $\Lambda$) at $\phi$ (though we omit this from notation). By Fact \ref{facts:hansen other}(C), $E_2^{i,j} = 0$ unless $j = q+\ell$, so the spectral sequence degenerates at the $E_2$ page; in particular, for $k \geq 0$, we have 
\begin{equation}\label{eq:k degen}
E_2^{-k,q+\ell} \cong E_\infty^{-k,q+\ell}.
\end{equation}
The terms contributing to the grading on $\hc{q+\ell-k}(K,\sD_\Lambda)_\phi^{\leq h}$ are the $E_\infty^{-k-i,q+\ell+i}$. As subquotients of $E_{2}^{-k-i,q+\ell+i}$, these vanish unless $i=0$. Thus $E_{\infty}^{-k,q+\ell} \cong \hc{q+\ell-k}(K,\sD_\Lambda)_\phi^{\leq h}$, whence \eqref{eq:tor iso} follows from  \eqref{eq:k degen}. The $\varepsilon$-statement follows similarly using Fact \ref{facts:hansen other}(D).
\end{proof}

Henceforth we assume all assumptions from Theorem \ref{thm:main etale}. Thus:
\begin{myquote}
$\phi$ is a non-critical classical cuspidal tempered eigenpacket of cohomological multiplicity one at $\varepsilon$, and we assume non-abelian Leopoldt holds for $\phi$. 
\end{myquote}
Then Proposition \ref{prop:top cyclic}(ii) yields $J_\phi^\varepsilon \subset \cO(\Omega)$ cutting out a Zariski-closed subspace $\Sigma \subset \Omega$ (as in \eqref{eq:sigma}). 

    	\begin{lemma}\label{lem:equidimensional}
		Assume $(\dagger)$. After possibly shrinking $\Omega$:
        \begin{itemize}
        \item[(i)] $\Sigma$ is equidimensional of dimension $\dim \sW - \ell$;
        \item[(ii)] there exists an open neighbourhood 
        \[
            \sY \subset \opn{Sp}(\bT_{\Omega,h}^\varepsilon(K)) \subset \sX_{G,K}^\varepsilon
        \]
        of $x_\phi^\varepsilon$ such that $\Sigma = w(\sY)$. 
        \end{itemize}
	\end{lemma}
	
	\begin{proof}
		From the definition, irreducible components $\sV$ of $\Sigma$ are in bijection with minimal primes $\fp$ of $\cO(\Omega)$ containing $J_\phi^\varepsilon$ (that is, minimal with respect to this containment). As $\cO(\Sigma)$ is Noetherian, there are only finitely many. Up to shrinking the original $\Omega$, we may thus assume that every such $\fp$ is contained in $\m_\lambda$, that is, the irreducible components of $\Sigma$ are in bijection with minimal primes $\fp$ of $\cO(\Omega)_{\lambda}$ containing $J^{\varepsilon}_\phi \cO(\Omega)_\lambda$ (and the corresponding $\sV$ contains $\lambda$).
		
		From the proofs of Corollaries \ref{cor:top etale} and \ref{cor:main etale}, $\bT_{\Omega,h}^\varepsilon(K)_\phi$ is free of rank one over $\cO(\Omega)_\lambda/J_\phi^\varepsilon \cO(\Omega)_\lambda$. In particular any minimal prime $\fp$ as above arises as the contraction of a minimal prime $\fP$ of $\bT_{\Omega,h}^\varepsilon(K)_\phi$, corresponding to an irreducible component $\sI$ of $\sX_{G,K}^\varepsilon$ through $x_\phi$; and $w : \sI \to \sV$ is \'etale. Thus $\dim \sV = \dim\sI = \dim \sW - \ell$, where the last equality is non-abelian Leopoldt. This proves (i).
        
        To see (ii), note that the identification of $\sV$ to $\sI$ gives a bijection between the irreducible components $\{\sV_i\} \subset \Sigma$ containing $\lambda$ and $\{\sI_i\} \subset \opn{Sp}(\bT_{\Omega,h}^\varepsilon(K))$ containing $x_\phi^{\varepsilon}$. Let $\sY = \bigcup \sI_i$ be the union of all such irreducible components of $\opn{Sp}(\bT_{\Omega,h}^\varepsilon(K))$. Then $\Sigma = w(\sY)$. Further, by construction $\sY$ is the connected component of $\opn{Sp}(\bT_{\Omega,h}^\varepsilon(K))$ through $x_\phi^\varepsilon$. Thus $\sY$ is a rigid open neighbourhood of $x_\phi^\varepsilon$ in $\opn{Sp}(\bT_{\Omega,h}^\varepsilon(K))$, hence in $\sX_{G,K}^\varepsilon$, as required.
	\end{proof}

With this, we have proved the entirety of assertion (i) of Theorem \ref{thm:main etale}, as well as its final sentence. It remains to prove (ii).
    
\begin{theorem}\label{thm:regular sequence}
Assume $(\dagger)$. Then:
		\begin{itemize}
			\item[(i)] The ideal $J_\phi^\varepsilon\cO(\Omega)_\lambda \subset \cO(\Omega)_\lambda$ can be generated by a regular sequence $y_1,\dots, y_{\ell} \in \cO(\Omega)_\lambda$. (In particular, $\Sigma$ is a local complete intersection at $\lambda$).

            \item[(ii)] For $0 \leq i \leq \ell$, the space $\hc{q+i}(K,\sD_\Sigma)^{\leq h, \varepsilon}_\phi$ is free of rank $\binom{\ell}{i}$ over $\cO(\Sigma)_\lambda$.
      \end{itemize}
      \end{theorem}

The proof of this result is where we crucially follow the strategy of \cite[Thm.\ 4.9]{HT17}.

\begin{proof}
(i) Applying Proposition \ref{prop:tor iso} with $\Lambda = \{\lambda\}$ and $k=1$, we have
\begin{align*}
            \opn{Tor}_1^{\cO(\Omega)_\lambda}\Big(\hc{q+\ell}(K,\sD_\Omega)^{\leq h, \varepsilon}_\phi, \cO(\Omega)/\m_\lambda\Big) &\cong \hc{q+\ell-1}(K,\sD_\lambda)^{\leq h, \varepsilon}_\phi\\ 
            &\cong \hc{q+\ell-1}(K,\sL_\lambda)^{\leq h, \varepsilon}_\phi,
\end{align*}
where the second isomorphism is non-criticality. By the cohomological multiplicity 1 assumption and Proposition \ref{prop:matsushima}, this is an $L$-vector space of dimension $\ell = \ell(G)$. Now, by Proposition \ref{prop:top cyclic} we also know 
\begin{align*}
            \opn{Tor}_1^{\cO(\Omega)_\lambda}\Big(\hc{q+\ell}(K,\sD_\Omega)^{\leq h, \varepsilon}_\phi, \cO(\Omega)/\m_\lambda \Big) &\cong \opn{Tor}_1^{\cO(\Omega)_\lambda}\Big(\cO(\Omega)_\lambda/J_\phi^\varepsilon\cO(\Omega)_\lambda, \cO(\Omega)/\m_\lambda\Big)\\
            &\cong J_\phi^\varepsilon\cO(\Omega)_\lambda \otimes_{\cO(\Omega)_\lambda} \cO(\Omega)/\m_\lambda. 
\end{align*}
Combining these two expressions for $\opn{Tor}_1^{\cO(\Omega)_\lambda}(-)$, we find that
\[
J_\phi^\varepsilon\cO(\Omega)_\lambda \otimes_{\cO(\Omega)_\lambda} \cO(\Omega)/\m_\lambda \cong \hc{q+\ell-1}(K,\sL_\lambda)^{\leq h, \varepsilon}_\phi
\]
is an $\ell$-dimensional $L$-vector space. Applying Nakayama's lemma, we see that $J_\phi^\varepsilon\cO(\Omega)_\lambda$ is generated by $\ell$ elements $y_1,\dots,y_\ell$. 

By Lemma \ref{lem:equidimensional}, we know $\dim \Sigma = \dim\Omega - \ell$. Since $\cO(\Sigma)_\lambda = \cO(\Omega)_\lambda/(y_1,\dots,y_\ell)$, this is only possible if the sequence $y_1,\dots,y_\ell$ is regular, as claimed.

\medskip

(ii) Applying Proposition \ref{prop:tor iso} with $\Lambda = \Sigma$, we have
\begin{align*}
\hc{q+i}(K,\sD_\Sigma)^{\leq h, \varepsilon}_\phi &\cong \opn{Tor}_{\ell-i}^{\cO(\Omega)_\lambda}\Big(\hc{q+\ell}(K,\sD_\Omega)^{\leq h, \varepsilon}_\phi, \cO(\Sigma)_\lambda\Big)\\
&\cong\opn{Tor}_{\ell-i}^{\cO(\Omega)_\lambda}\Big(\cO(\Omega)_\lambda/J_\phi^\varepsilon \cO(\Omega)_\lambda, \cO(\Omega)_\lambda/J_\phi^\varepsilon \cO(\Omega)_\lambda \Big),
\end{align*}
the second isomorphism by Proposition \ref{prop:top cyclic}(ii) and definition of $\Sigma$. As $J_\phi^\varepsilon \cO(\Omega)_\lambda$ is generated by a regular sequence of length $\ell$ by (i), via the Koszul complex we compute these Tor groups to be free of rank $\binom{\ell}{\ell-i} = \binom{\ell}{i}$ over $\cO(\Omega)_\lambda/J_\phi^\varepsilon \cO(\Omega)_\lambda = \cO(\Sigma)_\lambda$, as required.
\end{proof}

   Combining all of the above, we have proved Theorem \ref{thm:main etale}. As a consequence, we also have a partial generalisation of Proposition \ref{prop:top cyclic}(i) to all degrees.

\begin{corollary}
Suppose $(\dagger)$. For all $q \leq d \leq q+\ell$, we have an isomorphism
\[
    \hc{d}(K, \sD_\Sigma)_\phi^{\leq h,\varepsilon} \otimes_{\cO(\Sigma)_\lambda} \cO(\Sigma)/\m_\lambda \cong \hc{d}(K,\sD_\lambda)_\phi^{\leq h, \varepsilon}.
\]
\end{corollary}

\begin{proof}
Since $\hc{\bullet}(K, \sD_\Sigma)_\phi^{\leq h,\varepsilon}$ is free over $\cO(\Sigma)_\lambda$, the relative Tor spectral sequence
\[
    E_2^{i,j} = \opn{Tor}_{-i}^{\cO(\Sigma)_\lambda}(\hc{j}(K,\sD_\Sigma)_\phi^{\leq h, \varepsilon}, \cO(\Sigma)/\m_\lambda) \Rightarrow \hc{i+j}(K,\sD_\lambda)_\phi^{\leq h, \varepsilon}
\]
degenerates on the $E_2$ page. In particular, we deduce that for all $d$, the natural map
\begin{align*}
\opn{Tor}_{0}^{\cO(\Sigma)_\lambda}(\hc{d}(K,\sD_\Sigma)_\phi^{\leq h, \varepsilon}, \cO(\Sigma)/\m_\lambda) &=  \hc{d}(K, \sD_\Sigma)_\phi^{\leq h,\varepsilon} \otimes_{\cO(\Sigma)_\lambda} \cO(\Sigma)/\m_\lambda\\
&\longhookrightarrow \hc{d}(K,\sD_\lambda)_\phi^{\leq h, \varepsilon}
\end{align*}
is injective. Since both sides have dimension $\binom{\ell}{d}$ we conclude.
\end{proof}

\begin{remark}
Special cases of this result have been proved before, and have been used in constructions of $p$-adic $L$-functions; see, for example, \cite[\S4]{BW18} for the case of imaginary quadratic fields. We hope that the above may prove useful in other settings, for example the case of $\GL_4$ over an imaginary quadratic field; here the natural construction of $p$-adic $L$-functions uses neither top nor bottom degree, but rather a degree in between (as explained in \cite{Wil-2n}).
\end{remark}

\subsection{Intermediate affinoids}\label{sec:intermediate}
For completeness, we record some further results describing the structure of overconvergent cohomology over intermediate affinoids. Assume $(\dagger)$, and let $y_1, \dots, y_\ell$ be as in Theorem \ref{thm:regular sequence}(i). Possibly shrinking $\Omega$ to avoid denominators, we may assume that each $y_j \in \cO(\Omega)$ is global. For $0 \leq j \leq \ell$, let 
   \[
    J_{\phi,j}^{\varepsilon} \defeq \langle y_1, \dots, y_j\rangle \subset \cO(\Omega),
    \]
    and define $\Sigma_j = \Sigma_j^\varepsilon \defeq \opn{Sp}(\cO(\Omega)/J_{\phi,j}^\varepsilon$; then
    \[
        \Omega = \Sigma_0 \supset \Sigma_1 \supset \cdots \supset \Sigma_\ell = \Sigma,
    \]
    and each $\Sigma_{j+1}$ has codimension 1 in $\Sigma_{j}$.
    
    \begin{corollary}\label{cor:intermediate}
        Assume $(\dagger)$. For each $0 \leq j \leq \ell$:
        \begin{itemize}
        \item[(i)] $\hc{\bullet}(K,\sD_{\Sigma_j})^{\leq h,\varepsilon}_\phi$ vanishes above degree $q+\ell$; and
        \item[(ii)] For $k \geq 0$, the space $\hc{q+\ell-k}(K,\sD_{\Sigma_j})_\phi^{\leq h, \varepsilon}$ is free of rank $\binom{j}{k}$ over $\cO(\Sigma)_\lambda$.
\end{itemize}
    \end{corollary}
\begin{proof}
By Proposition  \ref{prop:tor iso} with $\Lambda = \Sigma_j$, we have
\begin{align*}
\hc{q+\ell-k}(K,\sD_{\Sigma_j})^{\leq h, \varepsilon}_\phi &\cong \opn{Tor}_{k}^{\cO(\Omega)_\lambda}\Big(\hc{q+\ell}(K,\sD_\Omega)^{\leq h, \varepsilon}_\phi, \cO(\Sigma_j)_\lambda\Big)\\
&\cong\opn{Tor}_{k}^{\cO(\Omega)_\lambda}\Big(\cO(\Omega)_\lambda/(y_1,\dots,y_\ell), \ \cO(\Omega)_\lambda/(y_1,\dots,y_j)\Big).
\end{align*}
As in Theorem \ref{thm:regular sequence}, we compute these Tor groups using the Koszul complex to conclude.
\end{proof}

In particular, $\hc{\bullet}(K,\sD_{\Sigma_j})_\phi^{\leq h, \varepsilon}$ is concentrated in degrees $\{q+\ell-j, \dots, q+\ell\}$, and is free of rank one over $\cO(\Sigma)_\lambda$ in degrees $q+\ell-j$ and $q+\ell$. Note that taking $j = \ell$ (and $k = \ell-i$), Corollary \ref{cor:intermediate} recovers exactly Theorem \ref{thm:regular sequence}(ii).

Corollary \ref{cor:intermediate} implies that for $j=0$ and $j=\ell$, the specialisation map
\[
\hc{q+\ell-j}(K,\sD_{\Sigma_j})_\phi^{\leq h, \varepsilon} \longrightarrow \hc{q+\ell-j}(K,\sD_\lambda)_\phi^{\leq h, \varepsilon}
\]
is surjective. The analogous statement for $1 \leq j \leq \ell-1$ obviously fails; the target is an $\binom{\ell}{j}$-dimensional $L$-vector space, whilst the image of the left-hand side is 1-dimensional. Instead, we have the following generalisation to such $j$. If $J \subset \{ 1, \dots, \ell\}$ is a subset of size $j$, we let $\Sigma_J \subset \Omega$ denote the vanishing locus $\{ y_j = 0 : j \in J \}$, which has codimension $j$.

\begin{proposition}\label{prop:spec surjective}
Assume $(\dagger)$. The sum of the natural specialisation maps
\begin{equation}\label{eq:specialisation oplus}
\bigoplus_{\substack{J \subset \{ 1, \dots, \ell\} \\ \#J = j}} \opn{H}^{q+\ell-j}_c(K, \mathscr{D}_{\Sigma_J})_{\phi}^{\leq h, \varepsilon} \isorightarrow \opn{H}^{q+\ell-j}_c(K, \mathscr{D}_{\Sigma})_{\phi}^{\leq h, \varepsilon} 
\end{equation}
is an isomorphism. In particular, the sum of the natural specialisation maps
\[
\bigoplus_{\substack{J \subset \{ 1, \dots, \ell\} \\ \#J = j}} \opn{H}^{q+\ell-j}_c(K, \mathscr{D}_{\Sigma_J})_{\phi}^{\leq h, \varepsilon} \longrightarrow \opn{H}^{q+\ell-j}_c(K, \mathscr{D}_{\lambda})_{\phi}^{\leq h, \varepsilon} 
\]
is surjective.
\end{proposition}

Note by Corollary \ref{cor:intermediate}, we already know both sides of \eqref{eq:specialisation oplus} are free over $\mathcal{O}(\Sigma)_{\lambda}$ of rank $\binom{\ell}{j}$. 

\begin{proof}
 Let $R = \cO(\Omega)_\lambda$, let $\by = (y_1,\dots,y_\ell) \subset R$, and for $J \subset \{1,\dots,\ell\}$, let $\by_J \defeq (y_j : j \in J)$. Let $M = \opn{H}^{q+\ell}_c(K, \mathscr{D}_{\Omega})_{\phi}^{\leq h, \varepsilon} \cong R/\by$, the latter isomorphism by Proposition \ref{prop:top cyclic}. 

For any $J$, let $K_\bullet^J(R)$ denote the Koszul complex with respect to the regular sequence $(y_j : j\in J)$, and let $K_\bullet^J(M) \defeq K_\bullet^J(R) \otimes_R M$. For $J = \{1,\dots,\ell\}$, we drop the superscripts, writing $K_\bullet(R)$ and $K_\bullet(M)$. 

There is a canonical map $\oplus_{\#J = j} K_\bullet^J(R) \to K_\bullet(R)$, and this is an isomorphism of the degree $-j$ terms (with both sides free $R$-modules of rank $\binom{\ell}{j}$). Tensoring with $M$, we see 
\begin{equation}\label{eq:koszul oplus}
    \bigoplus_{\substack{J \subset \{ 1, \dots, \ell\} \\ \#J = j}} K_{-j}^J(M) \isorightarrow K_{-j}(M).
\end{equation}

Now, note the differentials in $K_\bullet^J(M)$ are all zero, so for $k \geq 0$ we have $\h_{-k}(K_\bullet^J(M)) \cong K_{-k}^J(M)$. By definition of Tor, we also know
\begin{align}
     K_{-k}^J(M) \cong \h_{-k}(K_\bullet^J(M)) &= \opn{Tor}_k^R(M,R/\by_J) \notag\\
    &= \opn{Tor}_k^{\cO(\Omega)_\lambda}(\hc{q+\ell}(K,\sD_\Omega)_\phi^{\leq h, \varepsilon}, \cO(\Sigma_J)_\lambda) \notag\\
    & \cong \hc{q+\ell-k}(K,\sD_{\Sigma_J})_\phi^{\leq h, \varepsilon},\label{eq:tor homology}
\end{align}
using Proposition \ref{prop:tor iso} in the last isomorphism. Taking $k = j$ and substituting this into \eqref{eq:koszul oplus} yields exactly the claimed isomorphism.
\end{proof}

\section{\texorpdfstring{$\cL$}{L}-invariants}\label{s:L-invariants}

From now on we take $G= \mathrm{GL}_n$ over $\mbb{Q}$, and denote by $B\subset G$ the upper triangular Borel subgroup, by $N$ its unipotent radical and by $T$ the torus of diagonal matrices. We denote the opposite subgroups by $\overline{(-)}$.

Let $\pi$ be a $p$-ordinary RACAR of $\GL_n(\A)$ such that $\pi_p$ is the Steinberg representation of $\GL_n(\Qp)$. We let $G_{\Q} \defeq \opn{Gal}(\overline{\Q}/\Q)$, fix $L$ a sufficiently large finite extension of $\Q_p$, and denote by $\rho_{\pi}\colon G_{\Q} \to \mathrm{GL}_n(L)$ the continuous $p$-adic Galois representation associated with $\pi$ (see, for example, \cite{ScholzeTorsion, HLTTrigid}).

\subsection{Steinberg representations and extensions}\label{sec:extensions of steinberg} \subsubsection{Steinberg representations} If $c\in \{1,..., n-1\}$, we let $B \subset P_c \subset G$ be the maximal standard parabolic of type $(c, n-c)$, and let $M_c = \GL_c \times \GL_{n-c}$ be its Levi subgroup. Let $\Delta = \{ \alpha_c : c=1, \dots, n-1\}$ be the set of simple roots of $G$; then under the usual correspondence between standard parabolics and subsets of $\Delta$, $P_c$ corresponds to $\Delta \backslash \{\alpha_c\}$. We also denote by $Q_c$ the standard parabolic corresponding to the subset $\{\alpha_c\}$, that is, the parabolic of type $(1..., 1, 2, 1,..., 1)$ with the first $c-1$ entries equal to $1$. 

Given a parabolic subgroup $P$ of $G$ containing $B$, we denote by $\mathrm{I}^{\mathrm{sm}}_{\overline{P}}(L)$ the $L[G(\Q_p)]$-module of smooth functions $f: \overline{P}(\Q_p)\backslash G(\Q_p)\rightarrow L$. We endow this with a left $G(\Q_p)$-action given by right translation. Define three left-$G(\Q_p)$-representations by
\begin{align*}
\mathrm{St}^{\mathrm{sm}}(L)= \mathrm{I}^{\mathrm{sm}}_{\overline{B}}(L)\Big/ \sum_{B\subsetneqq P\subset G} & \mathrm{I}^{\mathrm{sm}}_{\overline{P}}(L)\text{,} \qquad v_{c}^{\mathrm{sm}}(L) = \mathrm{I}^{\mathrm{sm}}_{\overline{Q_c}}(L)\Big/ \sum_{Q_c\subsetneqq P\subset G}\mathrm{I}^{\mathrm{sm}}_{\overline{P}}(L), \\
&\mathrm{St}_c^{\mathrm{sm}}(L)= \mathrm{I}^{\mathrm{sm}}_{\overline{P}_c}(L)/ (\text{constants}).
\end{align*}
These are the \emph{generalised Steinberg representations} associated with the parabolics $B, Q_c$ and $P_c$ respectively. 
We also have locally analytic and continuous versions of all these representations, denoted $\mathrm{St}^{\mathrm{la}}_\star(L)$, $ \mathrm{St}^{\mathrm{cts}}_\star(L)$, etc.\ for $\star \in \{\varnothing, c\}$. Note that the action of $G(\Q_p)$ on the locally analytic versions of these representations is locally analytic. 

\subsubsection{Duals} We equip $\opn{St}^{\opn{sm}}(L)$ with the discrete topology; the space $\opn{St}^{\opn{cts}}(L)$ is naturally an admissible $L$-Banach representation with the supremum norm; and $\opn{St}^{\opn{la}}(L)$ is its subspace of locally analytic vectors, with its associated locally convex topology (in the sense of \cite{Eme17}). These topologies allow us to define, for $\bullet \in \{\opn{sm},\opn{cts},\opn{la}\}$, the continuous $L$-dual 
\[
\opn{St}^{\bullet}(L)^* \defeq \opn{Hom}_{L,\opn{cts}}(\opn{St}^{\bullet}(L),L).
\]
Note that in the case of $\bullet = \opn{sm}$ this is just the algebraic dual (as $L$-vector spaces).

\subsubsection{Extensions} From \cite[Thm. 2.15]{AutomorphicLinvariants} (c.f., \cite[Theorem 2.19]{DingLinvariants}) we have a canonical isomorphism
\begin{equation}\label{e:homs and extensions}
\mathrm{Hom}_{\mathrm{cts}}(\Q_p^{\times}, L) \isorightarrow  \mathrm{Ext}^1_{\mathrm{an}}(v_{c}^{\mathrm{sm}}(L), \mathrm{St}^{\mathrm{la}}(L)),
\end{equation}
where $\mathrm{Ext}^1_{\mathrm{an}}(v_{c}^{\mathrm{sm}}(L), \mathrm{St}^{\mathrm{la}}(L))$ is the group of locally analytic extensions of $v_{c}^{\mathrm{sm}}(L)$ by $\mathrm{St}^{\mathrm{la}}(L)$.

\subsection{Automorphic $\cL$-invariants} \label{sec:automorphic L-invariants}

We continue to work with level $K = K^p\opn{Iw} \subset \GL_n(\A_f)$, where $\opn{Iw}\subset \GL_n(\Zp)$ is the Iwahori subgroup  and $K^p= \prod_{\ell\neq p}K_\ell$. We may (and do) assume that $K_\ell = \GL_n(\Z_\ell)$ for almost all $\ell$. 

In our specific setting, we can make explicit the quantities $q$ and $\ell = \ell(\GL_n)$ from \S\ref{sec:set-up}. The locally symmetric space $Y(K)$ has dimension $n^2/2+ n/2-1$, and we have $q= \lfloor n^2/4\rfloor$ and $\ell= \lfloor (n-1)/2 \rfloor$. Then recall $\pi$ appears in the singular cohomology $\mathrm{H}^{j}(Y(K),\C)$ for $j \in \{q,q+1,\dots, q+ \ell\}$ (and similarly for the compactly supported cohomology).

\subsubsection{$p$-arithmetic cohomology}
\begin{definition}
Suppose $M$ is a $L[G(\Q)]$-module, and let $D_G$ denote the Steinberg module of $G$ (see for example \cite[Def. 3.21]{AutomorphicLinvariants}). Then:
\begin{itemize}
\item[(i)] Let $\cA_{G}(M)$ be the space of functions $f: \pi_0(G(\R))\times G(\bA_f^p)/K^p\rightarrow M$.
\item[(ii)] Let $\cA_{G,c}(M)= \mathrm{Hom}_{\Z}(D_G, \cA_{G}(M))$. 
\end{itemize}
The spaces $\cA_{G}(M)$ and $\cA_{G,c}(M)$ are left $L[G(\Q)]$-modules via
$$(\gamma\cdot f)([g_{\infty}], g^p)= \gamma\cdot f([\gamma^{-1}g_{\infty}], \gamma^{-1}g^p)\ \ \text{and} \ \ (\gamma\cdot h)(d)([g_{\infty}], g^p)= \gamma\cdot h(\gamma^{-1}d)([\gamma^{-1}g_{\infty}], \gamma^{-1}g^p)$$
for $f\in \cA_{G}(M)$, $h\in \cA_{G, c}(M)$, $\gamma\in G(\Q)$, $[g_{\infty}]\in \pi_0(G(\R))$, $g^p\in G(\bA_f^p)$ and $d\in D_G$.
\end{definition}

For a left $\opn{Iw}$-module $N$, let 
\[
\opn{Ind}_{\opn{Iw}}^{G(\Qp)} N \defeq \big\{f \colon G(\Qp) \to N : f(gk) = k^{-1}\cdot f(g) \ \ \forall g\in G(\Qp), k\in\opn{Iw}\big\}
\]
be its induction to $G(\Qp)$ (noting we make no topological assumption on $f$). If $R$ is any ring equipped with the trivial action of $\opn{Iw}$, then from \cite[after Cor.\ 3.4 and \S3.6]{AutomorphicLinvariants}, we have Hecke-equivariant isomorphisms\footnote{Gehrmann shows these isomorphisms hold when $K$ is neat. However given this, even if $K$ is not neat they hold more or less directly from our definitions (see Remark \ref{rem:neat}).}
\begin{align}
\mathrm{H}^j(G(\Q), \cA_{G}(\mathrm{Ind}_{\mathrm{Iw}}^{G(\Q_p)}R))&\cong \mathrm{H}^{j}(K, R), \label{eq:arithmetic to betti}\\
\mathrm{H}^j(G(\Q), \cA_{G, c}(\mathrm{Ind}_{\mathrm{Iw}}^{G(\Q_p)}R))&\cong \mathrm{H}^{j+ n-1}_c(K, R). \label{eq:arithmetic to betti c}
\end{align}

It will be important to work with the `$\pi^p$-part' of these groups. By \cite[Thm.\ C]{Jan18}, there exists a number field $\Q_\pi$ such that $\pi_f$ has a model over $\Q_\pi$.

\begin{definition}\label{def:pi-isotypic}
If $k$ is any field containing the field $\Q_\pi$ of definition of $\pi$, let 
\[
    \bT^{\opn{full}}(K^{p})_k \defeq C_{\opn{c}}(K^p\backslash \GL_n(\A_f^p)/K^p, k)
\]
be the (full) prime-to-$p$ Hecke algebra; note that, unlike $\bT^S(K^p)$ above, this includes contributions at $v \in S$, and is not in general commutative. If $H$ is a left $\bT^{\opn{full}}(K^{p})_k$-module, its $\pi^{p}$-isotypic part is
\[
H[\pi^{p}] \defeq \opn{Hom}_{\bT^{\opn{full}}(K^{p})_k}((\pi^{p}_f)^{K^{p}}, H).
\]

For later use, we also define the \emph{$\pi^p$-isotypic image} to be $H\{\pi^p\} \defeq \sum_{f \in H[\pi^p]} \opn{im}(f) \subset H$.
\end{definition}

By arguments as in Proposition \ref{prop:matsushima} (see \cite[Prop.\ 3.6]{AutomorphicLinvariants}), we know that for $0 \leq r \leq \ell$ and $(\varepsilon\colon K_\infty/K_\infty^\circ \to \{\pm 1\}) \in \cE_\pi$ (see Definition \ref{def:admissible signs}), we have
\[
\dim_L \mathrm{H}^{q+r}(G(\Q), \cA_{G, c}(\mathrm{Ind}_{\mathrm{Iw}}^{G(\Q_p)}L))^\varepsilon[\pi^p] = \big(\begin{smallmatrix}\ell \\ r \end{smallmatrix}\big),
\]
and that this space is 0 in other degrees (and for other signs).

As $\pi_p$ is Steinberg, we can `lift' cohomology to the Bruhat--Tits building, as in \cite[\S3.2]{AutomorphicLinvariants}. For this, observe that by \cite[above Cor.\ 2]{GK14}, the Iwahori-invariants $\opn{St}^{\opn{sm}}(L)^{\opn{Iw}}$ are a line. Then Frobenius reciprocity yields a $G(\Qp)$-equivariant map 
\[
    \opn{St}^{\opn{sm}}(L)^* \to \opn{Ind}_{\opn{Iw}}^{G(\Qp)} L,
\]
inducing a map
\[
\opn{ev}^{(j)} \colon \mathrm{H}^j(G(\Q), \cA_{G, c}(\opn{St}^{\opn{sm}}(L)^*)) \longrightarrow \mathrm{H}^j(G(\Q), \cA_{G, c}(\mathrm{Ind}_{\mathrm{Iw}}^{G(\Q_p)}L)).
\]

\begin{lemma}\label{lem:on the tree mult one}
For all $j$ the map $\opn{ev}^{(j)}$ induces an isomorphism
\[
\mathrm{H}^j(G(\Q), \cA_{G, c}(\opn{St}^{\opn{sm}}(L)^*))[\pi^p] \isorightarrow \mathrm{H}^j(G(\Q), \cA_{G, c}(\mathrm{Ind}_{\mathrm{Iw}}^{G(\Q_p)}L))[\pi^p].
\]
In particular, for any $\varepsilon \in \cE_\pi$, we have
\[
\dim_L \mathrm{H}^{q+r}(G(\Q), \cA_{G, c}(\opn{St}^{\opn{sm}}(R)^*))^\varepsilon[\pi^p] = \big(\begin{smallmatrix}\ell \\ r \end{smallmatrix}\big).
\]
\end{lemma}
\begin{proof}
This is \cite[Prop.\ 3.6]{AutomorphicLinvariants}.
\end{proof}

\begin{remark}
This is a cohomological analogue/generalisation of \cite[Lem.\ 1.3(3)]{Dar01}, which treats the case of $G = \GL_2$, and shows that $p$-new modular forms can be lifted to `forms on the Bruhat--Tits tree'. 
\end{remark}

\subsubsection{Automorphic $\cL$-invariants}

Let $\lambda: \Q_p^{\times}\rightarrow L$ be a continuous homomorphism and $c\in \{1, ..., n-1\}$. By \eqref{e:homs and extensions}, attached to $\lambda$ is an extension $\Delta_{\lambda, c}\in \mathrm{Ext}^1_{\mathrm{an}}(v_{c}^{\mathrm{sm}}(L), \mathrm{St}^{\mathrm{la}}(L))$. Forgetting topologies gives a map
\[
    \mathrm{Ext}^1_{\mathrm{an}}\Big(v_{c}^{\mathrm{sm}}(L), \mathrm{St}^{\mathrm{la}}(L)\Big)\longrightarrow \mathrm{Ext}^1_{L[G(\Q_p)]}\Big(v_{c}^{\mathrm{sm}}(L), \mathrm{St}^{\mathrm{la}}(L)\Big)\cong \mathrm{H}^{1}\Big(G(\Q_p), \mathrm{Hom}_L(v_{c}^{\mathrm{sm}}(L), \mathrm{St}^{\mathrm{la}}(L))\Big),
\]
and we can consider $\Delta_{\lambda, c}$ as an element of $\mathrm{H}^{1}(G(\Q_p), \mathrm{Hom}_L(v_{c}^{\mathrm{sm}}(L), \mathrm{St}^{\mathrm{la}}(L)))$. Via the natural map 
$$\cA_{G, c}(\mathrm{St}^{\mathrm{la}}(L)^{\ast})\otimes_L \mathrm{Hom}_L\Big(v_{c}^{\mathrm{sm}}(L), \mathrm{St}^{\mathrm{la}}(L)\Big)\longrightarrow \cA_{G, c}(v_{c}^{\mathrm{sm}}(L)^{\ast}),$$
the cup product with $\Delta_{\lambda, c}$ gives, for any $j \geq 0$, a prime-to-$p$ Hecke equivariant morphism
$$\Upsilon_{c, \lambda}^{j, \opn{la}}:  \mathrm{H}^{j}(G(\Q), \cA_{G, c}(\mathrm{St}^{\mathrm{la}}(L)^{\ast})) \rightarrow \mathrm{H}^{j+1}(G(\Q), \cA_{G, c}(v_{c}^{\mathrm{sm}}(L)^{\ast})).$$

By \cite[Cor. 2.19]{AutomorphicLinvariants} we have that $\mathrm{St}^{\mathrm{cts}}(L)$ is the unitary completion of $\mathrm{St}^{\mathrm{la}}(L)$; and then from \cite[Prop. 3.11]{AutomorphicLinvariants} we obtain a canonical isomorphism 
\[
    \mathrm{H}^{j}(G(\Q), \cA_{G, c}(\mathrm{St}^{\mathrm{cts}}(L)^{\ast})) \cong \mathrm{H}^{j}(G(\Q), \cA_{G, c}(\mathrm{St}^{\mathrm{\mathrm{sm}}}(L)^{\ast})).
    \]
    Then the inclusion $\mathrm{St}^{\mathrm{la}}(L)\subset \mathrm{St}^{\mathrm{cts}}(L)$ induces a natural map 
    \[
        \mathrm{H}^{j}(G(\Q), \cA_{G, c}(\mathrm{St}^{\mathrm{sm}}(L)^{\ast})) \rightarrow \mathrm{H}^{j}(G(\Q), \cA_{G, c}(\mathrm{St}^{\mathrm{la}}(L)^{\ast})).
    \]
    Composing with $\Upsilon_{c,\lambda}^{j, \opn{la}}$ we finally obtain a map
    \[
    \Upsilon_{c, \lambda}^{j}\colon \mathrm{H}^{j}(G(\Q), \cA_{G, c}(\mathrm{St}^{\mathrm{sm}}(L)^{\ast}))\longrightarrow \mathrm{H}^{j+1}(G(\Q), \cA_{G, c}(v_{c}^{\mathrm{sm}}(L)^{\ast})).
    \]

As $\Upsilon_{c, \lambda}^{j}$ is prime-to-$p$ Hecke-equivariant, it induces a map on the $\pi^p$-isotypic parts
\[
\Upsilon_{c, \lambda}^{j}[\pi^p]: \mathrm{H}^{i}(G(\Q), \cA_{G, c}(\mathrm{St}^{\mathrm{sm}}(L)^{\ast}))[\pi^{p}]\rightarrow \mathrm{H}^{j+1}(G(\Q), \cA_{G, c}(v_{c}^{\mathrm{sm}}(L)^{\ast}))[\pi^{p}].
\]
For any character $\varepsilon \in \cE_\pi$, we can perform similar constructions everywhere using the $\varepsilon$-parts of the cohomology, obtaining maps
\begin{equation}\label{eq:upsilon}
\Upsilon_{c, \lambda}^{j,\varepsilon}[\pi^p]: \mathrm{H}^{j}(G(\Q), \cA_{G, c}(\mathrm{St}^{\mathrm{sm}}(L)^{\ast}))^\varepsilon[\pi^{p}]\longrightarrow \mathrm{H}^{j+1}(G(\Q), \cA_{G, c}(v_{c}^{\mathrm{sm}}(L)^{\ast}))^\varepsilon[\pi^{p}].
\end{equation}

Recall that cuspidal cohomology is supported in degrees $q, q+1, \dots, q+\ell$, where $\ell = \ell(\GL_n)$. By \eqref{eq:arithmetic to betti c}, we see that the $\pi^p$-isotypic pieces of the $p$-arithmetic cohomology with compact support are supported in degrees $q', \dots, q'+\ell$, where $q' \defeq q-n+1$.

\begin{definition} \label{def:auto L-invariant}
For $c\in \{1,.., n-1\}$, $0\leq  r \leq \ell$, and $\varepsilon \in \cE_\pi$ an admissible sign, let $\bL_{c}^{r,\varepsilon}(\pi)$ be the subspace of $\mathrm{Hom}_{\mathrm{cts}}(\Q_p^{\times}, L)$ given by
\[
    \bL^{r, \varepsilon}_c(\pi) \defeq \{\lambda \in \mathrm{Hom}_{\mathrm{cts}}(\Q_p^{\times}, L): \ \Upsilon_{c, \lambda}^{q'+r, \varepsilon}[\pi^p]= 0\}.
\]
\end{definition}

\begin{proposition} \label{p: L-invariant} \cite[Prop.\ 3.14]{AutomorphicLinvariants}
Let $c \in \{1,\dots, n-1\}$ and $\varepsilon \in \cE_\pi$ an admissible sign.
\begin{itemize}
\item[(i)] For any choice of $0 \leq r \leq \ell$, we have $\mathrm{ord}_p\notin \bL^{r,\varepsilon}_c(\pi)$. In particular, the $L$-vector space $\bL^{r, \varepsilon}_c(\pi)$ has codimension at least 1 (equivalently, dimension at most 1).

\item[(ii)] The (bottom degree and top degree) spaces $\bL_{c}^{0,\varepsilon}(\pi)$ and $\bL_{c}^{\ell,\varepsilon}(\pi)$ have dimension exactly 1.
\end{itemize}
\end{proposition}
\begin{proof} 
For the singular cohomology, in \cite[Cor.\ 3.8, Rem.\ 2.16]{AutomorphicLinvariants} Gehrmann shows that if $\lambda$ is any non-zero \emph{smooth} character, then $\Upsilon^{q+r,\varepsilon}_{c,\lambda}[\pi^p]$ is an isomorphism, hence  $\lambda \not\in \bL_c^{r,\varepsilon}(\pi)$. Thus $\opn{ord}_p  \not\in \bL^{r,\varepsilon}_c(\pi)$. That these results carry over to the compactly supported cohomology is explained in \cite[\S3.6]{AutomorphicLinvariants}. The rest of (i) follows since $\opn{Hom}_{\opn{cts}}(\Qp^\times,L)$ is 2-dimensional, generated by $\opn{ord}_p$ and $\log_p$. 

To see (ii), it suffices to prove that $\bL_{c}^{r,\varepsilon}(\pi) \neq 0$ for $r = 0,\ell$. We note that in bottom degree $j = q'$ and top degree $j = q'+\ell$, the space $\mathrm{H}^{j}(G(\Q), \cA_{G, c}(\mathrm{St}^{\mathrm{sm}}(L)^{\ast}))^\varepsilon[\pi^{p}]$ is 1-dimensional (by Lemma \ref{lem:on the tree mult one}).  As $\Upsilon_{c,\opn{ord}_p}^{j,\varepsilon}[\pi^p]$ is an isomorphism, the target $\mathrm{H}^{j+1}(G(\Q), \cA_{G, c}(v_{c}^{\mathrm{sm}}(L)^{\ast}))^\varepsilon[\pi^{p}]$ is similarly a line. Since $\opn{Hom}_{\opn{cts}}(\Qp^\times,L)$ is 2-dimensional, it follows that for $r = 0$ and $r = \ell$, the cup product pairing
\[
\mathrm{H}^{j}(G(\Q), \cA_{G, c}(\mathrm{St}^{\mathrm{sm}}(L)^{\ast}))^\varepsilon[\pi^{p}] \times \opn{Hom}_{\opn{cts}}(\Qp^\times,L) \to  \mathrm{H}^{j+1}(G(\Q), \cA_{G, c}(v_{c}^{\mathrm{sm}}(L)^{\ast}))^\varepsilon[\pi^{p}]
\]
described above must degenerate, in that there exists $\lambda \neq 0$ such that $\langle \phi, \lambda\rangle = 0$ for all $\phi$. But then $\lambda \in \bL_{c}^{r,\varepsilon}(\pi)$ by definition.
\end{proof}

In \cite[Conj.\ A]{AutomorphicLinvariants}, Gehrmann makes the following further prediction. Note that part (i) is known for $r = 0$ or $r=\ell$ by Proposition \ref{p: L-invariant}. 

\begin{conjecture}\label{conj:L-invariant dim 1}
\begin{itemize}
\item[(i)] The space $\bL_{c}^{r,\varepsilon}(\pi)$ has codimension exactly 1 (equivalently, dimension exactly 1) for any $c \in \{1,\dots,n-1\}$, $0 \leq r \leq \ell$, and $\varepsilon \in \cE_\pi$.
\item[(ii)] The space $\bL_{c}^{r,\varepsilon}(\pi)$ is independent of both $r$ and $\varepsilon \in \cE_\pi$.
\end{itemize}
\end{conjecture}

Finally, we translate the space $\bL_{c}^{r,\varepsilon}(\pi)$ (which Gehrmann calls the automorphic $\mathcal{L}$-invariant) into an actual quantity, similar to descriptions of $\mathcal{L}$-invariants in, say, the case of modular forms. The following definition is conditional in general, but unconditional for $r = 0$ or $r=\ell$.

\begin{definition}  \label{def:automorphic number}
Let $c\in \{1,..., n-1\}$, $0 \leq r \leq \ell$ and $\varepsilon \in \cE_\pi$. Assuming Conjecture \ref{conj:L-invariant dim 1}(i) holds for these parameters, the \emph{automorphic $\cL$-invariant for $c, r,$ and $\varepsilon$}  is the unique quantity $\cL^{r,\varepsilon}_{\pi, c} \in L$ satisfying 
\[
    \mathrm{log}_p- \cL^{r,\varepsilon}_{\pi, c}\mathrm{ord}_p\in \bL^{r,\varepsilon}_c(\pi).
\]
\end{definition}

\subsection{Fontaine--Mazur $\cL$-invariants} Fontaine--Mazur $\cL$-invariants for $\pi$ are constructed via the Galois representation $\rho_{\pi}$. We impose additional assumptions:

\begin{hyp}\label{h: abs irred and decomposed generic} 
We assume that $p > 2n$ and  the residual representation $\overline{\rho}_\pi$ is absolutely irreducible and decomposed generic.
\end{hyp}

If $\omega$ denotes the $p$-adic cyclotomic character, then under the hypothesis above we deduce from \cite[Corollary 5.5.2]{10author} that we have
    \begin{equation}\label{eq:steinberg at p}
    \rho_{\pi}|_{G_{\Q_p}} \sim {\left(\begin{matrix} 1 & *_1 & \dots & * & * \\ & \omega^{-1} & & * & * \\ & & \ddots & *_{n-2} & \vdots \\ & & & \omega^{2-n} & *_{n-1} \\ & & & & \omega^{1-n}\end{matrix}\right)}.
    \end{equation}

\begin{definition} \label{def: non-crystalline at c} 
Let $c\in \{1, ..., n-1\}$. The local Galois representation $\rho_{\pi}\mid_{G_{\Q_p}}$ is called \emph{non-split at $c$} if $\ast_c \neq 0$. We say $*_c$ is non-crystalline if the corresponding extension is non-crystalline. 
\end{definition}

\begin{remark} \label{rem:essentially self-dual}
If $\pi$ is essentially self-dual, then -- independent of Hypothesis \ref{h: abs irred and decomposed generic} -- equation \eqref{eq:steinberg at p} holds, and $*_c$ is non-crystalline for all $c$. For this, it suffices to know that local-global compatibility holds for $\rho_\pi$ at $\ell= p$ (up to Frobenius-semisimplification, not just semisimplification), as the Weil--Deligne representation associated with the Steinberg representation under local Langlands is non-crystalline at all $c$. If $n$ is odd, it is proved in \cite{BGGTII} (noting that, in the language \emph{op.\ cit}., every weight is Shin-regular for odd $n$); and if $n$ is even, in \cite{CaraianiANT}.
\end{remark}

Let $c\in \{1, ..., n-1\}$ and suppose that $\rho_{\pi}\mid_{G_{\Q_p}}$ is non-split at $c$. We denote by
\[
    \bL^{\mathrm{FM}}_c(\pi)= \langle \ast_c \rangle^{\perp}\subset \mathrm{Hom}_{\mathrm{cts}}(\Q_p^{\times}, L)\cong \mathrm{H}^1(\Q_p, L),
\]
the orthogonal complement of $\langle \ast_c \rangle\subset \mathrm{H}^1(\Q_p, L(1))$ under local Tate duality. As $\ast_c$ is non-split by assumption, $\bL^{\mathrm{FM}}_c(\pi)$ is a $1$-dimensional $L$-vector space, and $\mathrm{ord}_p\notin \bL^{\mathrm{FM}}_c(\pi)$ if $*_c$ is non-crystalline (since, by Kummer theory, $\mbb{Z}_p^{\times} \otimes_{\mbb{Z}_p} L \cong \opn{H}^1_f(\mbb{Q}_p, L(1))$).

\begin{definition} \label{def:FM number}
Suppose $*_c$ is non-crystalline. The \emph{Fontaine-Mazur $\cL$-invariant} at $c$ is the unique element $\cL_{\pi, c}^{\mathrm{FM}}\in L$ such that $\mathrm{log}_p- \cL_{\pi, c}^{\mathrm{FM}}\mathrm{ord}_p \in \bL^{\mathrm{FM}}_c(\pi)$.
\end{definition}

\section{The automorphic Benois--Colmez--Greenberg--Stevens formula}\label{sec:automorphic BCGS}

Let $\pi$ be as in \S\ref{s:L-invariants}. We now make crucial use of Theorem \ref{thm:main etale}. To ensure the required cohomological-multiplicity-one assumption, we must be more precise about the prime-to-$p$ level. Specifically, in the rest of the paper we will take Whittaker new level
\begin{equation}\label{eq:whittaker new}
    K^p = K_1^p(N)  \defeq \{\smallmatrd{a}{b}{c}{d} \in \GL_n(\A_f^p) : c \equiv 0 \newmod{N}, d \equiv 1 \newmod{N}\},
\end{equation}
where $N$ is the smallest (prime-to-$p$) positive integer such that $\pi_f^p$ admits $K_1^p(N)$-invariants. Here $\smallmatrd{a}{b}{c}{d}$ is a block matrix with bottom row $(c,d)$, where $c$ is a $1\times (n-1)$ vector and $d$ is a scalar.

Let $S$ denote a finite set of places containing $p$ and all primes $\ell$ where $K_\ell \neq \GL_n(\Z_\ell)$. With this $S$, we take other notation from \S\ref{sec:set-up}; so $\bT(K) = \bT^S(K^p) \otimes \sA_p^+$ is the Hecke algebra of level $K$, and 
\[
    \phi \colon \bT(K)_L \longrightarrow L
\]
is the classical cuspidal eigenpacket arising from $\pi_f^{K}$. The $p$-part $\phi_p = \phi|_{\sA_p^+}$ is trivial, since $\pi_p$ is Steinberg; so $\phi$ is $p$-ordinary, hence non-critical. Moreover, by \cite[\S5]{JPSS}, and since the Iwahori-invariants of the Steinberg representation are 1-dimensional, we have $\opn{dim}_{\C} \pi_f^{K} = 1$; so by Example \ref{ex:mult one}, $\phi$ is of cohomological multiplicity one at level $K$ at any admissible sign $\varepsilon \in \cE_\pi$.

\begin{assumption}\label{ass:nap}
We assume throughout that the non-abelian Leopoldt conjecture (Conjecture \ref{conj:non-abelian leopoldt}) holds for $\phi$. Note that this is automatic for $\GL_n$ with $n \leq 4$, since then $\ell(\GL_n) \leq 1$ (see Remark \ref{rem:non-abelian leopoldt}).
\end{assumption}

By the assumption, for any $\varepsilon \in \cE_\pi$ we may apply Theorem \ref{thm:main etale} at the weight $\lambda = 1$, yielding a rigid locally Zariski-closed subspace\footnote{Note that \emph{a priori} $\Sigma$ does depend on $\varepsilon$, though we do not encode this in notation.} $\Sigma = \Sigma^\varepsilon \subset \crW$ such that $\hc{q+\ell}(K, \sD_{\Sigma})^{\leq h, \varepsilon}_\phi$ is free of rank one over $\cO(\Sigma)_1$. Up to shrinking $\Sigma$ there is a connected component $\sC \subset \sX_{G,K}^\varepsilon$ such that
\begin{equation}\label{eq:family over Sigma}
\hc{q+\ell}(K, \sD_{\Sigma})^{\leq h, \varepsilon} \otimes_{\bT_{\Sigma,h}^\varepsilon} \cO(\sC) \text{ is free of rank one over }\cO(\Sigma).
\end{equation}
The Hecke algebra acts on this space by a character $\phi_\Sigma = \phi^S_\Sigma \otimes \phi_{\Sigma,p}: \bT(K)\rightarrow \cO(\Sigma)$. By construction, if $\m_1 \subset \cO(\Sigma)$ is the maximal ideal corresponding to $\lambda = 1$, then the classical eigenpacket $\phi$ factors as
\[
\phi \colon \bT(K) \xrightarrow{\ \phi_{\Sigma} \ } \cO(\Sigma) \xrightarrow{ \ \newmod{\m_1} \ } L.
\] 
The titular \emph{automorphic Benois--Colmez--Greenberg--Stevens formula} studies the infinitesimal behaviour of $\phi_\Sigma$ along a `good' tangent vector in $\Sigma$ at $\lambda = 1$. To prove it, we use fundamentally the strategy of Gehrmann--Rosso from \cite{GehrmannRosso}, via the Koszul complex with overconvergent cohomology from \cite{KS12}.

\medskip

For clarity in later sections, we summarise our complete running set-up in the following.

\begin{set-up}\label{setup}
let $\pi$ be a $p$-ordinary RACAR of $\GL_n(\A)$ such that $\pi_p$ is the Steinberg representation of $\GL_n(\Qp)$. (In particular $\pi$ has trivial cohomological weight). Let $\phi$ be the associated eigenpacket $\phi \colon \bT(K) \to L$. Assume that the description in \eqref{eq:steinberg at p} holds. From \S\ref{s:L-invariants}, for $1 \leq c\leq n-1$, we have:
\begin{itemize}
\item automorphic $\cL$-invariants $\bL_{c}^{r,\varepsilon}(\pi)$ for $0 \leq r \leq \ell(\GL_n) = \lfloor \tfrac{n-1}{2}\rfloor$ and any admissible sign $\varepsilon \in \cE_\pi$,
\item and if $*_c$ is non-split in \eqref{eq:steinberg at p}, we have a Fontaine--Mazur $\cL$-invariant $\bL_c^{\opn{FM}}(\pi)$. 
\end{itemize}
Both are subspaces of the 2-dimensional $L$-vector space $\opn{Hom}_{\opn{cts}}(\Qp^\times,L)$ of dimension at most 1.
\end{set-up}

\subsection{Weights, families and tangent vectors}\label{sec:infinitesimal weights}
\subsubsection{Infinitesimal weights}
For $\GL_n$, the weight space $\sW$ is the $\Qp$-rigid space whose $L$-points parametrise continuous characters $T(\Zp) \to L^\times$; that is, tuples $\kappa = (\kappa_1,\dots, \kappa_n)$, where each $\kappa_i$ is a continuous character $\Zp^\times \to L^\times$. It decomposes as the disjoint union of $n$-dimensional open discs. The connected component $\sW^\circ \subset \sW$ containing $\lambda = 1$ is the rigid generic fibre of $\opn{Spf}(\Zp[\![X_1, \dots, X_n]\!])$. It comes equipped with a universal character 
\[
    \kappa^{\opn{univ}} = (\kappa^{\opn{univ}}_1, \dots, \kappa^{\opn{univ}}_n) : T(\Zp) = (\Zp^\times)^n \to \Zp[\![X_1,\dots,X_n]\!]^\times,
    \]
    where $\kappa_i^{\opn{univ}}(x) = (1+X_i)^{\log_p(x)}$ for any $x \in \Zp^\times$. For any Zariski-closed affinoid $\Lambda \subset \sW^\circ$ contained in this component, let 
    \[
    \kappa_\Lambda^{\opn{univ}} \colon  T(\Zp) \to \cO(\Lambda)^\times
    \]
    be the composition of $\kappa^{\opn{univ}}$ with the natural map $\Zp[\![ X_1,\dots,X_n]\!] \to \cO(\Lambda)$.

    \medskip

    By construction, we have an inclusion $\Sigma \subset \sW_L$. After base-changing to $L$, we have the tangent space $\opn{T}_1(\sW_L) \cong L^n$  at the trivial character, spanned linearly by $X_1, \dots, X_n$. We thus have $\opn{T}_1(\Sigma) \subset L^n$.

    Any $v = (v_1,\dots,v_n) \in \opn{T}_1(\Sigma) \subset \opn{T}_1(\sW_L)$ corresponds to a morphism $\opn{Spec}(L[\epsilon]) \to \Sigma \subset \sW_L$, corresponding to a map
     \begin{equation}\label{eq:spec to v}
        \Zp[\![X_1,\dots,X_n]\!] \to \cO(\Sigma) \to L[\epsilon], \qquad X_i \longmapsto v_i\epsilon \ \ \text{ (for }i = 1,\dots, n)
    \end{equation}
    that factors through $\cO(\Sigma)$.    Note in particular that $X_i^2 \mapsto 0$ for all $i$. Composing with $\kappa^{\opn{univ}}$, we get a `universal' infinitesimal character in the direction of $v$, namely 
    \[
        \kappa_v^{\opn{univ}} = (\kappa_{v,1}^{\opn{univ}}, \dots, \kappa_{v,n}^{\opn{univ}}) \colon (\Zp^\times)^n \longrightarrow L[\epsilon]^n,
    \]
    where 
    \begin{equation}\label{eq:universal infinitesimal}
        \kappa_{v,i}^{\opn{univ}}(x) = 1 + \log_p(x) v_i\epsilon \qquad \forall x \in \Zp^\times.
    \end{equation}
    Note in particular that $\kappa_v^{\opn{univ}} \equiv 1 \newmod{\epsilon}$.

\subsubsection{Infinitesimal families}\label{ss: Infinitesimal families}
Recall the eigenpacket $\phi$ attached to $\pi$, and the eigenpacket 
\[
    \phi_\Sigma = \phi_{\Sigma}^S \otimes \phi_{\Sigma,p} \colon \bT(K) = \bT^S(K) \otimes \sA_p^+ \longrightarrow \cO(\Sigma)
\]
for the family over $\Sigma$ from \eqref{eq:family over Sigma}.

Recall $\sA_p^+$ can be identified with the free commutative $\Qp$-algebra generated by $t_1, \dots, t_n$, where 
\[
    t_i = \opn{diag}(p,\dots,p,1,\dots,1) \in T(\Qp), \qquad \text{with $i$ lots of $p$}.
\]
Since the elements $t_i$ (together with their inverses) generate $T(\Qp)/T(\Zp)$, the character $\phi_{\Sigma,p} \colon \sA_p^+ \to \cO(\Sigma)^\times$ naturally defines a character $\Phi_{\Sigma,p} \colon T(\Qp)/T(\Zp) \to \cO(\Sigma)^\times$. We freely lift this to a character on $T(\Qp)$, and write $\Phi_{\Sigma,p} = (\Phi_{\Sigma,p,1},\dots, \Phi_{\Sigma,p,n})$, where each $\Phi_{\Sigma,p,i}$ is a character of $\Qp^\times$ (trivial on $\Zp^\times$).

\medskip

If $v \in \opn{T}_1(\Sigma)$, then composing $\phi_\Sigma$ with the map $\cO(\Sigma) \to L[\epsilon]$ at $v$ from \eqref{eq:spec to v}, we obtain an infinitesimal eigenpacket
\[
\phi_v = \phi_v^S \otimes \phi_{v,p} \colon \bT(K) \longrightarrow L[\epsilon]
\]
in the direction of $v$. Note that $\phi_v \equiv \phi \newmod{\epsilon}$. In particular $\phi_{v,p} \equiv 1 \newmod{\epsilon}$.

Specialising $\Phi_{\Sigma,p}$ along the map $\cO(\Sigma) \to L[\epsilon]$ yields a character $\Phi_{v,p} = (\Phi_{v,p,1},\dots, \Phi_{v,p,n}) \colon T(\Qp) \to L[\epsilon]^\times$, where each $\Phi_{v,p,i}$ is a character of $\Qp^\times$ (trivial on $\Zp^\times$).

\begin{lemma}\label{lem:centre trivial 2}
\begin{itemize}
\item[(i)] The character $\Phi_{\Sigma,p}$ is trivial on $Z_{\GL_n}(\Qp)$. Moreover, $\Phi_{\Sigma,p} \equiv 1 \newmod{\m_1}$, where $\m_1 \subset \cO(\Sigma)$ is the maximal ideal corresponding to the trivial character $1 \in \Sigma$.
\item[(ii)] For any $v \in \opn{T}_1(\Sigma)$, the character $\Phi_{v,p}$ is trivial on $Z_{\GL_n}(\Qp)$. Moreover $\Phi_{v,p} \equiv 1 \newmod{\epsilon}$.
\end{itemize}
\end{lemma}

\begin{proof}
(i) For any $z \in Z_{\GL_n}(\Qp)$, the operator $U_z = [\opn{Iw} z \opn{Iw}] \in \sA_p^+$ acts on the overconvergent cohomology as $z$ itself. By construction, it also acts as $\Phi_{\Sigma,p}(z)$. The first claim follows since $z$ acts trivially on overconvergent cohomology (of any weight). Indeed, the action of $z$ on $\sD_\Sigma$ is induced from conjugation on $N(\Zp)$; and as $z$ is central, it thus acts trivially. The second claim follows since $\phi_{\Sigma,p} \equiv \phi_p = 1\newmod{\m_1}$.

Part (ii) follows specialising (i) along $v$, or directly noting $\phi_{v,p}$ (hence $\Phi_{v,p}$) is trivial modulo $\epsilon$.
\end{proof}

\subsubsection{Characters on $T(\Qp)$} \label{sec:chars on T}
We now draw all of the above together as preparation for the next subsection. The choice of uniformiser $p \in \Zp$ defines a splitting 
\[
T(\Qp) \cong T(\Zp) \times T(\Qp)/T(\Zp).
\]
In particular, we may define a character
\begin{equation}\label{eq:chisigma}
  \chi_\Sigma \colon T(\Qp) \longrightarrow \cO(\Sigma)^\times, \qquad \chi_\Sigma|_{T(\Zp)} = \kappa_\Sigma^{\opn{univ}}, \ \ \chi_\Sigma|_{T(\Qp)/T(\Zp)} = \Phi_{\Sigma,p}.
\end{equation}
Similarly, for any $v \in \opn{T}_1(\Sigma)$ we may define a character
\begin{equation}\label{eq:chiv}
    \chi_v \colon T(\Qp) \longrightarrow L[\epsilon]^\times, \qquad \chi_v|_{T(\Zp)} = \kappa_v^{\opn{univ}}, \ \ \chi_v|_{T(\Qp)/T(\Zp)} = \Phi_{v,p}.
\end{equation}
\begin{lemma}\label{lem:centre trivial 3}
We have $\chi_\Sigma \equiv 1 \newmod{\m_1}$, and $\chi_v \equiv 1 \newmod{\epsilon}$.
\end{lemma}
\begin{proof}
Immediate from the fact $\kappa_\Sigma^{\opn{univ}} \equiv 1 \newmod{\m_1}$, from \eqref{eq:universal infinitesimal}, and by Lemma \ref{lem:centre trivial 2}.
\end{proof}

It will be  convenient to have analogues of these characters for more general affinoids. We put ourselves in the set-up of \S\ref{sec:set-up}-\ref{sec:proof}. Let $\Omega \subset \sW^\circ$ be an open affinoid, and suppose $\Sigma = \opn{Sp}(\cO(\Omega)/J) \subset \Omega$, with $J \subset \cO(\Omega)$ an ideal. Indeed, this is how $\Sigma$ is constructed in the proof of Theorem \ref{thm:main etale} (see \eqref{eq:sigma}).

Now suppose $\Sigma \subset \Lambda \subset \Omega$ is any Zariski-closed affinoid. Choose lifts $\Phi_{\Lambda,p}(t_i) \in \cO(\Lambda)$ of each $\Phi_{\Sigma,p}(t_i) \in \cO(\Sigma)^\times$. Since the vanishing locus of each $\Phi_{\Lambda,p}(t_i)$ is closed and disjoint from $\Sigma$, up to shrinking $\Omega$ to avoid the zeros of $\Phi_{\Lambda,p}(t_1),\dots,\Phi_{\Lambda,p}(t_{n-1})$, we may assume each $\Phi_{\Lambda,p}(t_i) \in \cO(\Lambda)^\times$. These values thus determine a lift $\Phi_{\Lambda,p} \colon T(\Qp)/T(\Zp) \to \cO(\Lambda)^\times$ of $\Phi_{\Sigma,p}$. (We emphasise this lift is not unique).

We can then define a character 
\begin{equation}\label{eq:chilambda}
    \chi_\Lambda \colon T(\Qp) \longrightarrow \cO(\Lambda)^\times, \qquad \chi_\Lambda|_{T(\Zp)} = \kappa_\Lambda^{\opn{univ}}, \ \ \chi_\Lambda|_{T(\Qp)/T(\Zp)} = \Phi_{\Lambda,p}.
\end{equation}
This lifts $\chi_\Sigma$ under the specialisation $\cO(\Lambda) \to \cO(\Sigma)$. 

Summarising, for any $1 \in \Sigma \subset \Lambda \subset \Omega$ and $v \in \opn{T}_1(\Sigma)$, the following diagram commutes:

\[
\xymatrix{
T(\Qp) \ar[d]_{\chi_\Lambda} \ar[drr]_-{\chi_\Sigma} \ar[drrrr]_{\chi_v} \ar[drrrrrr]^{1} &&&&&&\\ 
\cO(\Lambda)^\times \ar[rr] &&\cO(\Sigma)^\times \ar[rr] && L[\epsilon]^\times \ar[rr] && L^\times,}
\]
where the horizontal arrows are the natural maps, and the $\chi$ maps were defined in \eqref{eq:chisigma}, \eqref{eq:chiv}, and \eqref{eq:chilambda}.

\subsection{The automorphic BCGS formula and Gehrmann--Rosso's criterion} 
We now  describe what the automorphic Benois--Colmez--Greenberg--Stevens (BCGS) formula actually is, and state a key criterion of Gehrmann--Rosso for it to hold. We maintain running notation from earlier in \S\ref{sec:automorphic BCGS}.

\subsubsection{The automorphic BCGS formula}\label{sec:auto BCGS}

Let $v \in \opn{T}_1(\Sigma)$ be a tangent vector, with associated character $\chi_v \colon T(\Qp) \to L[\epsilon]^\times$ from \eqref{eq:chiv}. By Lemma \ref{lem:centre trivial 3} (and \cite[\S2.2]{GehrmannRosso}), for each $1 \leq i \leq n$ there exists a character $\partial\chi_{v,i} \in \opn{Hom}_{\opn{cts}}(\Qp^\times,L)$ such that
\begin{equation}\label{eq:partial chi_i}
\chi_{v,i}(x) = 1 + \partial\chi_{v,i}(x)\epsilon \qquad \forall x \in \Qp^\times.
\end{equation}

For any $1 \leq c \leq n-1$, let 
\begin{align*}
\alpha_c^\vee \colon \Qp^\times &\longrightarrow T(\Qp)\\
x &\longmapsto \opn{diag}(1,\dots,1,x,x^{-1},1,\dots,1)
\end{align*}
be the associated coroot, with $x$ in the $c$-th entry and $x^{-1}$ in the $(c+1)$-st entry. Then $\chi_v \circ \alpha_c^\vee$ is a character $\Qp^\times \to L[\epsilon]^\times$. From \eqref{eq:partial chi_i}, we see that
\[
\chi_v \circ \alpha_c^\vee(x) = 1 + \Big(\partial\chi_{v,c}(x) - \partial\chi_{v,c+1}(x)\Big)\epsilon \qquad \forall x \in \Qp^\times.
\]
Let 
\begin{equation}\label{eq:partial-chi}
\partial\chi_{v}[c] \defeq \partial\chi_{v,c} - \partial\chi_{v,c+1} \in \opn{Hom}_{\opn{cts}}(\Qp^\times,L).
\end{equation}

For $v \in \opn{T}_1(\Sigma)$, $1 \leq c \leq n-1$, $0 \leq r \leq \ell$, and $\varepsilon : K_\infty/K_\infty^\circ \to \{\pm1\}$ an admissible character, we say the \emph{automorphic BCGS formula} holds for $v$, $c$, $r$ and $\varepsilon$ if 
\[
    \partial \chi_{v}[c] \in \bL^{r,\varepsilon}_c(\pi),
\]
where $\bL^{r,\varepsilon}_c(\pi)$ is the automorphic $\cL$-invariant from Definition \ref{def:auto L-invariant}.

\subsubsection{Gehrmann--Rosso's criterion}
Let 
\[
\chi: \overline{B}(\Q_p)\longrightarrow T(\Q_p)\longrightarrow L[\epsilon]^\times
\]
be a locally analytic character that is trivial modulo $\epsilon$. Let
\begin{equation}\label{eq:Ila}
\mathrm{I}_{\overline{B}}^{\mathrm{la}}(\chi)= \{f\colon G(\Q_p)\rightarrow L[\epsilon] \ \text{locally analytic} : \ f(bg)= \chi(b)f(g) \ \text{for} \ b\in \overline{B}(\Q_p), g\in G(\Q_p) \},
\end{equation}
and similarly let $\opn{I}_{\overline{B}}^{\opn{la}}(1)$ be the analogous space modulo $\epsilon$ (i.e., in the definition we replace $L[\epsilon]$ with $L$, and $\chi$ with the trivial character). As usual, these spaces carry a $G(\Qp)$-action by right-translation. 

Let $\opn{I}_{\overline{B}}^{\opn{la}}(\chi)^*$ denote the continuous $L[\epsilon]$-dual and $\opn{I}_{\overline{B}}^{\opn{la}}(1)^*$ the continuous $L$-dual. Reduction modulo $\epsilon$ yields a map $\mathrm{I}_{\overline{B}}^{\mathrm{la}}(\chi)^* \to \mathrm{I}_{\overline{B}}^{\mathrm{la}}(1)^*$ of $G(\Qp)$-modules, inducing, for any $d \geq 0$ and sign $\varepsilon$, a map
\[
    \opn{red}^{d,\varepsilon}_\chi \colon \mathrm{H}^{d}\Big(G(\Q), \cA_{G, c}(\mathrm{I}_{\overline{B}}^{\mathrm{la}}(\chi)^{\ast})\Big)^\varepsilon \longrightarrow \mathrm{H}^{d}\Big(G(\Q), \cA_{G, c}(\mathrm{I}_{\overline{B}}^{\mathrm{la}}(1)^{\ast})\Big)^\varepsilon.
\]
These cohomology groups are modules over the full Hecke algebra $\bT^{\opn{full}}(K^p)_L$. 

Now let $v \in \opn{T}_1(\Sigma)$ be a tangent vector, and take $\chi = \chi_v \colon \overline{B}(\Qp) \to T(\Qp) \to L[\epsilon]^\times$ the associated character defined in \eqref{eq:chiv}. This is trivial $\newmod{\epsilon}$ by Lemma \ref{lem:centre trivial 3}. For any $1 \leq c \leq n-1$, let $\partial\chi_{v}[c] \in \opn{Hom}_{\opn{cts}}(\Qp^\times,L)$ be the associated additive character from \S\ref{sec:auto BCGS}.

The following is \cite[Lem.\ 2.2]{GehrmannRosso}. They treat cohomology without compact support, but the compact support case follows identically, following the modifications in \cite[\S3.6]{AutomorphicLinvariants}. Let $q' \defeq q-n+1$, and recall that $q' \leq d \leq q'+\ell$ is the range of degrees such that $\pi$ contributes to compactly-supported $p$-arithmetic cohomology. For a $\bT^{\opn{full}}(K^p)_L$-module $H$, recall the $\pi^p$-isotypic image $H\{\pi^p\}$ from Definition \ref{def:pi-isotypic}.

\begin{proposition}\label{prop:GR}
For $0 \leq r \leq \ell$, suppose that the image of 
\[
    \mathrm{red}_{\chi_v}^{q'+r, \varepsilon}\colon  \mathrm{H}^{q'+r}\Big(G(\Q), \cA_{G, c}(\mathrm{I}_{\overline{B}}^{\mathrm{la}}(\chi_v)^{\ast})\Big)^\varepsilon \longrightarrow \mathrm{H}^{q'+r}\Big(G(\Q), \cA_{G, c}(\mathrm{I}_{\overline{B}}^{\mathrm{la}}(1)^{\ast})\Big)^\varepsilon
\]
contains the $\pi^p$-isotypic image $\mathrm{H}^{q'+r}(G(\Q), \cA_{G, c}(\mathrm{I}_{\overline{B}}^{\mathrm{la}}(1)^{\ast}))^\varepsilon\{\pi^p\}$. Then the automorphic BCGS formula holds for $v, \varepsilon, r$, and all $c$; that is,
\[
    \partial\chi_{v}[c] \in \bL_{c}^{r,\varepsilon}(\pi) \qquad \text{ for all $1\leq c \leq n-1$}.
\]
In this case, if $\partial \chi_v[c] \neq 0$, then it is a basis of $\bL_c^{r,\varepsilon}(\pi)$.
\end{proposition}
\begin{proof}
All but the last part is \cite{GehrmannRosso}. The last part follows as $\dim_{L} \bL_{c}^{r,\varepsilon}(\pi) \leq 1$ (see Proposition \ref{p: L-invariant}).

For completeness, we give a rough sketch of the proof. Let $\lambda = \partial\chi_v[c] \in \opn{Hom}_{\opn{cts}}(\Qp^\times,L)$. Via an explicit process analogous to \eqref{e:homs and extensions}, one interprets $\opn{I}_{\overline{B}}^{\opn{la}}(\chi_v)$ as an extension of $I_{\overline{B}}^{\opn{la}}(1)$. (Note that the locally analytic Steinberg is a quotient of the latter, relating this to \eqref{e:homs and extensions}.) One gets from this a long exact sequence of cohomology, and one essentially has $\lambda \in \bL_{c}^{r,\varepsilon}(\pi)$ if the $\pi^p$-isotypic image is contained in the kernel of the connecting map from degree $q'+r$ to $q'+r+1$. Now, via some commutative algebra, the image of the map $\opn{red}^{q'+r,\varepsilon}_{\chi_v}$ can be identified with the image of the previous map in the long exact sequence, hence the kernel of the connecting map. With some further massaging, the result follows.
\end{proof}

In practice, we verify a stronger (but simpler to check) local condition. Consider the prime-to-$S$ eigenpacket $\phi^S$ associated with $\pi$. If $H$ is a $\bT^
S(K^p)_L$-module, let $H_{\phi^S}$ denote its localisation at $\ker(\phi^S)$.

\begin{corollary}\label{cor:GR localisation}
With notation as above, suppose that the localisation
\[
    \mathrm{red}_{\chi_v, \phi^S}^{q'+r, \varepsilon}\colon  \mathrm{H}^{q'+r}\Big(G(\Q), \cA_{G, c}(\mathrm{I}_{\overline{B}}^{\mathrm{la}}(\chi_v)^{\ast})\Big)^\varepsilon_{\phi^S}\longrightarrow \mathrm{H}^{q'+r}\Big(G(\Q), \cA_{G, c}(\mathrm{I}_{\overline{B}}^{\mathrm{la}}(1)^{\ast})\Big)^\varepsilon_{\phi^S}
\]
of $\mathrm{red}_{\chi_v}^{q'+r,\varepsilon}$ is surjective. Then $\partial\chi_{v}[c] \in \bL_{c}^{r,\varepsilon}(\pi)$ for all $1\leq c \leq n-1$.
\end{corollary}

\begin{proof}

To ease notation, for $\chi = \chi_v$ or $1$, let 
\[
    \h^{q'+r,\varepsilon}_v(\chi) \defeq \mathrm{H}^{q'+r}(G(\Q), \cA_{G, c}(\mathrm{I}_{\overline{B}}^{\mathrm{la}}(\chi)^{\ast}))^\varepsilon.
\]

We first translate the condition in Proposition \ref{prop:GR} from a prime-to-$p$ to a prime-to-$S$ condition. If $H$ is a $\bT^{\opn{full}}(K^p)_L$-module, define its \emph{$\pi^S$-isotypic part} to be 
\[
    H[\pi^S] \defeq \opn{Hom}_{\bT^S(K)_L}((\pi_f^p)^{K^p}, H),
\]
and its \emph{$\pi^S$-isotypic image} to be $H\{\pi^S\}\defeq \sum_{f \in H[\pi^S]} \opn{im}(f) \subset H$. 

As any $\bT^{\opn{full}}(K^p)$-equivariant map is $\bT^S$-equivariant, we see (for any $H$) that $H[\pi^p] \subset H[\pi^S]$, hence $H\{\pi^p\} \subset H\{\pi^S\}$. Applying to $H = \h^{q'+r,\varepsilon}_v(1)$, an immediate corollary of Proposition \ref{prop:GR} is:
\begin{quote}
$(\ddagger)$ \ \ \emph{If $\opn{im}(\opn{red}_{\chi_v}^{q'+r,\varepsilon})$ contains the $\pi^S$-isotypic image, then $\partial\chi_{v}[c] \in \bL_{c}^{r,\varepsilon}(\pi)$ for all $c$.}
\end{quote}

Recall that for all $\ell \not\in S$,  $\pi_\ell$ is unramified and $K_\ell = \GL_n(\Z_\ell)$. In particular every element of $(\pi_f^p)^{K^p}$ is a $\bT^S(K^p)_L$-eigenvector with eigenvalues given by $\phi^S$. From the definition, for any $H$, we see that the $S$-isotypic image $H\{\pi^S\}$ is contained within the $\phi^S$-eigenspace for $\bT^S(K^p)_L$. From $(\ddagger)$, it thus suffices to show that $\opn{red}_{\chi_v}^{q'+r,\varepsilon}$ contains the $[\bT^S(K^p)=\phi^S]$-eigenspace in $\h^{q'+r,\varepsilon}_v(1)$.

Now we consider localisations. As $\ker(\phi^S)$ is prime, the $[\bT^S(K^p)=\phi^S]$-eigenspace in $\h^{q'+r,\varepsilon}_v(1)$ injects naturally and Hecke-equivariantly into the localisation $\h^{q'+r,\varepsilon}_v(1)_{\phi^S}$;  we thus consider it a submodule. 

Suppose the condition of the corollary holds true, so $\opn{red}^{q'+r,\varepsilon}_{\chi_v,\phi^S}$ is surjective. If 
\[
    \varphi \in \h^{q'+r,\varepsilon}_v(1)[\bT^S(K^p)=\phi^S] \subset \h^{q'+r,\varepsilon}_v(1)_{\phi^S}
\]
is a $\phi^S$-eigenvector, then by surjectivity there exists $r \in \bT^S(K^p)_L\backslash \ker(\phi^S)$ and $\Phi \in \h^{q'+r,\varepsilon}_v(\chi_v)$ such that
\[
    \opn{red}^{q'+r,\varepsilon}_{\chi_v,\phi^S}(r^{-1} \Phi) = \varphi.
\]
Since $r \not\in \ker(\phi^S)$, we have $\phi^S(r) \in L\backslash\{0\}$. Since the reduction map is prime-to-$S$ Hecke-equivariant, we see that $\phi^S(r)^{-1}\Phi \in \h^{q'+r,\varepsilon}_v(\chi_v)$ is a global vector lifting $\varphi$ under $\opn{red}^{q'+r,\varepsilon}_{\chi_v}$. In particular, this (global) reduction map surjects onto the $[\bT^S(K^p)=\phi^S]$-eigenspace. By the arguments above, its image thus contains the $\pi^S$-isotypic image, and we conclude by $(\ddagger)$.
\end{proof}

\subsection{Koszul resolutions and overconvergent cohomology} \label{ss: Koszul resolutions and overconvergent cohomology}

We now verify the criterion of Gehrmann--Rosso in some special cases. For this, we use the Koszul complex in families. Setting up notation:
\begin{itemize}
\item Let $R$ be an $L$-affinoid algebra. We allow $R=L$.
\item Let $\chi: \overline{B}(\Q_p)\rightarrow T(\Q_p)\rightarrow R^{\times}$ be a locally analytic character. Let $\opn{I}_{\overline{B}}^{\opn{la}}(\chi)$ be as in \eqref{eq:Ila} (with $L[\epsilon]$ replaced with $R$), and $\opn{I}_{\overline{B}}^{\opn{la}}(\chi)^*$ its continuous $R$-dual.
\item Let $\chi_0: \mathrm{Iw}\cap \overline{B}(\Q_p)\rightarrow T(\Zp) \to R^{\times}$ be the restriction of $\chi$.
\item As in \S\ref{sec:set-up}, we recall from \cite[\S2.2]{Han17}\footnote{To be precise, we work with an opposite convention to Hansen. He considers locally analytic inductions from the Borel. We instead use induction from the opposite Borel, as in \cite[Def.\ 3.5]{BW20}, giving cleaner comparison with $\opn{I}_{\overline{B}}^{\opn{la}}(\chi)$.}
 the $R$-module $\crA_{\chi_0}$ of locally analytic functions of weight $\chi_0$ on $\mathrm{Iw}$ (that is, the locally analytic induction of $\chi_0$ from $\overline{B}(\Zp)\cap \opn{Iw}$ to $\opn{Iw}$).  We also have its continuous $R$-dual $\sD_{\chi_0}$, the $R$-module of locally analytic distributions of weight $\chi_0$ on $\opn{Iw}$.
 \item Let $r[\chi_0]$ be the smallest integer $r$ such that $\chi$ is $r$-analytic. As in \cite[\S2.2]{Han17}, for any $r \geq r[\chi_0]$, we denote by $\bfA^{r}_{\chi_0}$ and $\bfD^{r}_{\chi_0}$ the $r$-analytic versions of $\sA_{\chi_0}$ and $\sD_{\chi_0}$; thus $\sA_{\chi_0} = \cup_{r\geq r[\chi_0]} \bfA_{\chi_0}^{r}$, whilst $\sD_{\chi_0} = \cap_{r\geq r[\chi_0]}\bfD_{\chi_0}^{r}$.
 \item If $N$ is an $R[\opn{Iw}]$-module, let
\[
\cInd_{\opn{Iw}}^{G(\Qp)} N \defeq \big\{f \colon G(\Qp) \to N : f\text{ has compact support}, f(gk) = k^{-1}\cdot f(g) \ \ \forall g\in G(\Qp),k\in\opn{Iw}\big\}
\]
denote its compact induction to an $R[G(\Qp)]$-module. 
\end{itemize}

The space $\crA_{\chi_0}$ can naturally be identified with the subspace of functions in $\mathrm{I}_{\overline{B}}^{\mathrm{la}}(\chi)$ with support in $\overline{B}(\Q_p)\mathrm{Iw}$. Frobenius reciprocity then yields an $R[G(\Q_p)]$-equivariant map 
\[
    \mathrm{aug}_{\chi}\colon \cInd_{\mathrm{Iw}}^{G(\Q_p)}(\bfA^{r}_{\chi_0})\longrightarrow \cInd_{\mathrm{Iw}}^{G(\Q_p)}(\crA_{\chi_0})\longrightarrow \mathrm{I}_{\overline{B}}^{\mathrm{la}}(\chi),
\]
where the first map is the natural inclusion. 

To define the differentials in the complex we must introduce some operators. Let $T^- \subset T(\Qp)$ be the set of elements such that $v_p(\alpha_j(t)) \leq 0$ for all simple roots $\alpha_j$; that is, the $p$-adic valuations of the entries of $t$ increase down the diagonal. There is a natural action of $T^-$ on $\bfA_{\chi_0}^{r}$ induced by conjugation on $B(\Zp) \subset \Iw$. For any $t \in T^-$, let 
\[
U_{t} \colon G(\mbb{Q}_p) \to \opn{End}_R(\bfA^{r}_{\chi_0})
\]
denote the unique bi-$\opn{Iw}$-equivariant morphism satisfying
\begin{itemize}
\item $\opn{supp}(U_{t}) = \opn{Iw} \cdot t \cdot \opn{Iw}$, and 
\item $U_t(t) = t \star -$ (where $\star$ denotes the action of $T^-$ on $\bfA^{r}_{\chi_0}$).
\end{itemize}
This gives rise to an endomorphism of $\cInd_{\opn{Iw}}^{G(\mbb{Q}_p)}(\bfA^{r}_{\chi_0})$, also denoted $U_t$ and described in \cite[(2.2)]{KS12}.

For $i \in \{1,\dots,n-1\}$ recall $t_i = \opn{diag}(p,\dots,p, 1,\dots, 1)$, with $i$ lots of $p$, from \S\ref{ss: Infinitesimal families}. Then $t_i^{-1} \in T^-$, and we let 
\[
    y_i \defeq U_{t_i^{-1}} - \chi(t_i),
\]
an endomorphism of $\cInd_{\opn{Iw}}^{G(\mbb{Q}_p)}(\bfA^{r}_{\chi_0})$. 

\begin{definition}
       The \emph{Koszul complex}, concentrated in degrees $[1-n,0]$, is 
\[
    \Lambda^{\bullet}_{R}(R^{\oplus n-1}) \otimes_{R} \cInd_{\opn{Iw}}^{G(\Q_p)}(\bfA^{r}_{\chi_0}),
\]
with differentials $d_k$ given by
\begin{align*}
    d_k \colon \Lambda^{k}_{R}(R^{\oplus n-1}) \otimes_{R} \cInd_{\opn{Iw}}^{G(\Q_p)}(\bfA^{r}_{\chi_0}) &\to \Lambda^{k-1}_{R}(R^{\oplus 2}) \otimes_{R} \cInd_{\opn{Iw}}^{G(\Q_p)}(\bfA^{r}_{\chi_0}) \\
    e_{i_1} \wedge \cdots \wedge e_{i_k} \otimes f &\mapsto \sum_{l=1}^k (-1)^{l+1} e_{i_1} \wedge \cdots \wedge \widehat{e_{i_l}} \wedge \cdots \wedge e_{i_k} \otimes y_{i_l}(f) .
\end{align*}
Here $e_1, \dots, e_{n-1}$ denotes the standard basis of $R^{\oplus n-1}$, and $\widehat{e_{i_l}}$ means omit the element $e_{i_l}$. 
\end{definition}

\begin{lemma}[{c.f. \cite[Theorem 2.5]{KS12}, \cite[Theorem 3.4]{GehrmannRosso}, \cite[Prop.\ 4.12]{Tarrach}}] \label{Lem:KoszulResolution}
The \emph{augmented Koszul complex}
\[
\Lambda^{\bullet}_{R}(R^{\oplus n-1}) \otimes_{R} \cInd_{\opn{Iw}}^{G(\Q_p)}(\bfA^{r}_{\chi_0} ) \xrightarrow{\ \ \opn{aug}_\chi \ \ } \mathrm{I}_{\overline{B}}^{\opn{la}}(\chi)\rightarrow 0
\]
is exact.
\end{lemma}
\begin{proof} 
Suppose first that $\chi$ is trivial on the centre of $G(\Qp)$. Then we can identify the induction $\opn{I}_{\overline{B}}^{\opn{la}}(\chi)$ as an induction from the lower-triangular Borel in $\mathrm{PGL}_n$. Similarly, we can view $\bfA^r_{\chi_0}$ as an induction to the image of $\opn{Iw}$ in $\opn{PGL}_n$. The lemma is thus a consequence of the same result for $\mathrm{PGL}_n$ proved by Kohlhasse and Schraen in \cite[Theorem 2.5]{KS12}. Note we are free to arrange $\overline{B}$ to be the choice of Borel and $\opn{Iw}$ to be the choice of Iwahori subgroup \emph{op.\ cit}.

In general, there exists a finite \'{e}tale extension $R \to S$ of $L$-affinoid algebras and a character $\nu \colon \mbb{Q}_p^{\times} \to S^{\times}$ such that $\chi|_{Z_G(\mbb{Q}_p)} = \nu^n$. Therefore we can factorise $\chi$ as $\chi = \omega \cdot (\nu \circ \opn{det})$ over $S$, where $\omega = \chi \cdot (\nu \circ \opn{det})^{-1}$ is trivial on the centre. By the first part of the proof, we know that the augmented Koszul complex 
\[
\Lambda^{\bullet}_{S}(S^{\oplus n-1}) \otimes_{R} \cInd_{\opn{Iw}}^{G(\Q_p)}(\bfA^{r}_{\omega_0} ) \xrightarrow{\ \ \opn{aug}_\omega \ \ } \mathrm{I}_{\overline{B}}^{\opn{la}}(\omega)\rightarrow 0
\]
is exact, and the twist of this complex by $(\nu \circ \opn{det})$ (which is still exact) is identified with the augmented Koszul complex
\[
\Lambda^{\bullet}_{S}(S^{\oplus n-1}) \otimes_{S} \cInd_{\opn{Iw}}^{G(\Q_p)}(\bfA^{r}_{\chi_0} ) \xrightarrow{\ \ \opn{aug}_\chi \ \ } \mathrm{I}_{\overline{B}}^{\opn{la}}(\chi)\rightarrow 0
\]
via the natural $G(\mbb{Q}_p)$-equivariant identifications 
\[
\cInd_{\opn{Iw}}^{G(\Q_p)}(\bfA^{r}_{\omega_0} ) \otimes (\nu \circ \opn{det}) \cong \cInd_{\opn{Iw}}^{G(\Q_p)}(\bfA^{r}_{\omega_0} \otimes (\nu \circ \opn{det})_0 ) \cong \cInd_{\opn{Iw}}^{G(\Q_p)}(\bfA^{r}_{\chi_0} )
\]
and $\mathrm{I}_{\overline{B}}^{\opn{la}}(\omega) \otimes (\nu \circ \opn{det}) \cong \mathrm{I}_{\overline{B}}^{\opn{la}}(\chi)$. The lemma now follows from finite \'{e}tale descent.
\end{proof}

Now we dualise. Define a complex $M^{\bullet}$ 
\[
M^\bullet = \Lambda^{\bullet}_{R}(R^{\oplus n-1}) \otimes_{R} \opn{Ind}_{\opn{Iw}}^{G(\Q_p)}(\bfD^{r}_{\chi_0})
\]
of $G(\mbb{Q}_p)$-representations, concentrated in degrees $[0, n-1]$, with differentials given by the duals of $d_k$ above (identifying $M^k$ with the continuous dual of $\Lambda^k_R(R^{\oplus n-1}) \otimes_R \cInd_{\opn{Iw}}^{G(\mbb{Q}_p)}(\bfA^{r}_{\chi_0})$). Note that the dual of $d_0$ is given by
\begin{equation}\label{eq:d_0}
d^*_0(f) = \sum_{i=1}^{n-1} e_i \otimes z_i(f),
\end{equation}
where $z_i = U_{t_i} - \chi(t_i)$ (resp. $U_{t_i}$) denotes the dual of $y_i$ (resp. $U_{t_i^{-1}}$). From Lemma \ref{Lem:KoszulResolution} we obtain:

\begin{corollary} \label{c: dual spectral sequence}
With notation and assumptions as in Lemma \ref{Lem:KoszulResolution}, assume also that $\chi(t_i) \in (R^+)^\times$ for each $i$. Then there is a resolution of $G(\Q)$-representations
\[
0\rightarrow \mathcal{A}_{G, c}(\mathrm{I}_{\overline{B}}^{\opn{la}}(\chi)^{\ast})\rightarrow \mathcal{A}_{G, c}(M^{\bullet}).
\]
\end{corollary}
\begin{proof}
   This follows exactly as in \cite[Cor.\ 8.2.10]{GL3-ExceptionalZeros}, following \cite[Prop.\ 3.9]{GehrmannRosso}. Since the proof in both of these references is rather brief, we give a more detailed account here. Firstly, by the same arguments as in Lemma \ref{Lem:KoszulResolution}, we can (and do) assume that $\chi$ is trivial on $Z_G(\mbb{Q}_p)$. 
   
   We first claim that $M^{\bullet}$ is exact in degrees $i > 0$. Similarly to \cite[Theorem 2.5]{KS12} (but dualising), the cohomology in degree $i$ of this complex is computed by $\opn{Ext}^i_{A}(A/\ide{m}, \opn{Ind}_{\opn{Iw}}^{G(\Q_p)}(\bfD^{r}_{\chi_0}))$, where $A = R[X_1, \dots, X_{n-1}]$ is the polynomial algebra over $R$ with $X_i$ acting on $\opn{Ind}_{\opn{Iw}}^{G(\Q_p)}(\bfD^{r}_{\chi_0})$ via the operator $U_{t_i}$, and $\ide{m} = (X_1 - \chi(t_1), \dots, X_{n-1} - \chi(t_{n-1}))$ is the ideal generated the regular sequence $(X_i - \chi(t_i))_{i=1, \dots, n-1}$ (in any order).\footnote{Our notation is slightly different than \cite{KS12}; in \emph{op.cit.}, $A$ is denoted $R$.} It therefore suffices to show that these Ext groups vanish in degree $i >0$. 

   Let
   \[
    N \defeq \opn{c-Ind}_{\opn{Iw}}^{G(\Q_p)}(\mbf{A}^{r}_{\chi_0})(G^-) \subset \opn{c-Ind}_{\opn{Iw}}^{G(\Q_p)}(\mbf{A}^{r}_{\chi_0})
    \]
    be the subspace of functions supported on $G^- \defeq G(\mbb{Z}_p) T^- \opn{Iw}$. By the same arguments as in \cite[Theorem 2.5]{KS12}, we see that the natural inclusion induces an isomorphism
    \[
        \opn{Ext}^i_{A}(A/\ide{m}, N^*) \cong  \opn{Ext}^i_{A}(A/\ide{m}, \opn{Ind}_{\opn{Iw}}^{G(\Q_p)}(\bfD^{r}_{\chi_0})),
    \]
    where $N^*$ is the $R$-dual of $N$. Thus it suffices to show that the left-hand Ext groups vanish.
    
   It is shown in  \cite[Theorem 2.5]{KS12} that the augmented Koszul complex for $N$, namely $\wedge_R^{\bullet}(R^{\oplus n-1}) \otimes_R N$ with the differentials $d_{\bullet}$ above, is exact in degrees $< 0$.
   
   Furthermore, the operators $y_i \defeq U_{t_i^{-1}} - \chi(t_i)$ form a \emph{strict} regular sequence for $N$; indeed they form a regular sequence by \cite[Theorem 2.6]{KS12}, and by using the grading on $N$ (see p.\ 232 \emph{op.cit.}), one can easily check that the operators $y_i$ are isometries, hence strict. Therefore the differentials for the Koszul complex for $N$ are also strict.

   By the Hahn--Banach theorem for locally convex inductive limits of orthonormalisable Banach $R$-modules, we see that the dual Koszul resolution for $N$, namely $\wedge_R^{\bullet}(R^{\oplus n-1}) \otimes_R N^*$ with the differentials $d^*_{\bullet}$ above, is exact in degrees $> 0$. Hence we must have $\opn{Ext}^i_{A}(A/\ide{m}, N^*) = 0$ for $i > 0$, implying, by the arguments above, that  $\opn{Ext}^i_{A}(A/\ide{m}, \opn{Ind}_{\opn{Iw}}^{G(\Q_p)}(\bfD^{r}_{\chi_0})) = 0$ for $i>0$, and thus that $M^{\bullet}$ is exact in degrees $i > 0$.

   Finally, we note that applying $\mathcal{A}_{G, c}(-)$ is exact (because $\mathcal{A}_G(-)$ is exact and $D_G$ is a projective module), so we obtain the desired resolution in part (i).
\end{proof}

\subsection{The Tor spectral sequence for locally analytic principal series}

To verify Gehrmann--Rosso's criterion, we develop and study a Tor spectral sequence for locally analytic principal series representations. This is an analogue of Hansen's spectral sequences, as used in the proof of Theorem \ref{thm:main etale}.

Let $R \to S$ a finite map of $L$-affinoid algebras, and $\chi_R \colon T(\mbb{Q}_p) \to (R^+)^{\times}$ a character. Let $M^{\bullet}_R$ denote the complex where 
\[
M^{\bullet}_R = \wedge^{\bullet}_R R^{\oplus n-1} \otimes_R \opn{Ind}_{\opn{Iw}}^{G(\mbb{Q}_p)}(\mbf{D}^r_{\chi_{R, 0}}).
\]
By Corollary \ref{c: dual spectral sequence}, $\mathcal{A}_{G, c}(M_R^{\bullet})$ is a resolution of $\mathcal{A}_{G, c}(\opn{I}_{\overline{B}}^{\opn{la}}(\chi_R))$. We have similar complexes over $S$.

To construct the Tor spectral sequence, we equate the usual tensor product with the derived tensor product, for which the following result is crucial. We expect there are many ways to prove this; we give a proof which uses condensed mathematics.

\begin{lemma}
    Each term $\mathcal{A}_{G, c}(M_R^{\bullet})$ is a flat $R$-module.
\end{lemma}
\begin{proof}
 We first claim that $M_R^{\bullet}$ is a (strongly countable) Fr\'{e}chet module, namely a countable inverse limit of orthonormalisable Banach $R$-modules with dense transition maps. It suffices to prove this for $\opn{Ind}_{\opn{Iw}}^{G(\mbb{Q}_p)}(\mbf{D}^r_{\chi_{R, 0}})$, which we present as an inverse limit
    \[
    A \defeq \opn{Ind}_{\opn{Iw}}^{G(\mbb{Q}_p)}(\mbf{D}^r_{\chi_{R, 0}}) = \varprojlim_C \opn{Ind}_{\opn{Iw}}^{G(\mbb{Q}_p)}(\mbf{D}^r_{\chi_{R, 0}})(C) = \varprojlim_n A_n,
    \]
    where $\opn{Ind}_{\opn{Iw}}^{G(\mbb{Q}_p)}(\mbf{D}^r_{\chi_{R, 0}})(C) \subset \opn{Ind}_{\opn{Iw}}^{G(\mbb{Q}_p)}(\mbf{D}^r_{\chi_{R, 0}})$ denotes functions with support in a compact open subspace $C \subset G(\mbb{Q}_p)$. Each term in the inverse limit is an orthonormalisable Banach $R$-module by taking the sup-norm over $C$.

    It is easy to check that orthonormalisable Banach $R$-modules are $R$-flat. We claim strongly countable Fr\'{e}chet modules are too. For this we consider the condensed $\underline{R}$-module $\underline{A} = \varprojlim_n \underline{A}_n$. Following the proof of \cite[Proposition 2.4.3]{FourierPaper}, we have an exact sequence
    \[
    0 \to \underline{A} \to \prod_{n} \underline{A}_n \to \prod_{n} \underline{A}_n \to 0 
    \]
    and $\prod_n \underline{A}_n = \varinjlim_i \underline{B}_i$ is a filtered colimit of orthornormalisable Banach $R$-modules. Similarly, we have an exact sequence of $R$-modules
    \[
    0 \to A \to \prod_n A_n \to \prod_n A_n \to 0,
    \]
    noting it is exact on the right by the topological Mittag-Leffler property. 

    Recall the functor $(-)(*)$ from condensed $\underline{R}$-modules to $R$-modules given by evaluating at the point. This functor preserves colimits, is right-exact, and $\underline{X}(*) = X$ for compactly generated topological spaces (see \cite[Proposition 1.7]{CondensedNotes}). In particular, the morphism $A = \underline{A}(*) \to (\prod_{n} \underline{A}_n)(*)$ factors $A \hookrightarrow \prod_{n} A_n$, hence is injective. Combining these properties, we have an exact sequence
    \[
    0 \to A \to (\prod_{n} \underline{A}_n)(*) \to (\prod_{n} \underline{A}_n)(*) \to 0 
    \]
    and $(\prod_{n} \underline{A}_n)(*) \cong \varinjlim_i B_i$. Since each $B_i$ is $R$-flat, we see that $(\prod_{n} \underline{A}_n)(*)$ is $R$-flat, and hence $A$ is $R$-flat. 

    To conclude the proof, it is a simple check that $\mathcal{A}_G(M_R^{\bullet})$ is $R$-flat (because $M_R^{\bullet}$ is). Finally $\mathcal{A}_{G, c}(M_R^{\bullet})$ is $R$-flat because $D_G$ is projective and $\opn{Hom}_{\mbb{Z}}(D_G, \mathcal{A}_G(M_R^{\bullet})) \otimes_R C = \opn{Hom}_{\mbb{Z}}(D_G, \mathcal{A}_G(M_R^{\bullet}) \otimes_R C)$.
\end{proof}

As a consequence of this lemma, we have the following:

\begin{lemma}
    Let $R \to S$ be a finite morphism of $L$-affinoid algebras. We have a Tor-spectral sequence
    \[
    E_{2}^{i, j} \colon \opn{Tor}^R_{-i}(\opn{H}^j(G(\mbb{Q}), \mathcal{A}_{G, c}(\opn{I}_{\overline{B}}^{\opn{la}}(\chi_R))), S) \Rightarrow \opn{H}^{i+j}(G(\mbb{Q}), \mathcal{A}_{G, c}(\opn{I}_{\overline{B}}^{\opn{la}}(\chi_S))) .
    \]
    This is equivariant with respect to the prime-to-$S$ Hecke algebra as well as $K_{\infty}/K_{\infty}^{\circ}$, so we can localise it at $\phi^S$ and take $\varepsilon$-parts.
\end{lemma}
\begin{proof}
    By the above lemma, $\mathcal{A}_{G, c}(M_R^{\bullet}) \otimes^{\mbf{L}}_R S = \mathcal{A}_{G, c}(M_R^{\bullet}) \otimes_R S = \mathcal{A}_{G, c}(M_R^{\bullet}) \widehat{\otimes}_R S = \mathcal{A}_{G, c}(M_S^{\bullet})$, where the second equality follows from the map $R \to S$ being finite. From this, we obtain the usual Tor-spectral sequence (see, e.g., \cite[Tag 061Y]{stacks-project}).
\end{proof}

Now we specialise to our case of interest. Fix $\varepsilon \in \mathcal{E}_{\pi}$, and let $R = \mathcal{O}(\Omega)$ and $S = \cO(\Lambda)$, where $\Lambda \subset \Omega$ is a Zariski-closed affinoid. Let $\chi_R = \chi_\Omega \colon T(\Qp) \to (\cO(\Omega)^+)^\times$ and $\chi_S = \chi_\Lambda$ be the characters from \eqref{eq:chilambda}. Note that $\chi_{R,0} = \chi_\Omega|_{T(\Zp)} = \kappa_\Omega^{\opn{univ}}$, the universal weight character over $\Omega$. In particular, $\bfA_{\chi_{R,0}}^r = \bfA_{\Omega}^r$ and $\bfD_{\chi_{R,0}}^r = \bfD_\Omega^r$ are the usual modules of $r$-analytic functions and distributions over $\Omega$. Similarly $\bfD_{\chi_{S,0}}^r = \bfD_\Lambda^r$. We let $\lambda = 1 \in \Omega$ denote the trivial character.

As we are assuming non-abelian Leopoldt, by Facts \ref{facts:hansen other}(C) and \eqref{eq:arithmetic to betti c}, we know 
\begin{equation}\label{eq:cohom top degree}
    \h^\bullet(G(\Q),\cA_{G,c}(\opn{Ind}_{\opn{Iw}}^{G(\Qp)}(\bfD_{\Omega}^r)))_{\phi^S}^{\opn{ord}, \varepsilon} \text{ is concentrated in degree $q'+\ell$,}
\end{equation}
where $q' = q-n+1$. To ease notation, set 
\[
\mathcal{N}_{\Omega} = \opn{H}^{q'+\ell}(G(\mbb{Q}), \mathcal{A}_{G, c}(\opn{Ind}_{\opn{Iw}}^{G(\mbb{Q}_p)}(\mbf{D}^r_{\Omega})))_{\phi^S}^{\opn{ord}, \varepsilon}.
\]
\begin{proposition}\label{prop:quasi-iso sum}
In the derived category, we have a quasi-isomorphism
\[
R\Gamma(G(\mbb{Q}), \mathcal{A}_{G, c}(\opn{I}_{\overline{B}}^{\opn{la}}(\chi_\Omega)))^{\varepsilon}_{\phi^S} \cong \bigoplus_{k=0}^{n-1} (\wedge^{k}_{R_\lambda} R_\lambda^{\oplus n-1} \otimes_{R_\lambda} \mathcal{N}_{\Omega})[-q'-\ell-k] .
\]
\end{proposition}
Here $M[-j]$ denotes the one-term complex $\dots \to 0 \to M \to 0 \to \dots$, with $M$ in the $j$th place.

\begin{proof}
Recall from Corollary \ref{c: dual spectral sequence} the resolution of $G(\Q)$-modules
\begin{equation}\label{eq:resolution 2}
0\rightarrow \mathcal{A}_{G, c}(\mathrm{I}_{\overline{B}}^{\opn{la}}(\chi_\Omega)^{\ast})\rightarrow \mathcal{A}_{G, c}(M_R^{\bullet}).
\end{equation}
Let $C_0^{p, \bullet} = \opn{Hom}_{R}(R[G(\mbb{Q})^{\bullet +1}], \mathcal{A}_{G, c}(M_R^p))$ be the standard acyclic resolution of $G(\mbb{Q})$-representations. This yields a double complex. Let $C_1^{\bullet, \bullet} = (C_0^{\bullet, \bullet})^{G(\mbb{Q})}$, a double complex whose associated total complex is quasi-isomorphic to $R\Gamma(G(\Q),\mathcal{A}_{G, c}(\mathrm{I}_{\overline{B}}^{\opn{la}}(\chi_{\Omega})^{\ast}))$ in the derived category.

By construction, we can describe each column of the double complex via
\begin{align*}
    C^{p,\bullet}_1 &= R\Gamma(G(\Q), \cA_{G,c}(M_R^p))\\
    &\cong \wedge^{p}_{R} R^{\oplus n-1} \otimes_{R} R\Gamma(G(\Q),\cA_{G,c}(\opn{Ind}_{\opn{Iw}}^{G(\Qp)}(\bfD_{\Omega}^r))).
\end{align*}
In particular every term of $C_1^{\bullet,\bullet}$ has an action of the prime-to-$p$ Hecke algebra $\mbb{T}^S(K^p)$, and of the operators $U_{t_i}$ for $1 \leq i \leq n-1$. Moreover, all the differentials in $C^{\bullet,\bullet}_1$ are equivariant with respect to the prime-to-$p$ Hecke operators, since the maps in \eqref{eq:resolution 2} are $G(\A_f^p)$-equivariant. These maps, hence the differentials, are also equivariant with respect to the operators $U_{t_i}$; this is because all the $U_{t_i}$'s commute by \cite[Lem.\ 2.3]{KS12}, hence they commute with the operators $z_i$. 

Let $t = t_1\cdots t_{n-1}$ and $U_t = U_{t_1}\cdots U_{t_{n-1}}$. Then the action of $t$ on $\mbf{D}^r_{\Omega}$ is compact, hence by the same argument at the start of the proof of \cite[Proposition 3.4]{KS12}, the action of $U_t$ is compact on the orthonormalisable $R$-Banach space $\opn{Ind}_{\opn{Iw}}^{G(\mbb{Q}_p)}(\mbf{D}^r_{\Omega})(C)$ (for any compact open $C \subset G(\mbb{Q}_p)$). Therefore there exists an ordinary idempotent on $\opn{Ind}_{\opn{Iw}}^{G(\mbb{Q}_p)}(\mbf{D}^r_{\Omega})(C)$. Passing to the inverse limit and taking a finite direct sum, we obtain an ordinary idempotent on $M^{\bullet}_R$. This in turn induces an ordinary idempotent on $\mathcal{A}_{G, c}(M_R^{\bullet})$ and hence $C_1^{\bullet, \bullet}$ (note that the ordinary idempotent commutes with the $G(\mbb{Q})$-action, since the same is true for the operator $U_t$). 

Let $C_2^{\bullet,\bullet}$ be the double complex obtained by applying this ordinary projector to $C_1^{\bullet,\bullet}$. We claim the total complex $\opn{Tot}(C_2^{\bullet,\bullet})$ is still quasi-isomorphic to $R\Gamma(G(\Q),\mathcal{A}_{G, c}(\mathrm{I}_{\overline{B}}^{\opn{la}}(\chi_{\Omega})^{\ast}))$. Indeed, as in the end of the proof of \cite[Prop.\ 3.9]{GehrmannRosso}, a standard property of Koszul complexes means that after taking cohomology in the $p$-direction, each $z_i$ acts as 0; so $U_{t_i}$ acts by $\chi_\Sigma(t_i)$, which is a $p$-adic unit (since it describes the Hecke eigenvalues over a $p$-adic family through an ordinary point). In particular, $U_{t}$ acts by $p$-adic units on the cohomology of the total complex, and hence 
\[
\h^r(\opn{Tot}(C_1^{\bullet,\bullet})) = \h^r(\opn{Tot}((C_1^{\bullet,\bullet})^{\opn{ord}}) = \h^r(\opn{Tot}(C_2^{\bullet,\bullet})).
\]
In particular 
\[
\opn{Tot}(C_2^{\bullet,\bullet}) \cong R\Gamma(G(\Q),\mathcal{A}_{G, c}(\mathrm{I}_{\overline{B}}^{\opn{la}}(\chi_{\Omega})^{\ast})).
\]
Let $C_3^{\bullet,\bullet} = (C_2^{\bullet,\bullet})^{\varepsilon}_{\phi^S}$ be the complex obtained by localising all terms at $\phi^S$ and taking $\varepsilon$-parts. Then 
\[
    \opn{Tot}(C_3^{\bullet,\bullet}) \cong R\Gamma(G(\Q),\mathcal{A}_{G, c}(\mathrm{I}_{\overline{B}}^{\opn{la}}(\chi_{\Omega})^{\ast}))^{\varepsilon}_{\phi^S}.
\]
Now, in the derived category, we have
\[
    C_3^{p,\bullet} \cong \wedge^{p}_{R_\lambda} R_{\lambda}^{\oplus n-1} \otimes_{R_\lambda} R\Gamma(G(\Q),\cA_{G,c}(\opn{Ind}_{\opn{Iw}}^{G(\Qp)}(\bfD_{\Omega}^r))_{\phi^S}^{\opn{ord}, \varepsilon}.
\]
By \eqref{eq:cohom top degree}, this has cohomology supported in degree $q'+\ell$, hence we have quasi-isomorphisms
\[
C_3^{p, \bullet} \xleftarrow{\sim} \tau^{\leq q'+\ell} C_3^{p, \bullet} \xrightarrow{\sim} \tau^{\geq q'+\ell} \tau^{\leq q'+\ell} C_3^{p, \bullet} \xrightarrow{\sim} \opn{H}^{q'+\ell}(\tau^{\geq q'+\ell} \tau^{\leq q'+\ell} C_3^{p, \bullet})[-q'-\ell] =: C^{p, \bullet}_{4}
\]
where $\tau^{\geq i}$ and $\tau^{\leq i}$ denote the natural truncation functors. This composition of quasi-isomorphisms is functorial with respect to maps of complexes, so we obtain a double complex $C^{\bullet, \bullet}_4$. Note that
\[
    C_4^{p,\bullet} \cong \wedge^{p}_{R_\lambda} R_{\lambda}^{\oplus n-1} \otimes_{R_\lambda}\cN_\Omega[-q'-\ell],
\]
with differentials given by those in the Koszul complex. The complex $C_4^{\bullet, \bullet}$ has the property that:
\begin{itemize}
\item only row $q'+\ell$ is non-zero,
\item and $\opn{Tot}(C_4^{\bullet,\bullet}) \cong \opn{Tot}(C_3^{\bullet,\bullet}) \cong R\Gamma(G(\Q),\mathcal{A}_{G, c}(\mathrm{I}_{\overline{B}}^{\opn{la}}(\chi_{\Omega})^{\ast}))^{\varepsilon}_{\phi^S}.$
\end{itemize}
In particular, 
\begin{equation}\label{eq:RGamma C4}
R\Gamma(G(\Q),\mathcal{A}_{G, c}(\mathrm{I}_{\overline{B}}^{\opn{la}}(\chi_{\Omega})^{\ast}))^{\varepsilon}_{\phi^S} \cong \wedge^{\bullet}_{R_\lambda} R_{\lambda}^{\oplus n-1} \otimes_{R_\lambda}\cN_\Omega[-q'-\ell],
\end{equation}
with differentials given by those in the Koszul complex. We claim all these differentials are zero, even before taking cohomology. Indeed, $\cN_\Omega$ is a free module over $\cO(\Sigma)_\lambda$ upon which $U_{t_i}$ acts by $\chi_\Sigma(t_i)$. In particular, each $z_i = U_{t_i}-\chi_{\Omega}(t_i)$, and hence each differential, vanishes on $\cN_\Omega$ (because scaling by $\chi_{\Omega}(t_i)$ on $\mathcal{N}_{\Omega}$ is the same as scaling by $\chi_{\Sigma}(t_i)$). Thus the right-hand side of \eqref{eq:RGamma C4} gives exactly the claimed direct sum decomposition. 
\end{proof}

\begin{remark}
Concretely, this proposition (and its proof) implies $\h^\bullet(G(\mbb{Q}), \mathcal{A}_{G, c}(\opn{I}_{\overline{B}}^{\opn{la}}(\chi_\Omega)))^{\varepsilon}_{\phi^S}$ is supported in degrees $q'+\ell, \dots, q'+\ell+n-1$, and that for any $0 \leq k \leq n-1$, we have
\[
    \h^{q'+\ell+k}(G(\mbb{Q}), \mathcal{A}_{G, c}(\opn{I}_{\overline{B}}^{\opn{la}}(\chi_\Omega)))^{\varepsilon}_{\phi^S} \cong \wedge^{k}_{R_\lambda} R_\lambda^{\oplus n-1} \otimes_{R_\lambda} \mathcal{N}_{\Omega}.
\]
However, the direct sum decomposition of the left-hand $R\Gamma(-)$ is stronger than this statement alone.
\end{remark}

\begin{corollary}\label{cor:sigma projective}
\begin{itemize}
\item[(i)] For any Zariski-closed affinoid $\Lambda \subset \Omega$, the localised Tor spectral sequence
    \[
    E_{2}^{i, j} \colon \opn{Tor}^{\cO(\Omega)_\lambda}_{-i}(\opn{H}^j(G(\mbb{Q}), \mathcal{A}_{G, c}(\opn{I}_{\overline{B}}^{\opn{la}}(\chi_\Omega)))^{\varepsilon}_{\phi^S}, \cO(\Lambda)_\lambda) \Rightarrow \opn{H}^{i+j}(G(\mbb{Q}), \mathcal{A}_{G, c}(\opn{I}_{\overline{B}}^{\opn{la}}(\chi_\Lambda)))^{\varepsilon}_{\phi^S} 
    \]
    degenerates on the second page (i.e., all differentials on the $r$-th page, for $r \geq 2$, are zero).

\item[(ii)]    Let $\Lambda = \Sigma$. Then the $\cO(\Sigma)_\lambda$-modules $\opn{H}^{\bullet}(G(\mbb{Q}), \mathcal{A}_{G, c}(\opn{I}_{\overline{B}}^{\opn{la}}(\chi_\Sigma)))_{\phi^S}^{\varepsilon}$ are projective.
\end{itemize}
\end{corollary}
\begin{proof}
For (i), we observe that the Tor spectral sequence is constructed at the level of $R\Gamma$. As the complex is quasi-isomorphic to a direct sum of complexes in the derived category, we can thus decompose the spectral sequence as a direct sum, with differentials respecting this decomposition. Since each complex in the direct sum has only one term, it follows that the differentials in this summand are identically zero; hence we conclude.

\medskip

    For (ii), the degeneration of the localised Tor spectral sequence implies that $\opn{H}^{a}(G(\mbb{Q}), \mathcal{A}_{G, c}(\opn{I}_{\overline{B}}^{\opn{la}}(\chi_S)))^{\varepsilon}_{\phi^S}$ can be obtained by successive extensions of $E_2^{i, j}$ with $i+j = a$. But $E_2^{i, j}$ is a projective $S$-module because
    \[
    \opn{Tor}^{\mathcal{O}(\Omega)_{\lambda}}_{i}(\mathcal{N}_{\Omega}, \mathcal{O}(\Sigma)_{\lambda}) \cong \hc{q+\ell-i}(K,\bfD_\Sigma^r)_{\phi^S}^{\opn{ord}, \varepsilon}
    \]
    is a projective (even free) $\mathcal{O}(\Sigma)_{\lambda}$-module, the isomorphism being Proposition \ref{prop:tor iso} and \eqref{eq:arithmetic to betti c}, and freeness being Theorem \ref{thm:main etale}. The result follows from the fact that being projective is closed under extensions.
\end{proof}

To conclude, we now use the localised Tor spectral sequence again, but now with $R = \mathcal{O}(\Sigma)$ and $S = L$, with the map $R \to S$ being specialisation at $\lambda = 1$. We take $\chi_R = \chi_\Sigma$ and $\chi_S = 1$, the trivial character.

\begin{proposition}\label{prop:red specialise}
   For all $j$, specialisation at $\lambda = 1$ induces an isomorphism
    \[
    \opn{H}^{j}(G(\mbb{Q}), \mathcal{A}_{G, c}(\opn{I}_{\overline{B}}^{\opn{la}}(\chi_{\Sigma})))^{\varepsilon}_{\phi^S} \otimes_{\mathcal{O}(\Sigma)_{\lambda}} \cO(\Sigma)/\m_\lambda \cong \opn{H}^{j}(G(\mbb{Q}), \mathcal{A}_{G, c}(\opn{I}_{\overline{B}}^{\opn{la}}(1)))^{\varepsilon}_{\phi^S}.
    \]
\end{proposition}
\begin{proof}
Let $(E_2^{i,j})^{\varepsilon}_{\phi^S}$ be the localised Tor spectral sequence for $\cO(\Sigma)\to L$.     By Corollary \ref{cor:sigma projective}, this vanishes unless $i = 0$. In particular, the spectral sequence degenerates on the $E_2$ page, immediately giving the claimed isomorphism.
\end{proof}

\subsection{The automorphic Benois--Colmez--Greenberg--Stevens formula} \label{ss: automorphic BCGS}

The following is the main result of \S\ref{sec:automorphic BCGS}. We maintain our running assumptions on $\pi$ from Set-up \ref{setup}, including Assumption \ref{ass:nap}, and the notation  introduced between the start of \S\ref{sec:automorphic BCGS} and the end of \S\ref{sec:infinitesimal weights}. In particular, recall we have fixed an admissible sign $\varepsilon \in \mathcal{E}_{\pi}$ and we have a continuous homomorphism $\partial\chi_v[c] \in \opn{Hom}_{\opn{cts}}(\Qp^\times,L)$ from \eqref{eq:partial-chi}.

\begin{theorem}\label{p: automorphic BCGS} 
Let $\pi$ and $\Sigma$ be as in Set-up \ref{setup}, and assume the non-abelian Leopoldt conjecture holds. Let $1 \leq c \leq n-1$ and $v= (v_1,..., v_{n})\in \opn{T}_1(\Sigma)$. 
\begin{itemize}
\item[(i)] For all $0 \leq r \leq \ell$, we have 
\[
    \partial\chi_{v}[c] \in \bL^{r,\varepsilon}_c(\pi).
    \]
\item[(ii)] If $v_c \neq v_{c+1}$, then $\bL^{r,\varepsilon}_c(\pi) = \langle \partial\chi_v[c]\rangle \subset \opn{Hom}_{\opn{cts}}(\Qp^\times,L)$ for all $r$.
\end{itemize}
\end{theorem}

\begin{proof}
To show (i), we must verify Gehrmann--Rosso's criterion. This follows by considering the isomorphism
    \[
    \opn{H}^{q'+r}(G(\mbb{Q}), \mathcal{A}_{G, c}(\opn{I}_{\overline{B}}^{\opn{la}}(\chi_{\Sigma})))^{\varepsilon}_{\phi^S} \otimes_{\mathcal{O}(\Sigma)_{\lambda}} \cO(\Sigma)/\m_\lambda \cong \opn{H}^{q'+r}(G(\mbb{Q}), \mathcal{A}_{G, c}(\opn{I}_{\overline{B}}^{\opn{la}}(1)))^{\varepsilon}_{\phi^S}
    \]
from Proposition \ref{prop:red specialise}, and then noting that it factorises through specialisation $\cO(\Sigma) \to L[\epsilon] \to L$ at any $v \in \opn{T}_1(\Sigma)$ (see \cite[\S2.3]{GehrmannRosso}, particularly Proposition 2.4). In particular, surjectivity of the reduction from $\Sigma$ to $L$ implies surjectivity of the reduction from $L[\epsilon]$ to $L$ along $v$, as required.

\medskip

It remains to show (i) implies (ii).  We know
\begin{equation}\label{eq:partial chi}
    \partial\chi_v[c]|_{\Zp^\times} = (\partial\chi_{v,c}- \partial\chi_{v,c+1})|_{\Zp^\times} = (v_c-v_{c+1})\log_p,
\end{equation}
as $1 + \partial\chi_{v,i}\epsilon = \chi_{v,i}$, and $\chi_{v,i}|_{\Zp^\times} = \kappa^{\opn{univ}}_{v,i} = 1 + v_i\log_p \epsilon$. If $v_c \neq v_{c+1}$,  then by \eqref{eq:partial chi} we see $\partial\chi_v[c]\neq 0$. By Proposition \ref{p: L-invariant}, this means $\partial\chi_v[c]$ will be a basis of the automorphic $\cL$-invariant $\bL_{c}^{r,\varepsilon}(\pi)$ for all $r$.
\end{proof}

\begin{definition}\label{def:non-c-parabolic}
Let $1 \leq c \leq n-1$. We say $\pi$ \emph{admits non-$c$-parabolic deformations}\footnote{To explain the terminology, note that if $v_c \neq v_{c+1}$, then $v$ is not a tangent vector in the $Q_c$-parabolic weight space at $1$ (defined in \cite[Def.\ 3.3]{BW20}).} if there exists $v \in \opn{T}_1(\Sigma)$ such that $v_c \neq v_{c+1}$.
\end{definition}

We obtain the following immediate corollary.

\begin{corollary}
Let $1 \leq c \leq n-1$, and suppose $\pi$ admits non-$c$-parabolic deformations. Then \cite[Conjecture A(i)--(ii)]{AutomorphicLinvariants} holds for $\pi$ at $c$; that is, $\mbb{L}_c^{r, \varepsilon}(\pi) \subset \opn{Hom}_{\opn{cont}}(\mbb{Q}_p^{\times}, L)$ has codimension one, and is independent of the degree $r$. 
\end{corollary}

\begin{proof}
Let $v \in \opn{T}_1(\Sigma)$ be the given non-$c$-parabolic weight deformation, so that $v_c \neq v_{c+1}$. By part (ii) of the theorem, we know $\bL_c^{r,\varepsilon}(\pi) = \langle \partial\chi_v[c]\rangle$ for all $r$. The result follows since $\partial\chi_v[c]$ is independent of $r$.
\end{proof}

In this case, we thus have a complete description of the $\cL$-invariant for $\pi$, which we can exactly pin down in terms of infinitesimal data in a $p$-adic family through $\pi$.

\begin{remark}\label{rem:symplectic}
It is natural to ask when Definition \ref{def:non-c-parabolic} is satisfied. For $\GL_3$, a weak analogue of the Calegari--Mazur condition would predict it always holds, for all $c$; see \cite[Rem.\ 8.2.4]{GL3-ExceptionalZeros}. If one could prove that the $Q_c$-parabolic eigenvariety is reduced at $\pi$, then this also implies the condition.

If $\pi$ is essentially-self-dual, this assumption should be easier to verify. In this case we expect $\pi$ to vary in an ordinary family over the pure weight space $\sW^0$, yielding non-$c$-parabolic deformations for all $c$ (see Remark \ref{rem:esd}). Such a $\pi$ has either orthogonal or symplectic type. In the symplectic case, it is shown in \cite[Thm.\ A$'$]{BDGJW} that if $\pi_p$ is unramified and its $p$-adic $L$-function $L_p(\pi)$ is not identically zero, then $\pi$ varies over $\sW^0$.  In the Steinberg case, combining the methods \emph{op.\ cit}.\ with the $p$-adic $L$-function of Liu--Sun \cite{LS25} should also yield this result. We hope to return to this in future work.

Finally, in \S\ref{sec:symmetric powers} below, we give a class of unconditional examples; namely, we show that if $\pi$ is the symmetric power lift of a classical modular form, then it admits non-$c$-parabolic deformations for all $c$.
\end{remark}

\begin{remark}
Part (iii) of Gehrmann's conjecture also predicts independence of the sign $\varepsilon$. If $\Sigma = \Sigma^\varepsilon$ is independent of the choice of (admissible) sign, then we would also obtain this independence under the above hypotheses. This would also follow under the weaker condition that there exists an affinoid $\Sigma'$ contained in $\Sigma^\varepsilon$ for all $\varepsilon$, and such that $v \in \opn{T}_1(\Sigma')$, where $v$ is the tangent vector in the theorem.

When there exists a \emph{classical} family $\sC$ through $\pi$ -- that is, a component containing a Zariski-dense set of classical points -- then exploiting the Zariski-density of classical cuspidal points, Matsushima's formula, and Newton--Johansson's $p$-adic Langlands functoriality \cite{JoNew}, we can show that $\Sigma' \defeq w(\sC) \subset \Sigma^\varepsilon$ for all admissible signs $\varepsilon \in \cE_\pi$. This idea is explained fully in \cite[Cor.\ 7.22]{BDW20}. We expect such classical families to exist if $\pi$ is essentially self-dual; in particular, combining these observations with Remark \ref{rem:symplectic} should yield independence of sign in the symplectic-type setting.
\end{remark}

\begin{remark}\label{rem:venkatesh}
    We note that Gehrmann shows (see \cite[Theorem 4.6]{AutomorphicLinvariants}) that $\mbb{L}_c^{r, \varepsilon}(\pi)$ is independent of the degree $r$, conditional on the conjecture of Venkatesh \cite{Ven19} that $\opn{H}^*(Y(K), L)[\pi]$ is generated by $\opn{H}^q(Y(K), L)[\pi]$ under the action of the derived Hecke algebra. This conjecture is currently still open. Our approach instead assumes the non-abelian Leopoldt conjecture (which is also not known in general) rather than Venkatesh's conjecture, and provides an alternative way to show independence of the degree. 
    
    That such results are obtainable is perhaps not that surprising given the relationship between Venkatesh's conjecture and the non-abelian Leopoldt conjecture (explored in \cite{HT17}). Indeed, in \cite{Ven19} Venkatesh makes a motivic conjecture. As explained in \cite[Intro.]{HT17}, the conjecture Gehrmann uses -- studied in \cite[\S7]{Ven19}, via Taylor--Wiles-patching-style arguments -- can be viewed as an $\ell$-adic realisation of the motivic conjecture. Hansen--Thorne's treatment, in \cite[Cor.\ 4.10]{HT17}, instead shows that non-abelian Leopoldt implies an analogous $p$-adic realisation. (Whilst we don't use any of this in the present paper, for completeness we do note that Hansen--Thorne assume that $\pi_p$ is unramified, which excludes the Steinberg case we consider.)
\end{remark}

\begin{remark}\label{rem:GR bottom}
In \cite{GehrmannRosso}, a version of Theorem \ref{p: automorphic BCGS} is proved for groups $G$ admitting discrete series, i.e.\ with $\ell = 0$. They consider a commutative diagram
\[
    \xymatrix{
 \mathrm{H}^{j}(G(\Q), \cA_{G, c}(\mathrm{I}_{\overline{B}}^{\mathrm{la}}(\chi_\Sigma)^{\ast}))\ar^{\newmod{\m_1}}[d]\ar^-{\mathrm{aug}_{\chi_\Sigma}^{j}}[rr] && \hc{j+n-1}(K,\bfD_\Sigma^r)\ar^{\newmod{\m_1}}[d]\\
 \mathrm{H}^{j}(G(\Q), \cA_{G, c}(\mathrm{I}_{\overline{B}}^{\mathrm{la}}(1)^{\ast})) \ar^-{\mathrm{aug}_{1}^{j}}[rr] && \hc{j+n-1}(K,\bfD_1^r),
}
\]
where the horizontal maps are induced from the augmentation map in the Koszul complex. The double complex in Proposition \ref{prop:quasi-iso sum} yields a spectral sequence $E_1^{p,q}$. When $\ell=0$, after localising and taking ordinary parts, this sequence degenerates from the $E_2$ page, from which one can determine the image of the horizontal maps above in the (unique) degree $j=q'$ of non-zero cohomology. By Proposition \ref{prop:top cyclic}(i), the right-hand map is surjective. Gehrmann and Rosso deduce that the top two maps are surjective.  They also show, using crucially that $\ell=0$, that the bottom map is injective. In particular, the left-hand map must be surjective, verifying the criterion.

In our setting, one can still control the image of the horizontal maps in all degrees (using intermediate affinoids, and Proposition \ref{prop:spec surjective}). However the bottom map is only injective in bottom degree $j=q'$. In particular, whilst one can deduce Theorem \ref{p: automorphic BCGS} in bottom degree, one cannot use this method to go any further. This is why we have introduced the Tor spectral sequence for locally analytic principal series, to allow direct study of the left-hand map (and direct verification of Gehrmann--Rosso's criterion).
\end{remark}

\section{The Galois Benois--Colmez--Greenberg--Stevens formula} \label{s: families of triangulation}

In this section we continue using notation from \S\ref{sec:automorphic BCGS}.  We assume $\pi$ and $\Sigma$ are as in Set-up \ref{setup}, after fixing a sign $\varepsilon \colon K_\infty/K_\infty^\circ \to \{\pm 1\}$. Moreover, we will assume that Hypothesis \ref{h: abs irred and decomposed generic} holds. Given $v\in T_1(\Sigma)$, recall in particular the characters $\chi_{v, 1}, \chi_{v, 2},..., \chi_{v, n}$ from the end of \S\ref{sec:infinitesimal weights}, and let $\omega$ denote the $p$-adic cyclotomic character (i.e., $\omega \colon \Qp^\times \to \Qp^\times$ with $\omega(p) = 1$ and $\omega(x) = x$ for all $x \in \Zp^\times$).

\subsection{Main result and strategy}
The goal of this section is to prove the following result.

\begin{theorem} \label{t: TriangulationInFamilies}
  Assume that Hypothesis \ref{h: abs irred and decomposed generic} holds. If $v\in T_1(\Sigma)$, there exists a rank $n$ trianguline $(\varphi, \Gamma)$-module $D_{L[\epsilon]}$ over $L[\epsilon]$ with parameters $\{ \chi_{v, 1}, \chi_{v, 2}\omega^{-1},..., \chi_{v, n}\omega^{1-n} \}$. Moreover, the specialisation $D_{L[\epsilon]}$ mod $\epsilon$ corresponds, via $p$-adic Hodge theory, to $\rho_{\pi}|_{G_{\mbb{Q}_p}}$.
\end{theorem}

We follow the strategy of the proof of \cite[Theorem 8.3.5]{GL3-ExceptionalZeros}, crucially using \cite[Cor. 9.7.3]{GL3-ExceptionalZeros}. Moreover, we will adopt significant notation from \cite{GL3-ExceptionalZeros}, giving precise references where necessary. In particular, let $\varpi \in \cO_L$ denote a uniformiser, and let $\m = (\varpi, \phi) \subset \bT(K)_{\cO_L}$, the maximal ideal cutting out the mod $\varpi$ Hecke eigensystem associated with $\pi$. By Hypothesis \ref{h: abs irred and decomposed generic}, $\m$ is non-Eisenstein. 

For integers $m\geq 1$ and $c\geq b\geq 0$ with $c\geq 1$ we denote $K(b,c)= K^p \mathrm{Iw}(b,c)\subset \mathrm{GL}_n(\A_f)$ an open compact subgroup, where $K^p$ was introduced at the beginning of \S\ref{sec:automorphic BCGS} and $\mathrm{Iw}(b,c)$ is as in \cite[\S 9.1.5]{GL3-ExceptionalZeros}. We consider the cohomology $R\Gamma(Y(K(b,c)), \cO_L/\varpi^m)$ of the locally symmetric space $Y(K(b,c))$, and we denote by $\bT(K)(b, c; m)\subset \mathrm{End}_{\mathbf{D}(\cO_L)}(R\Gamma(Y(K(b,c)), \cO_L/\varpi^m)^{\mathrm{ord},\varepsilon})$  the $\cO_L$-module generated by the image of the Hecke algebra $\bT(K)_{\cO_L}$. Moreover, we denote 
\[
    \cT= \varprojlim_{b, c, m}\bT(K)(b, c; m).
\]
 Note that $\m$ is in the support of $\bT(K)(0, 1; m)$ for some $m\geq 1$, hence gives a maximal ideal in $\cT$.

The following infinitesimal triangulation result was proved in \cite[Cor. 9.7.3]{GL3-ExceptionalZeros}.

\begin{theorem}\label{thm:BGW LGC}
There exists a nilpotent ideal $J \subset \cT_{\m}$ such that for any morphism $f \colon \cT_{\m}/J \to L[\epsilon]$, there exists a rank $n$ trianguline $(\varphi,\Gamma)$-module with prescribed parameters $f \circ \tilde\chi_i$, as described in Definition 9.1.2(2) \emph{op.\ cit.}\footnote{Here we use the notation $\tilde{\chi}_i$ for the characters in \cite[Def.\ 9.1.2(2)]{GL3-ExceptionalZeros} to avoid clashing with the notation earlier in this article.}
\end{theorem}

It remains to exhibit a morphism $f$ such that $f \circ \tilde\chi_i$ can be identified with the required $\chi_{v,i}\omega^{1-i}$. To do this, we compare overconvergent cohomology (as used in our construction of eigenvarieties, hence to define $\chi_{v,i}$) and completed cohomology (giving rise to the Hecke algebra $\cT_{\m}$). For completeness we prove the weak comparison we require in our special case, for $\GL_n$ around ordinary points. We remark that a stronger (and more general) version of this comparison can be found in \cite{Johansson-McDonald-Tarrach}.

\subsection{Comparison of overconvergent and completed cohomology}

To compare completed and overconvergent cohomology, we need an integral version of the latter. Shrinking $\Omega$ (see \ref{convention}) we may assume that $\Omega\subset \Omega^{\circ}\subset \Omega'\subset \sW$, where $\Omega'$ is an open affinoid neighbourhood of the trivial weight and $\Omega^\circ$ is a wide open disc, i.e., $\Omega^\circ$ is the rigid generic fibre of $\mathrm{Spf}(\cO_L[\![X_1/p^{\rho}, \dots, X_{n}/p^{\rho}]\!])$ for some $\rho\in \Q_{>0}$ (see \S\ref{sec:infinitesimal weights}). Moreover, we denote by $\kappa_{\Omega^\circ}$ its universal weight. Recall that if $R$ is an $L$-affinoid algebra, then $R^+ \subset R$ denotes the unit ball with respect to the supremum semi-norm on $R$.

As in \cite[\S2.2]{Han17} and \S\ref{ss: Koszul resolutions and overconvergent cohomology}, for sufficiently large $r$ we consider the following $\cO(\Omega^{\circ})^+$-module:
\[
\bfA_{\Omega^{\circ}}^{\circ, r}= \{f: \mathrm{Iw}\rightarrow \cO(\Omega^{\circ})^+ : f\text{ is $r$-analytic}, \ f(bg)= \kappa_{\Omega^\circ}(b)f(g) \ \forall \ b\in \overline{B}(\Q_p)\cap \mathrm{Iw}, g\in \rm Iw\}.
\]
As before this module is endowed with left actions of $\mathrm{Iw}$ and $T^-$. We denote by $\bfD_{\Omega^{\circ}}^{\circ, r}$ the continuous $\cO(\Omega^{\circ})^+$-linear dual of $\bfA_{\Omega^{\circ}}^{\circ, r}$, which inherits left actions by $\mathrm{Iw}$ and $T^+$.  Moreover, we have a Hecke-equivariant isomorphism 
\begin{equation}\label{e: + versus usual oc}
\mathrm{H}_c^{\bullet}(\YK, \bfD_{\Omega^{\circ}}^{\circ, r})^{\rm ord, \varepsilon}\otimes_{\cO(\Omega^{\circ})^+}\cO(\Omega)\cong \mathrm{H}_c^{\bullet}(\YK, \bfD_{\Omega}^{r})^{\rm ord, \varepsilon}\cong \mathrm{H}_c^{\bullet}(\YK, \crD_{\Omega})^{\rm ord, \varepsilon}.
\end{equation}

Recall from \S\ref{ss: Infinitesimal families} the map $\phi_v\colon \bT(K)\rightarrow L[\epsilon]$, the specialisation along $v$ of a map $\phi_{\Sigma} \colon \bT(K) \to \cO(\Sigma)$.
By assumption, there exists an irreducible component $\Sigma' \subset \Sigma$ such that $v \in \opn{T}_1(\Sigma')$. Without loss of generality, we replace $\Sigma$ with $\Sigma'$, and thus assume $\cO(\Sigma)$ is reduced. (This will allow us to ignore the nilpotent ideal $J$.)

Let $\fn\subset \bT(K)\otimes_{\Z_p}\cO(\Omega^{\circ})^+$ be the maximal ideal obtained by pulling back $\m$ under the reduction modulo $(X_1/p^{\rho},.., X_{n}/p^{\rho})$. By construction of the eigenvariety, and Theorem \ref{thm:main etale}, the map $\phi_{\Sigma}:\bT(K)\rightarrow \cO(\Sigma)$ factors though $\bT(K)\rightarrow \bT(K)(\mathrm{H}^{\bullet}_c(\YK, \crD_{\Omega})^{\mathrm{ord}, \varepsilon})$. Moreover, equation (\ref{e: + versus usual oc}) shows $\phi_\Sigma$ induces a map 
\[
    \bT(K)(\mathrm{H}^{\bullet}_c(\YK, \bfD_{\Omega^{\circ}}^{\circ, r})^{\mathrm{ord}, \varepsilon})\rightarrow \cO(\Sigma)
\]
which factors through $\bT(K)(\mathrm{H}^{\bullet}_c(\YK, \bfD_{\Omega^{\circ}}^{\circ, r})^{\mathrm{ord}, \varepsilon}_{\fn})$. Summarising this discussion, we have a commutative diagram
\[
\xymatrix{
\bT(K) \ar[r]& \bT(K)(\mathrm{H}^{\bullet}_c(\YK, \bfD_{\Omega^{\circ}}^{\circ, r})^{\mathrm{ord}, \varepsilon})\ar[r]\ar[d]& \bT(K)(\mathrm{H}^{\bullet}_c(\YK, \crD_{\Omega})^{\mathrm{ord}, \varepsilon}) \ar[d] \\
& \bT(K)(\mathrm{H}^{\bullet}_c(\YK, \bfD_{\Omega^{\circ}}^{\circ, r})^{\mathrm{ord}, \varepsilon}_{\fn}) \ar[r] & \cO(\Sigma),
}
\]
where any path from $\bT(K)$ to $\cO(\Sigma)$ is the map $\phi_\Sigma$.

\begin{proposition} \label{p: factorization}
Suppose the natural map $\bT(K)\rightarrow \bT(K)(\mathrm{H}^{\bullet}_c(\YK, \bfD_{\Omega^{\circ}}^{\circ, r})^{\mathrm{ord}, \varepsilon}_{\fn})$ factors through $\cT_{\m}$. Then Theorem \ref{t: TriangulationInFamilies} holds.
\end{proposition}

\begin{proof}
The hypothesis allows us to complete the commutative diagram above to the diagram
\[
\xymatrix{
\bT(K) \ar[r]\ar[d]& \bT(K)(\mathrm{H}^{\bullet}_c(\YK, \bfD_{\Omega^{\circ}}^{\circ, r})^{\mathrm{ord}, \varepsilon})\ar[r]\ar[d]& \bT(K)(\mathrm{H}^{\bullet}_c(\YK, \crD_{\Omega})^{\mathrm{ord}, \varepsilon}) \ar[d] \\
\cT_{\m}\ar[r]& \bT(K)(\mathrm{H}^{\bullet}_c(\YK, \bfD_{\Omega^{\circ}}^{\circ, r})^{\mathrm{ord}, \varepsilon}_{\fn}) \ar[r] & \cO(\Sigma).
}
\]
In particular, $\phi_\Sigma$ factors through $\bT(K)\rightarrow \cT_{\m}$. Since $\cO(\Sigma)$ is reduced, and $J$ is nilpotent, the resulting map $\cT_{\m} \to \cO(\Sigma)$ factors through $\cT_{\m}/J$. Specialising along $v \in \opn{T}_1(\Sigma)$, we obtain a factorisation $\phi_v \colon \bT(K) \to \cT_{\m}/J \xrightarrow{\ f_v \ } L[\epsilon]$. The parameters $f_v \circ \tilde\chi_i$ associated with $f_v$ match exactly the $\chi_{v,i}\omega^{1-i}$, and hence we conclude from Theorem \ref{thm:BGW LGC}.
\end{proof}

To complete the proof of Theorem \ref{t: TriangulationInFamilies}, it remains to prove that $\bT(K)\rightarrow \bT(K)(\mathrm{H}^{\bullet}_c(\YK, \bfD_{\Omega^{\circ}}^{\circ, r})^{\mathrm{ord}, \varepsilon}_{\fn})$ factors through $\cT_{\m}$. To do this, we exploit the profinite structure of the module $\bfD_{\Omega^{\circ}}^{\circ, r}$. There is an exhaustive  $(\mathrm{Iw}, T^+)$-stable descending filtration $\mathrm{Fil}^m \bfD_{\Omega^{\circ}}^{\circ, r}$ for $m\in \N$ (see \cite[\S2]{Hanpreprint} and \cite[\S5.1]{JoNewExt}) such that:
\begin{itemize}
\item $\mathrm{Fil}^0\bfD_{\Omega^{\circ}}^{\circ, r}= \bfD_{\Omega^{\circ}}^{\circ, r}$,
\item $\bfD_{\Omega^{\circ}}^{\circ, r}= \varprojlim_{m} \bfD_{\Omega^{\circ}}^{\circ, r}/\mathrm{Fil}^m \bfD_{\Omega^{\circ}}^{\circ, r},$
\item for each $m$ the module $\bfD_{\Omega^{\circ}}^{\circ, r}/\mathrm{Fil}^m \bfD_{\Omega^{\circ}}^{\circ, r}$ is a finite set, and the action of $\mathrm{Iw}_{m+ r}= \mathrm{ker}(\mathrm{Iw}\rightarrow \mathrm{GL}_n(\Z/p^{m+r}\Z))$ on this is trivial.
\end{itemize}
The following is proved similarly to \cite[Lemma 9.8.1]{GL3-ExceptionalZeros}, and completes the proof of Theorem \ref{t: TriangulationInFamilies}.

\begin{lemma}\label{l:factorize}
\begin{enumerate}
\item For $m\geq 1$ and $c\geq b\geq m+r$ there exists a $\bT(K)$-equivariant quasi-isomorphism between $R\Gamma(\YK, \bfD_{\Omega^{\circ}}^{\circ, r}/\mathrm{Fil}^m \bfD_{\Omega^{\circ}}^{\circ, r})^{\mathrm{ord}, \varepsilon}$ and
\[
R\Gamma(T(\Z_p)/T(b), R\Gamma(Y(K^p\mathrm{Iw}(b,c)), \cO_L/\varpi^m)^{\mathrm{ord}, \varepsilon}\otimes_{\cO_L/\varpi^m}(\cO(\Omega^{\circ})^+/\fa^m)(\kappa_{\Omega^\circ}^{-1})),
\]
where $T(b)= \{t\in T(\Z_p): t\equiv 1 \ \text{mod} \ p^b\}$ and $\fa$ is the maximal ideal of $\cO(\Omega^{\circ})^+$.
\item The map $\bT(K)\rightarrow \bT(K)(R\Gamma(\YK, \bfD_{\Omega^{\circ}}^{\circ, r}/\mathrm{Fil}^m)^{\mathrm{ord}, \varepsilon})$ factors through $\bT(K)(b, c; m)$. Moreover, this is compatible with the change of $b,c$ and $m$. 
\item The map $\bT(K)\rightarrow \bT(K)(\mathrm{H}^{\bullet}_c(\YK, \bfD_{\Omega^{\circ}}^{\circ, r})^{\mathrm{ord}, \varepsilon}_{\fn})$ factors through $\cT_{\fm}$.
\end{enumerate}
\end{lemma}
\begin{proof}
Part (2) is a direct consequence of the Hecke-equivariance in (1). To prove (3), it is enough to observe that from (2) the following map factors through $\cT_{\m}$:
\begin{align*}
\bT(K) &\to \varprojlim_m \bT(K)(R\Gamma(\YK, \bfD_{\Omega^{\circ}}^{\circ, r}/\mathrm{Fil}^m)^{\mathrm{ord}, \varepsilon}_{\fn}) = \varprojlim_m \bT(K)(R\Gamma_c(\YK, \bfD_{\Omega^{\circ}}^{\circ, r}/\mathrm{Fil}^m)^{\mathrm{ord}, \varepsilon}_{\fn})\\
&= \bT(K)(R\Gamma_c(\YK, \bfD_{\Omega^{\circ}}^{\circ, r})^{\mathrm{ord}, \varepsilon}_{\fn}) \to \bT(K)(\mathrm{H}^{\bullet}_c(\YK, \bfD_{\Omega^{\circ}}^{\circ, r})^{\mathrm{ord}, \varepsilon}_{\fn}).
\end{align*}
Here we use that $\m$ is non-Eisenstein in the first equality, and that $R\Gamma_c(\YK, \bfD_{\Omega^{\circ}}^{\circ, r})^{\mathrm{ord}, \varepsilon}_{\fn}$ is a perfect complex in the derived category of $\cO(\Omega^{\circ})^+$-modules in the second one.

To prove (1), firstly because the ordinary part of the cohomology at level $\mathrm{Iw}(b,c)$ is independent of $c$, we deduce that
\begin{equation}
\label{e: 1} R\Gamma(\YK, \bfD_{\Omega^{\circ}}^{\circ, r}/\mathrm{Fil}^m \bfD_{\Omega^{\circ}}^{\circ, r})^{\mathrm{ord}, \varepsilon} \cong R\Gamma(Y(K(0,c)), \bfD_{\Omega^{\circ}}^{\circ, r}/\mathrm{Fil}^m \bfD_{\Omega^{\circ}}^{\circ, r})^{\mathrm{ord}, \varepsilon}.
\end{equation}

Now let $F_{\Omega^{\circ}}: B(\Z_p)\rightarrow \cO(\Omega^{\circ})^+$ be the natural map obtained from the universal character $\kappa_{\Omega^{\circ}}$. For sufficiently large $r$, we consider $F_{\Omega^{\circ}}\in \bfA_{\Omega^\circ}^{\circ, r}$ in the natural way, restricting to the constant function $1$ on $N(\Z_p)$. Evaluating at $F_{\Omega^{\circ}}$ we obtain a $\cO(\Omega^{\circ})^+$-linear map $\Lambda: \bfD_{\Omega^{\circ}}^{\circ, r}\rightarrow \cO(\Omega^{\circ})^+(\kappa_{\Omega^\circ}^{-1})$ which is equivariant for the actions of $B(\Z_p)$ and $T^+$. For each $m\geq 1$, from $\Lambda$ we obtain a map 
\[
    \Lambda_m: \bfD_{\Omega^{\circ}}^{\circ, r}/\mathrm{Fil}^m \bfD_{\Omega^{\circ}}^{\circ, r}\rightarrow (\cO(\Omega^{\circ})^+/ \fa^m)(\kappa_{\Omega^\circ}^{-1})
\]
which is equivariant for the action of $(\mathrm{Iw}(0,c), T^+)$. Thus we obtain a map
\begin{equation}\label{e: 2}
R\Gamma(Y(K(0,c)), \bfD_{\Omega^{\circ}}^{\circ, r}/\mathrm{Fil}^m \bfD_{\Omega^{\circ}}^{\circ, r})^{\mathrm{ord}, \varepsilon}\rightarrow R\Gamma(Y(K(0,c)), (\cO(\Omega^{\circ})^+/ \fa^m)(\kappa_{\Omega^\circ}^{-1}))^{\mathrm{ord}, \varepsilon}.
\end{equation}

To study this morphism, we observe that, analogously to \cite[Claim 9.8.3]{GL3-ExceptionalZeros}, we have $t^{r+1}\mathrm{ker}(\Lambda_m)\subset p\mathrm{ker}(\Lambda_m)$ where $t= t_1 \cdots t_{n-1}$. From this we deduce that a big enough power of $U_t$ acts by zero on the cohomology with coefficients in $\mathrm{ker}(\Lambda_m)$. Thus, from the short exact sequence obtained from $\Lambda_m$, we deduce that the morphism (\ref{e: 2}) is a quasi-isomorphism. Now, using \cite[Lemma 5.2.9]{10author} we obtain
\begin{equation}\label{e: 3}
R\Gamma(Y(K(0,c)), (\cO(\Omega^{\circ})^+/ \fa^m)(\kappa_{\Omega^\circ}^{-1}))^{\mathrm{ord}, \varepsilon}\cong R\Gamma(T(\Z_p), R\Gamma (N(\Z_p), (\cO(\Omega^{\circ})^+/ \fa^m)(\kappa_{\Omega^\circ}^{-1}))^{\mathrm{ord}, \varepsilon}).
\end{equation}

Moreover, from the equality $R\Gamma(T(\Z_p),-)= R\Gamma(T(\Z_p)/T(b), R\Gamma(T(b),-))$, the fact that $\mathrm{Iw}(b,c)$ acts trivially on $(\cO(\Omega^{\circ})^+/ \fa^m)(\kappa_{\Omega^\circ}^{-1})$, and \cite[Lemma 5.2.9]{10author} again, we deduce that 
\begin{multline*}
R\Gamma(T(\Z_p), R\Gamma (N(\Z_p), (\cO(\Omega^{\circ})^+/ \fa^m)(\kappa_{\Omega^\circ}^{-1}))^{\mathrm{ord}, \varepsilon})\cong\\
R\Gamma(T(\Z_p)/T(b), R\Gamma(Y(K^p\mathrm{Iw}(b,c)), \cO_L/\varpi^m)^{\mathrm{ord}, \varepsilon}\otimes_{\cO_L/\varpi^m}(\cO(\Omega^{\circ})^+/\fa^m)(\kappa_{\Omega^\circ}^{-1})).
\end{multline*}
Finally, we complete the proof of (1) by combining the quasi-isomorphisms \eqref{e: 1}, \eqref{e: 2} and \eqref{e: 3}. \end{proof}

\begin{remark} 
The $n=3$ case of Theorem \ref{t: TriangulationInFamilies} was proved in \cite[Thm. 8.3.5]{GL3-ExceptionalZeros}. In proving both of these results, the crucial input is Theorem \ref{thm:BGW LGC} above. In turn, the key step in proving this was a local-global compatibility result at $\ell= p$, established in  \cite[\S9]{GL3-ExceptionalZeros} for Galois representations attached to $p$-ordinary torsion classes for $\mathrm{GL}_n$. 
\end{remark}

 The following is our main application of the triangulation above.

\begin{theorem} \label{t: galois Benois--Colmez--Greenberg--Stevens formula} 
Let $\pi$ and $\Sigma$ be as in Set-up \ref{setup}. Assume the non-abelian Leopoldt conjecture, and that Hypothesis \ref{h: abs irred and decomposed generic} holds.
Then for any $c\in \{1,\dots, n-1\}$, and any $v = (v_1,\dots,v_n) \in \opn{T}_1(\Sigma)$, we have 
\[
\partial \chi_{v}[c] \in \bL^{\mathrm{FM}}_c(\pi).
\]
In particular, if $v_c \neq v_{c+1}$, we have $\bL^{\mathrm{FM}}_c(\pi) = \langle \partial\chi_v[c]\rangle.$
\end{theorem}
\begin{proof} 
The first statement is the (Galois) Benois--Colmez--Greenberg--Stevens formula (see \cite{Benois11} and \cite{DingLinvariants}). In fact, we have $\chi_{c, v}\omega^{1-c} (\chi_{c+1, v}\omega^{1-(c+1)})^{-1}= \omega^{1-c}(\omega^{1-(c+1)})^{-1}(1+ \partial\chi_{v}[c]\epsilon)$. Using the triangulation of Theorem \ref{t: TriangulationInFamilies}, we deduce the result from \cite[Thm.\ 3.4]{DingLinvariants}. The second statement is immediate (as in the proof of Theorem \ref{p: automorphic BCGS}).
\end{proof}

Combining Theorems \ref{t: galois Benois--Colmez--Greenberg--Stevens formula} and \ref{p: automorphic BCGS} we have proved Theorem \ref{thm:intro equality}.

\section{Symmetric power functoriality}\label{sec:symmetric powers}

In this final section, we discuss an application of our results to the symmetric power functoriality of automorphic $\mathcal{L}$-invariants. For this, let $f \in S_2^{\opn{new}}(\Gamma_1(Np), \chi_f)$ be a weight two newform of level $\Gamma_1(Np)$ and nebentypus $\chi_f$, where $p\nmid N$. We assume that $a_p(f) = 1$ and the conductor of $\chi_f$ is prime to $p$. 

Let $\pi$ be the (unitary) automorphic representation of $\GL_2(\A)$ corresponding to $f$; our assumptions ensure that $\pi_p$ is the Steinberg representation of $\GL_2(\Qp)$ and $\pi$ is $p$-ordinary. Moreover, $\pi$ is non-CM; CM automorphic representations of $\GL_2(\A)$, whose local Weil--Deligne representations have no monodromy, are locally unramified or supercuspidal (i.e., not Steinberg) at all finite primes.

\subsection{Symmetric power functoriality for Hida families}

Let $\opn{Sym}^{n-1}\pi$ denote the symmetric power lifting of $\pi$, as constructed by Newton--Thorne \cite{NewtonThorneSymII}. This is a RACAR of $\opn{GL}_n(\mbb{A})$ of trivial cohomological weight which is unramified outside $Np$, with the property that the Satake parameters away from $Np$ match with $\opn{Sym}^{n-1}$ of the Satake parameters of $\pi$. Since $\opn{Sym}^{n-1}\pi$ is essentially self-dual, and by construction its associated semisimple Galois representation is $\opn{Sym}^{n-1}\rho_{\pi}$, by known local-global compatibility results (see Remark \ref{rem:essentially self-dual}) we see that $\opn{Sym}^{n-1}\pi$ must be Steinberg at $p$.

\medskip

We now explain how to vary functoriality in families. Fix admissible signs $\varepsilon \in \mathcal{E}_{\pi}$ and $\varepsilon' \in \mathcal{E}_{\opn{Sym}^{n-1}\pi}$; all constructions in this section will be with respect to these signs. Let $\mathscr{C} \to \mathscr{W}_{\opn{GL}_2}$ denote the unique Hida family passing through $f$ in the $\opn{GL}_2$-eigenvariety (built from considering the Hecke operators away from $N$). Here, our conventions are that the point corresponding to $f$ on $\mathscr{C}$ has weight equal to the trivial character of $T_{\operatorname{GL}_2}(\mathbb{Z}_p)$. This Hida family is finite over the component of $\mathscr{W}_{\opn{GL}_2}$ containing the trivial character. We say a point $x \in \mathscr{C}$ is classical if the weight of $x$ is locally algebraic, and the algebraic part is dominant. By Hida's classicality theory, this implies that $x$ is associated with a classical finite-slope ordinary cusp form $f_x$ of level $\Gamma_1(Np^r)$ for some $r \geq 0$.

Consider the morphism of weight spaces $\iota \colon \mathscr{W}_{\opn{GL}_2} \to \mathscr{W}_{\opn{GL}_n}$ given by sending a weight $(\kappa_1, \kappa_2)$ to $(\kappa_1^{n-1}, \kappa_1^{n-2}\kappa_2, \dots, \kappa_1\kappa_2^{n-2}, \kappa_2^{n-1})$.\footnote{This is locally (on the source) an immersion.} Let $K = K^p \opn{Iw} \subset \opn{GL}_n(\mbb{A}_f)$ denote the compact open subgroup with $K^p$ equal to the Whittaker new level of conductor $N$ (defined in \eqref{eq:whittaker new}). Let $S$ denote the set of primes dividing $Np$, and let $\mbf{T}_{\opn{GL}_n}(Np) = \mbf{T}_{\opn{GL}_n}^S \otimes_{\mbb{Q}_p} \mbb{Q}_p[T_{\opn{GL}_n}^+]$, where $\mbf{T}_{\opn{GL}_n}^S$ denotes the spherical Hecke algebra away from $S$. Similarly, let $\mbf{T}_{\opn{GL}_2}(Np) = \mbf{T}_{\opn{GL}_2}^{S} \otimes_{\mbb{Q}_p} \mbb{Q}_p[T_{\opn{GL}_2}^+]$. 

Consider the morphism of tori $T_{\opn{GL}_n} \to T_{\opn{GL}_2}$ given by 
\[
\opn{diag}(x_1, \dots, x_n) \mapsto \opn{diag}(x_1^{n-1} x_2^{n-2} \cdots x_{n-1}, x_2 x_3^2 \cdots x_n^{n-1}) .
\]
This induces a morphism $X_*(T_{\opn{GL}_n}) \to X_*(T_{\opn{GL}_2})$, and hence a morphism $\mbf{T}^S_{\opn{GL}_n} \to \mbf{T}^S_{\opn{GL}_2}$ via the Satake isomorphisms; this is defined over $\mbb{Q}_p$, because the pullback of the half-sum of positive roots for $\opn{GL}_2$ along this map is the half-sum of positive roots for $\opn{GL}_n$. This also induces a natural map $\mbb{Q}_p[T^+_{\opn{GL}_n}] \to \mbb{Q}_p[T^+_{\opn{GL}_2}]$. Taking the tensor product, we obtain a natural morphism 
\[
\jmath \colon \mbf{T}_{\opn{GL}_n}(Np) \rightarrow \mbf{T}_{\opn{GL}_2}(Np) .
\]
If $\phi \colon \mbf{T}_{\opn{GL}_2}(Np) \to \mbb{C}_p$ is a Hecke eigensystem attached to a classical RACAR of weight $\lambda$, then $\phi \circ \jmath$ is a Hecke eigensystem attached to its $\opn{Sym}^{n-1}$-lift, which has weight $\iota(\lambda)$.

\begin{lemma}\label{lem:functorial}
    Let $\mathscr{C}'$ denote the unique Hida family for $\opn{GL}_n$ passing through the unique ordinary $p$-refinement in $\opn{Sym}^{n-1}\pi$. Then the morphism $\jmath$ induces a map $\mathscr{C} \to \mathscr{C}'$ fitting into the diagram
    \[
\begin{tikzcd}
\mathscr{C} \arrow[d] \arrow[r]    & \mathscr{C}' \arrow[d]   \\
\mathscr{W}_{\opn{GL}_2} \arrow[r] & \mathscr{W}_{\opn{GL}_n},
\end{tikzcd}
    \]
    interpolating the $\opn{Sym}^{n-1}$-lift on classical points.
\end{lemma}
\begin{proof}
It is standard that the classical points are Zariski dense in the $\GL_2$-Hida family $\mathscr{C}$. Let $x$ be such a classical point of weight $\lambda_x = (\kappa_1, \kappa_2)$, corresponding to an eigensystem $\phi_x$ in a $\GL_2$-RACAR $\pi_x$. The $\opn{Sym}^{n-1}$-lift $\phi_x \circ \jmath$ is $p$-ordinary, hence non-critical. Thus by \cite[Prop.\ 4.5.2]{Han17} it contributes to $\opn{H}^{\bullet}_c(K, \mathscr{D}_{\Omega})^{\opn{fs}, \varepsilon'}$ for an open affinoid neighbourhood $\Omega \subset \mathscr{W}_{\opn{GL}_n}$  of $\iota(\lambda_x)$. In particular, $\phi_x \circ \jmath$ yields a point $\opn{Sym}^{n-1} x \in \sC'$. Note that $\varepsilon'$ is admissible for $\pi_x$ because, if $n$ is odd, the central character of $\pi_x$ has the same parity as $\pi$ and $(\kappa_1 \kappa_2)^{(n-1)/2} \equiv 0$ modulo $2$, c.f., \cite[(3.2)]{Mah05} (here we are using the fact that $\lambda_x$ is in the same component of weight space as the trivial character, so necessarily $\kappa_i \equiv 0$ modulo $2$). 

The result now follows from Johansson--Newton's $p$-adic functoriality theorem \cite[Thm.\ 3.2.1]{JoNew}.
\end{proof}

\subsection{$\mathcal{L}$-invariants}

We now apply the automorphic BCGS formula. Let $u = (1, 0) \in \opn{T}_1(\mathscr{W}_{\opn{GL}_2})$, which maps to the tangent vector $v = (v_1, \dots, v_n) = (n-1, n-2, \dots, 1,0) \in \opn{T}_1(\mathscr{W}_{\opn{GL}_n})$. Note $v_c \neq v_{c+1}$ for all $1 \leq c \leq n-1$. The tangent vector $u$ lifts to a tangent vector $\tilde{u} \in \opn{T}_{f}(\mathscr{C})$, which -- via Lemma \ref{lem:functorial} -- maps to a tangent vector $\tilde{v} \in \opn{T}_{\opn{Sym}^{n-1}f}(\mathscr{C}')$ lifting $v$ via the map $\jmath$. In particular, both $\pi$ and $\opn{Sym}^{n-1}\pi$ admit non-$c$-parabolic deformations for all possible $c$.

Since the weight map is locally an isomorphism around $f$, the family in $\sC$ varying in the direction of $\tilde{u}$ is unique. Its corresponding parameter at $p$, as defined in \eqref{eq:chiv}, is described by a character $\chi_u = (\chi_{u, 1}, \chi_{u, 2}) \colon T_{\opn{GL}_2}(\mbb{Q}_p) \to L[\epsilon]^{\times}$. The parameter at $p$ corresponding to the family through $\opn{Sym}^{n-1}f$ in the direction $\tilde{v}$, which by Theorem \ref{thm:main etale} is also unique, is then given by the character 
\[
\chi'_{v} = (\chi'_{v, 1}, \dots, \chi'_{v, n}) = (\chi_{u, 1}^{n-1}, \chi_{u, 1}^{n-2} \chi_{u, 2}, \dots, \chi_{u, 2}^{n-1}) .
\]

\begin{proposition}
    Suppose the non-abelian Leopoldt conjecture holds for $p$-adic families through $\opn{Sym}^{n-1}\pi$. Then for all admissible signs $\varepsilon \in \mathcal{E}_{\pi}$ and $\varepsilon' \in \mathcal{E}_{\opn{Sym}^{n-1}\pi}$, and for all $1 \leq c \leq n-1$ and $0 \leq r \leq \ell(\opn{GL}_n)$, we have
    \[
    \mbb{L}^{r, \varepsilon'}_c(\opn{Sym}^{n-1}\pi) = \mbb{L}^{0, \varepsilon}_1(\pi) .
    \]
\end{proposition}
\begin{proof}
    For $\pi$, we know (by Proposition \ref{p: automorphic BCGS}, or indeed from \cite{GehrmannRosso}) that $\partial \chi_u [1] = \partial \chi_{u, 1} - \partial \chi_{u, 2} \neq 0$ is a basis of $\mbb{L}^{0, \varepsilon}_1(\pi)$. From the definitions, one can easily compute that
    \[
    \partial \chi'_v [c] = \partial \chi_{u, 1} - \partial \chi_{u, 2},
    \]
    and we know by Theorem \ref{p: automorphic BCGS} that this is also a basis of $\mbb{L}^{r, \varepsilon'}_c(\opn{Sym}^{n-1} \pi)$. The result follows.
\end{proof}

\begin{remark}
    Since the Galois and automorphic $\mathcal{L}$-invariants are known to agree for the $\GL_2$ representation $\pi$, this result gives an alternative way to prove that the Galois and automorphic $\mathcal{L}$-invariants match up for $\opn{Sym}^{n-1}\pi$, without appealing to local-global compatibility results in families at $\ell = p$.
\end{remark}

\newcommand{\etalchar}[1]{$^{#1}$}
\renewcommand{\MR}[1]{}
\providecommand{\bysame}{\leavevmode\hbox to3em{\hrulefill}\thinspace}
\providecommand{\MR}{\relax\ifhmode\unskip\space\fi MR }
\providecommand{\MRhref}[2]{%
  \href{http://www.ams.org/mathscinet-getitem?mr=#1}{#2}
}
\providecommand{\href}[2]{#2}

\Addresses

\end{document}